\documentclass{article}
\usepackage{amssymb}
\usepackage{amsthm}
\usepackage{graphicx}
\usepackage{enumerate}
\usepackage{mathtools}
\usepackage{fullpage}
\usepackage{xcolor}
\usepackage{pgfplots}

\pgfplotsset{compat=newest}

\usepackage{tikz}
\usetikzlibrary{patterns}
\usetikzlibrary{decorations.pathmorphing}
\usetikzlibrary{decorations.pathreplacing,decorations.markings}
\usetikzlibrary{shapes}

\makeatletter
\def\grd@save@target#1{%
	\def\grd@target{#1}}
\def\grd@save@start#1{%
	\def\grd@start{#1}}
\tikzset{
	grid with coordinates/.style={
		to path={%
			\pgfextra{%
				\edef\grd@@target{(\tikztotarget)}%
				\tikz@scan@one@point\grd@save@target\grd@@target\relax
				\edef\grd@@start{(\tikztostart)}%
				\tikz@scan@one@point\grd@save@start\grd@@start\relax
				\draw[minor help lines] (\tikztostart) grid (\tikztotarget);
				\draw[major help lines] (\tikztostart) grid (\tikztotarget);
				\grd@start
				\pgfmathsetmacro{\grd@xa}{\the\pgf@x/1cm}
				\pgfmathsetmacro{\grd@ya}{\the\pgf@y/1cm}
				\grd@target
				\pgfmathsetmacro{\grd@xb}{\the\pgf@x/1cm}
				\pgfmathsetmacro{\grd@yb}{\the\pgf@y/1cm}
				\pgfmathsetmacro{\grd@xc}{\grd@xa + \pgfkeysvalueof{/tikz/grid with coordinates/major step}}
				\pgfmathsetmacro{\grd@yc}{\grd@ya + \pgfkeysvalueof{/tikz/grid with coordinates/major step}}
				\foreach \x in {\grd@xa,\grd@xc,...,\grd@xb}
				\node[anchor=north] at (\x,\grd@ya) {\pgfmathprintnumber{\x}};
				\foreach \y in {\grd@ya,\grd@yc,...,\grd@yb}
				\node[anchor=east] at (\grd@xa,\y) {\pgfmathprintnumber{\y}};
			}
		}
	},
	minor help lines/.style={
		help lines,
		step=\pgfkeysvalueof{/tikz/grid with coordinates/minor step}
	},
	major help lines/.style={
		help lines,
		line width=\pgfkeysvalueof{/tikz/grid with coordinates/major line width},
		step=\pgfkeysvalueof{/tikz/grid with coordinates/major step}
	},
	grid with coordinates/.cd,
	minor step/.initial=.2,
	major step/.initial=1,
	major line width/.initial=0.25mm,
}

\pgfdeclarepatternformonly{north east lines wide}%
{\pgfqpoint{-1pt}{-1pt}}%
{\pgfqpoint{10pt}{10pt}}%
{\pgfqpoint{5pt}{5pt}}%
{
	\pgfsetlinewidth{.8pt}
	\pgfpathmoveto{\pgfqpoint{0pt}{0pt}}
	\pgfpathlineto{\pgfqpoint{9.1pt}{9.1pt}}
	\pgfusepath{stroke}
}

\tikzset{
	on each segment/.style={
		decorate,
		decoration={
			show path construction,
			moveto code={},
			lineto code={
				\path [#1]
				(\tikzinputsegmentfirst) -- (\tikzinputsegmentlast);
			},
			curveto code={
				\path [#1] (\tikzinputsegmentfirst)
				.. controls
				(\tikzinputsegmentsupporta) and (\tikzinputsegmentsupportb)
				..
				(\tikzinputsegmentlast);
			},
			closepath code={
				\path [#1]
				(\tikzinputsegmentfirst) -- (\tikzinputsegmentlast);
			},
		},
	},
	mid arrow/.style={postaction={decorate,decoration={
				markings,
				mark=at position .5 with {\arrow[#1]{stealth}}
	}}},
	rmid arrow/.style={postaction={decorate,decoration={
				markings,
				mark=at position .5 with {\arrowreversed[#1]{stealth}}
	}}},
	end arrow/.style={postaction={decorate,decoration={
				markings,
				mark=at position 1 with {\arrow[#1]{stealth}}
	}}},
	start arrow/.style={postaction={decorate,decoration={
				markings,
				mark=at position 0 with {\arrow[#1]{stealth}}
	}}},
	mid3 arrow/.style={postaction={decorate,decoration={
				markings,
				mark=at position .3 with {\arrow[#1]{stealth}}
	}}},
	rmid3 arrow/.style={postaction={decorate,decoration={
				markings,
				mark=at position .7 with {\arrowreversed[#1]{stealth}}
	}}},
	mid4 arrow/.style={postaction={decorate,decoration={
				markings,
				mark=at position .4 with {\arrow[#1]{stealth}}
	}}},
	rmid4 arrow/.style={postaction={decorate,decoration={
				markings,
				mark=at position .4 with {\arrowreversed[#1]{stealth}}
	}}},
}

\makeatother
\tikzset{every state/.style={minimum size=0pt}}

\tikzset{
	mark position/.style args={#1(#2)}{
		postaction={
			decorate,
			decoration={
				markings,
				mark=at position #1 with \coordinate (#2);
			}
		}
	}
}

\usetikzlibrary{3d}
\usepackage{xcolor}
\usepackage{xifthen}

\usetikzlibrary{decorations}

\pgfdeclaredecoration{ignore}{final}
{
	\state{final}{}
}

\pgfdeclaremetadecoration{middle}{initial}{
	\state{initial}[
	width={0pt},
	next state=middle
	]
	{\decoration{moveto}}
	
	\state{middle}[
	width={\pgfdecorationsegmentlength*\pgfmetadecoratedpathlength},
	next state=final
	]
	{\decoration{curveto}}
	
	\state{final}
	{\decoration{ignore}}
}

\tikzset{middle segment/.style={decoration={middle},decorate, segment length=#1}}

\newtheorem{thm}{Theorem}[section]
\newtheorem{lm}[thm]{Lemma}

\newtheorem{notation}[thm]{Notation}
\newtheorem{prop}[thm]{Proposition}
\newtheorem{conj}[thm]{Conjecture}
\newtheorem{cor}[thm]{Corollary}

\newtheorem{rmk}[thm]{Remark}
\numberwithin{equation}{section}

\newcommand{\bigO}{\mathcal{O}}

\newcommand{\CC}{\mathcal{C}}
\newcommand{\Ch}{\mathcal{C}}
\newcommand{\complexC}{\mathbb{C}}

\newcommand{\ddbar}[2]{\frac{{\mathrm d}#1}{2\pi {\mathrm i}#2}}
\newcommand{\dg}{\mathrm{dist}}
\newcommand{\dtt}{\text{dist}}

\newcommand{\exe}{\mathrm{err}}
\newcommand{\gd}{{\mathcal{G}}}
\newcommand{\ii}{\mathrm{i}}
\newcommand{\inn}{\mathrm{in}}
\newcommand{\intZ}{\mathbb{Z}}
\newcommand{\limp}{\mathrm{p}}
\newcommand{\LL}{\mathrm{L}}
\newcommand{\out}{\mathrm{out}}
\newcommand{\prob}{\mathbb{P}}
\newcommand{\realR}{\mathbb{R}}
\newcommand{\rg}{\mathrm{g}}
\newcommand{\RR}{\mathrm{R}}
\newcommand{\rr}{\mathrm{r_{max}}}
\newcommand{\rz}{\mathrm{z}}

\begin{document}

\title{One-point distribution of the geodesic in directed last passage percolation}

\author{Zhipeng Liu}
\maketitle

\begin{abstract}
	We consider the geodesic of the directed last passage percolation with iid exponential weights. We find the explicit one-point distribution of the geodesic location joint with the last passage times, and its limit as the parameters go to infinity under the KPZ scaling. 
\end{abstract}

\section{Introduction}

In recent twenty years, there has been a huge progress towards to understanding a universal class of random growth models, the so-called Kardar-Parisi-Zhang (KPZ) universality class \cite{Baik-Deift-Johansson99,Johansson00,Johansson03,Borodin-Ferrari-Prahofer-Sasamoto07,Tracy-Widom08,Tracy-Widom09,Borodin-Corwin13,Matetski-Quastel-Remenik17,Dauvergne-Ortmann-Virag18,Johansson-Rahman19,Liu19}. Very recently, studies about the geodesics of these models started to appear \cite{Basu-Sarkar-Sly17,Hammond20,Hammond-Sarkar20,Basu-Hoffman-Sly18,Basu-Ganguly-Hammon19,Bates-Ganguly-Hammond19b,Busani-Ferrari20,Dauvergne-Sarkar-Virag20,Corwin-Hammond-Hegde-Matetski21,Dauvergne-Virag21}. However, the explicit distributions of the geodesic are still not well understood. As far as we know, the only known related results are the distribution of the geodesic endpoint location \cite{MorenoFlores-Quastel-Remenik13,Schehr12,Baik-Liechty-Schehr12}. 

This is the first paper of an ongoing project to investigate the limiting distributions of the geodesics in one representative model, the directed last passage percolation with exponential weights, using the methods in integrable probability. We obtain the finite time one-point distribution of the geodesic location joint with the last passage times, see Theorem~\ref{thm:finite_time_geodesic1}. We are also able to find the large time limit of this distribution function, see Theorem~\ref{thm:limiting_geodesic}. We remark that our results are for the point-to-point geodesic. In the follow-up papers, we will consider the point-to-point and point-to-line geodesics using a different approach, and the multi-point distributions of the point-to-point geodesic.

The limiting distributions obtained in this paper are expected to be universal for all models in the KPZ universality class. See \cite{Dauvergne-Virag21} for more discussions related to the geodesics.

\bigskip
Below we introduce the main results of the paper. We start from a short description of the model.

The directed last passage percolation is defined on the lattice set $\intZ^2$. We assign to each integer site $\mathbf{p}\in\intZ^2$ an i.i.d. exponential random variable $w(\mathbf{p})$ with mean $1$. Assume that $\mathbf{p}$ and $\mathbf{q}$ are two lattice points satisfying $\mathbf{q}-\mathbf{p}\in\intZ_{\ge 0}^2$, i.e., the point $\mathbf{q}$ lies in the upper right direction of $\mathbf{p}$. The last passage time from  $\mathbf{p}$ to $\mathbf{q}$ is
\begin{equation}
	\label{eq:def_dlpp}
	L_{\mathbf{p}}(\mathbf{q}):=\max_{\pi} \sum_{\mathbf{r}\in\pi} w(\mathbf{r}),
\end{equation}
where the maximum is over all possible up/right lattice paths from $\mathbf{p}$ to $\mathbf{q}$.

Since the random variables $w(\mathbf{r})$'s are continuous, the last passage time $L_{\mathbf{p}}(\mathbf{q})$ in~\eqref{eq:def_dlpp} is almost surely obtained at a unique up/right lattice path, which we call the geodesic from $\mathbf{p}$ to $\mathbf{q}$ and denote $\gd_{\mathbf{p}}(\mathbf{q})$.

Note that the two neighboring sites $\mathbf{r}$ and $\mathbf{r}_+$ with $\mathbf{r}_+-\mathbf{r}\in\{(0,1), (1,0)\}$ are on the geodesic $\gd_{\mathbf{p}}(\mathbf{q})$, if and only if  the sites $\mathbf{p},\mathbf{r},\mathbf{r}_+,\mathbf{q}$ satisfy $\mathbf{r}-\mathbf{p},\mathbf{q}-\mathbf{r}_+\in\intZ_{\ge 0}^2$, and the last passage times $L_{\mathbf{p}}(\mathbf{r})$ and $L_{\mathbf{r}_+}(\mathbf{q})$ satisfy
\begin{equation}
	\label{eq:dlpp_split}
     L_{\mathbf{p}}(\mathbf{r})+L_{\mathbf{r}_+}(\mathbf{q})
	=L_{\mathbf{p}}(\mathbf{q}).
\end{equation}
Throughout this paper, we always use $\mathbf{r}_+$ to denote the lattice point following $\mathbf{r}$ in the geodesic.

\subsection{Finite time joint probabilities of geodesic location and last passage times}

The first main result of this paper is about the joint probability that a fixed pair of neighboring sites $\mathbf{r}$ and $\mathbf{r}'$ are on the geodesic $\gd_{\mathbf{p}}(\mathbf{q})$, and the two last passage times $L_{\mathbf{p}}(\mathbf{r})$, $L_{\mathbf{r}'}(\mathbf{q})$ lie in some intervals.
\begin{thm}
	\label{thm:finite_time_geodesic1}
	Set $\mathbf{p}=(1,1)$, $\mathbf{q}=(M,N)$. Suppose $\mathbf{r}=(m,n)$ and $\mathbf{r'}=(m+1,n)$, with $m,n$ satisfying $1\le m\le M-1$ and $1\le n\le N$. Assume that $t_1,t_2,\epsilon_1,\epsilon_2$ are all positive real numbers. We have
	\begin{equation}
		\label{eq:main_thm01}
		\prob
		    \left(
		       \mathbf{r},\mathbf{r}'\in\gd_{\mathbf{p}}(\mathbf{q}),
		       L_{\mathbf{p}}(\mathbf{r})\in[t_1,t_1+\epsilon_1], L_{\mathbf{r'}}(\mathbf{q})\in[t_2,t_2+\epsilon_2]
		     \right)
		=
		    \int_{t_1}^{t_1+\epsilon_1}
		    \int_{t_2}^{t_2+\epsilon_2}
		      p(s_1,s_2;m,n,M,N)
		    \mathrm{d}s_2\mathrm{d}s_1,
	\end{equation}
where the function $p(s_1,s_2;m,n,M,N)$ is defined in~\eqref{eq:def_finite_time_density}. Similarly, if $\mathbf{r}=(m,n)$ and $\mathbf{r'}=(m,n+1)$, with $m,n$ satisfying $1\le m\le M$ and $1\le n\le N-1$, the formula~\eqref{eq:main_thm01} holds with $p(s_1,s_2;m,n,M,N)$ replaced by $p(s_1,s_2;n,m,N,M)$.
\end{thm}
\begin{rmk}
	\label{rmk:geodesic_probability}
	By setting $t_1=t_2=0$ and $\epsilon_1=\epsilon_2=\infty$, one can derive a  formula for the probability of $\mathbf{r},\mathbf{r}'\in\gd_{\mathbf{p}}(\mathbf{q})$ without the double integral with respect to the last passage times. See~\eqref{eq:finite_time_geodesic_distribution}. However, we are not able to directly perform the asymptotics analysis of this formula since the summand~\eqref{eq:2021_12_22_01} diverges when the parameters go to infinity under the KPZ scaling, the scaling of most interests to us. Moreover, it is not very surprising that the geodesic information is intertwisted with the last passage times. In fact, it has been proved that the geodesic $\gd_{\mathbf{p}}(\mathbf{q})$ becomes more rigid (or localized) around its expected location when the last passage time $L_{\mathbf{p}}(\mathbf{q})$ becomes very large \cite{Basu-Ganguly19,Liu21}. On the other hand, it is not concentrated around any deterministic curve when the last passage time becomes very small \cite{Basu-Ganguly-Sly19}.
\end{rmk}

The proof of Theorem~\ref{thm:finite_time_geodesic1} is given in Section~\ref{sec:proof_thm1}.

\subsection{The probability density function $p(s_1,s_2;m,n,M,N)$}

We first introduce three notations. Suppose $W=(w_1,\cdots,w_k)\in\complexC^k$ is a vector, we denote
\begin{equation}
	\label{eq:def_Delta1}
	\Delta(W):=\prod_{1\le i<j\le k}(w_j-w_i).
\end{equation}
If $W=(w_1,\cdots,w_k)\in\complexC^k$ and $W'=(w'_1,\cdots,w'_{k'})\in\complexC^{k'}$ are two vectors, we denote
\begin{equation}
	\label{eq:def_Delta2}
	\Delta(W;W'):= \prod_{i=1}^k\prod_{i'=1}^{k'}(w_i-w'_{i'}).
\end{equation}
Finally, if $f:\complexC\to\complexC$ is a function and $W=(w_1,\cdots,w_k)\in\complexC^k$ is a vector, or $W=\{w_1,\cdots,w_k\}$ with each element $w_i\in\complexC$, we write
\begin{equation}
	\label{eq:def_f_vector}
	f(W):=\prod_{i=1}^k f(w_i).
\end{equation}
Throughout this paper, we allow the empty product and define it to be $1$.

We need to introduce six contours. Suppose $\Sigma_{\LL,\out},\Sigma_\LL$ and $\Sigma_{\LL,\inn}$ are three nested contours, from outside to inside, enclosing $-1$ but not $0$. Similarly, $\Sigma_{\RR,\out},\Sigma_\RR$ and $\Sigma_{\RR,\inn}$ are three nested contours, from outside to inside, enclosing $0$ but not $-1$. We further assume that the contours enclosing $-1$ are disjoint from those enclosing $0$. In other words, the two outermost contours $\Sigma_{\LL,\out}$ and $\Sigma_{\RR,\out}$ do not intersect. All the closed contours throughout this paper are counterclockwise oriented. See Figure~\ref{fig:contours_finite_time} for an illustration of these contours.

	\begin{figure}[t]
	\centering
	\begin{tikzpicture}[scale=1]
		\draw [line width=0.4mm,lightgray] (-4.5,0)--(1.5,0) node [pos=1,right,black] {$\realR$};
		\draw [line width=0.4mm,lightgray] (0,-1.5)--(0,1.5) node [pos=1,above,black] {$\mathrm{i}\realR$};
		\fill (0,0) circle[radius=2.5pt] node [below,shift={(0pt,-3pt)}] {$0$};
		\fill (-3,0) circle[radius=2.5pt] node [below,shift={(0pt,-3pt)}] {$-1$};
		
		
		\draw[thick] (0,0) circle (20pt);
		\draw[thick] (0,0) circle (27pt);
		\draw[thick] (0,0) circle (34pt);
		
		\draw[thick] (-3,0) circle (20pt);
		\draw[thick] (-3,0) circle (27pt);
		\draw[thick] (-3,0) circle (34pt);
	\end{tikzpicture}
	\caption{Illustration of the contours: The three contours around $-1$ from outside to inside are $\Sigma_{\LL,\out},\Sigma_{\LL}$ and $\Sigma_{\LL,\inn}$ respectively, and the three contours around $0$ from outside to inside are $\Sigma_{\RR,\out},\Sigma_{\RR}$ and $\Sigma_{\RR,\inn}$ respectively.}\label{fig:contours_finite_time}
\end{figure}
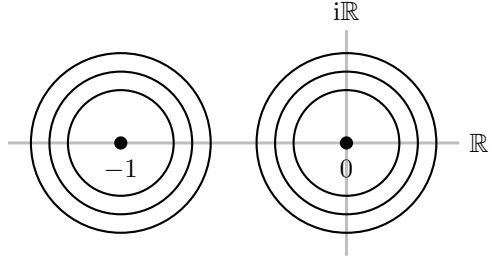	

We also introduce the notation of an integral along a small loop around a point $z_0$ in the complex plane
\begin{equation*}
	\oint_{z_0} f(z) \mathrm{d}z := \int_{|z-z_0|=\epsilon} f(z) \mathrm{d}z,
\end{equation*} 
where $f(z)$ is an arbitrary meromorphic function defined in a neighborhood of $z_0$ and $\epsilon$ is a sufficiently small constant.

\bigskip

The probability density function $p(s_1,s_2;m,n,M,N)$ is defined to be
\begin{equation}
	\label{eq:def_finite_time_density}
	p(s_1,s_2;m,n,M,N)
	:=\oint_0\ddbar{\mathrm{z}}{\mathrm{(1-z)^2}}\sum_{k_1,k_2\ge 1}
	     \frac{1}{(k_1!k_2!)^2}
	     T_{k_1,k_2}(\mathrm{z};s_1,s_2;m,n,M,N)
\end{equation}
with
\begin{equation}
	\label{eq:def_T}
	\begin{split}
		&T_{k_1,k_2}(\mathrm{z};s_1,s_2;m,n,M,N)\\
		&:=
		    \prod_{i_1=1}^{k_1}     
		        \left(  \frac{1}{1-\mathrm{z}} \int_{\Sigma_{\LL,\inn}} 	    
		               \ddbar{u^{(1)}_{i_1}}{}   
		               -\frac{\mathrm{z}}{1-\mathrm{z}} \int_{\Sigma_{\LL,\out}} \ddbar{u^{(1)}_{i_1}}{}
		        \right)
		        \left(  \frac{1}{1-\mathrm{z}}  \int_{\Sigma_{\RR,\inn}}  
		               \ddbar{v^{(1)}_{i_1}}{}
		              -\frac{\mathrm{z}}{1-\mathrm{z}}   \int_{\Sigma_{\RR,\out}}
		               \ddbar{v^{(1)}_{i_1}}{}
		        \right)
		\\
		&\quad \cdot 
		         \prod_{i_2=1}^{k_2}  
		            \int_{\Sigma_\LL} \ddbar{u^{(2)}_{i_2}}{} 
		            \int_{\Sigma_\RR} \ddbar{v^{(2)}_{i_2}}{} 
		       \cdot    
		           \left(1-\mathrm{z}\right)^{k_2}
		           \left(1-\frac{1}{\mathrm{z}}\right)^{k_1}
		       \cdot 
		           \frac{f_1(U^{(1)};s_1)f_2(U^{(2)};s_2)}{f_1(V^{(1)};s_1)f_2(V^{(2)};s_2)}
		       \cdot
		            H(U^{(1)},U^{(2)};V^{(1)},V^{(2)})\\
		&\quad \cdot 
		           \prod_{\ell=1}^2
		           \frac{\left( \Delta(U^{(\ell)}) \right)^2
		           	     \left( \Delta(V^{(\ell)}) \right)^2
	           	         }
		                {\left( \Delta(U^{(\ell)}; V^{(\ell)}) \right)^2}
		       \cdot 
		           \frac{\Delta(U^{(1)};V^{(2)})\Delta(V^{(1)};U^{(2)})}
		                {\Delta(U^{(1)};U^{(2)})\Delta(V^{(1)};V^{(2)})},
	\end{split}
\end{equation}
where the vectors $U^{(\ell)}=(u_1^{(\ell)},\cdots,u_{k_\ell}^{(\ell)})$ and $V^{(\ell)}=(v_1^{(\ell)},\cdots,v^{(\ell)}_{k_\ell})$ for $\ell\in\{1,2\}$, the functions $f_1,f_2$ are defined by
\begin{equation}
	\label{eq:def_f}
	\begin{split}
	f_1(w;s)&:=(w+1)^{-m} w^n e^{s w},\\
	f_2(w;s)&:= (w+1)^{-M+m} w^{N-n} e^{s w},
	\end{split}
\end{equation}
and the function $H$ is defined by
\begin{equation}
	\label{eq:def_H}
	\begin{split}
		&H(U^{(1)},U^{(2)};V^{(1)},V^{(2)})\\
		&:=\frac12 \left(
		                 \sum_{i_1=1}^{k_1} (u^{(1)}_{i_1}-v^{(1)}_{i_1})
		                -\sum_{i_2=1}^{k_2} (u^{(2)}_{i_2}-v^{(2)}_{i_2})
		           \right)^2
		           \left(
		                 1 + \prod_{i_1=1}^{k_1}
		                          \frac{v_{i_1}^{(1)}}
		                               {u_{i_1}^{(1)}}
		                     \prod_{i_2=1}^{k_2}
		                          \frac{u_{i_2}^{(2)}}
		                               {v_{i_2}^{(2)}}		                     
		           \right)\\
		&\quad +
	       \frac12 \left(
	                    -\sum_{i_1=1}^{k_1}    
	                            \left(
	                                (u^{(1)}_{i_1})^2-(v^{(1)}_{i_1})^2
	                            \right)
	                    +\sum_{i_2=1}^{k_2} 
	                            \left(
	                                (u^{(2)}_{i_2})^2-(v^{(2)}_{i_2})^2
	                            \right)
	               \right)
	               \left(
	                     1 - \prod_{i_1=1}^{k_1}
	                             \frac{v_{i_1}^{(1)}}
	                                  {u_{i_1}^{(1)}}
	                         \prod_{i_2=1}^{k_2}
	                             \frac{u_{i_2}^{(2)}}
	                                  {v_{i_2}^{(2)}}
	               \right).
	\end{split}
\end{equation}

\bigskip

We remark that the formula~\eqref{eq:def_finite_time_density} has a very similar structure with the two-point distribution formula of TASEP in \cite{Liu19} (with step initial condition), except that we have different $\rz$ factors in the integral, and that we have an extra factor $H(U^{(1)},U^{(2)};V^{(1)},V^{(2)})$. See equations (2.1) and (2.14) in \cite{Liu19}. It is not hard to prove that $T_{k_1,k_2}$ becomes zero when $k_1$ or $k_2$ becomes large, hence the formula~\eqref{eq:def_finite_time_density}  only involves finite many nonzero terms in the summation and is well defined. \footnote{In fact, we can view the integrand of~\eqref{eq:def_T} as a function of $V^{(1)}$ and $V^{(2)}$, which equals to the product of the following three terms: $\Delta(V^{(1)})\Delta(V^{(2)})$, a Cauchy-type factor $\Ch(V^{(1)};V^{(2)})=\Delta(V^{(1)})\Delta(V^{(2)}) /\Delta(V^{(1)};V^{(2)})$ (see the definition in~\eqref{eq:2021_12_23_02}), and some function which is meromorphic for each $v_{i_\ell}^{(\ell)}$ with a possible pole at $0$ but the degree of this pole is at most $\max\{n,N-n+1\}$. Note that expanding the first term $\Delta(V^{(1)})\Delta(V^{(2)})$ gives a sum of terms $\prod_{1\le \ell_1\le k_1}(v_{\sigma(\ell_1)}^{(1)})^{k_1-\ell_1}\prod_{1\le \ell_2\le k_2}(v_{\pi(\ell_2)}^{(2)})^{k_2-\ell_2}$ over permutations $\sigma\in S_{k_1}$ and $\pi\in S_{k_2}$, here $S_k$ denotes the permutation group of $\{1,2,\cdots,k\}$. If $k_1$ is large enough (the case when $k_2$ is large is similar), for example if $k_1>N$, the integrand is analytic for $v_{\sigma(1)}^{(1)}$ at $0$ by checking the degrees. So when we integrate $v_{\sigma(1)}^{(1)}$, the only possible nontrivial contribution is from the residues $v_{\sigma(1)}^{(1)}=v_{j}^{(2)}$ if $v_j^{(2)}$ lies inside the contour of $v_{\sigma(1)}^{(1)}$ due to the Cauchy-type factor. However, if we further integrate $v_{j}^{(2)}$ we find each residue contribution is also zero by checking the degree of $v_j^{(2)}$ which is $k_1-1-n-(N-n+1)=k_1-N>0$. We remark that the proof does not rely on the explicit formula of $H$ or the variable $\rz$, and it is similar to the argument for the two-point distribution formula of TASEP (see Remark 2.8 of \cite{Liu19}) where they do not have the factor $H$.}

Finally, by exchanging the integral and summations, and using the identity $\int_0^\infty \frac{f_\ell(U^{(\ell)};s_\ell)}{f_\ell(V^{(\ell)};s_\ell)}\mathrm{d}s_\ell =  \frac{f_\ell(U^{(\ell)};0)}{f_\ell(V^{(\ell)};0)}\cdot \frac{1}{\sum_{i_\ell=1}^{k_\ell}(v_{i_\ell}^{(\ell)}-u_{i_\ell}^{(\ell)})}$ since $\mathrm{Re} (v_{i_\ell}^{(\ell)}-u_{i_{\ell}}^{(\ell)})<0$ due to the locations of the contours, we obtain
\begin{equation}
	\label{eq:finite_time_geodesic_distribution}
	\begin{split}
	\prob
	\left(
	\mathbf{r},\mathbf{r}'\in\gd_{\mathbf{p}}(\mathbf{q})
	\right) &= \int_0^\infty\int_0^\infty p(s_1,s_2;m,n,M,N)\mathrm{d}s_1\mathrm{d}s_2 \\
	&= \oint_0\ddbar{\mathrm{z}}{\mathrm{(1-z)^2}}\sum_{k_1,k_2\ge 1}
	\frac{1}{(k_1!k_2!)^2}
	\mathcal{T}_{k_1,k_2}(\mathrm{z};m,n,M,N),
	\end{split}
\end{equation}
where
\begin{equation}
	\label{eq:2021_12_22_01}
	\begin{split}
		&\mathcal{T}_{k_1,k_2}(\mathrm{z};m,n,M,N)\\
		&:=
		\prod_{i_1=1}^{k_1}     
		\left(  \frac{1}{1-\mathrm{z}} \int_{\Sigma_{\LL,\inn}} 	    
		\ddbar{u^{(1)}_{i_1}}{}   
		-\frac{\mathrm{z}}{1-\mathrm{z}} \int_{\Sigma_{\LL,\out}} \ddbar{u^{(1)}_{i_1}}{}
		\right)
		\left(  \frac{1}{1-\mathrm{z}}  \int_{\Sigma_{\RR,\inn}}  
		\ddbar{v^{(1)}_{i_1}}{}
		-\frac{\mathrm{z}}{1-\mathrm{z}}   \int_{\Sigma_{\RR,\out}}
		\ddbar{v^{(1)}_{i_1}}{}
		\right)
		\\
		&\quad \cdot 
		\prod_{i_2=1}^{k_2}  
		\int_{\Sigma_\LL} \ddbar{u^{(2)}_{i_2}}{} 
		\int_{\Sigma_\RR} \ddbar{v^{(2)}_{i_2}}{} 
		\cdot    
		\left(1-\mathrm{z}\right)^{k_2}
		\left(1-\frac{1}{\mathrm{z}}\right)^{k_1}
		\cdot 
		\frac{f_1(U^{(1)};0)f_2(U^{(2)};0)}{f_1(V^{(1)};0)f_2(V^{(2)};0)}
		\cdot
		\frac{1}{\prod_{\ell=1}^2\sum_{i_\ell=1}^{k_\ell} (u_{i_\ell}^{(\ell)}-v_{i_\ell}^{(\ell)})}
		\\
		&\quad \cdot
		H(U^{(1)},U^{(2)};V^{(1)},V^{(2)})\cdot 
		\prod_{\ell=1}^2
		\frac{\left( \Delta(U^{(\ell)}) \right)^2
			\left( \Delta(V^{(\ell)}) \right)^2
		}
		{\left( \Delta(U^{(\ell)}; V^{(\ell)}) \right)^2}
		\cdot 
		\frac{\Delta(U^{(1)};V^{(2)})\Delta(V^{(1)};U^{(2)})}
		{\Delta(U^{(1)};U^{(2)})\Delta(V^{(1)};V^{(2)})}.
	\end{split}
\end{equation}

\subsection{Limiting joint distribution of geodesic location and last passage times}

For any two lattice points $\mathbf{p}=(\mathrm{p}_1,\mathrm{p}_2)$ and $\mathbf{q}=(\mathrm{q}_1,\mathrm{q}_2)$ satisfying $\mathrm{p}_1\le \mathrm{q}_1$ and $\mathrm{p}_2\le \mathrm{q}_2$, we define
\begin{equation}
	d(\mathbf{p},\mathbf{q})
	  :=\left(
	           \sqrt{\mathrm{q}_1-\mathrm{p}_1}
	          +\sqrt{\mathrm{q}_2-\mathrm{p}_2}
	    \right)^2.
\end{equation}

We say a geodesic $\gd_{\mathbf{p}}(\mathbf{q})$ exits a set $A$ at a point $\mathbf{r}$, if and only if the geodesic intersects $A$ and $\mathbf{r}$ is the last point of the intersection, i.e., $\mathbf{r}\in\gd_{\mathbf{p}}(\mathbf{q})\cap A$ and $\mathbf{r}_+\in\gd_{\mathbf{p}}(\mathbf{q})\setminus A$.

\begin{thm}
	\label{thm:limiting_geodesic}
	Suppose $\alpha>0$, $\gamma\in(0,1)$ are fixed constants. Assume $x_1,x_2,x'_1,x'_2$ are four real numbers satisfying $x_1>x'_1$ and $x_2<x'_2$. Let
	\begin{equation}
		\begin{split}
			M&= [\alpha N],\\
			m&=[\gamma\alpha N + x_1 \alpha^{2/3}(1+\sqrt{\alpha})^{2/3} N^{2/3}],\\
			n&=[\gamma N + x_2 \alpha^{-1/3}(1+\sqrt{\alpha})^{2/3} N^{2/3}],\\
			m'&=[\gamma\alpha N + x'_1 \alpha^{2/3}(1+\sqrt{\alpha})^{2/3} N^{2/3}],\\
			n'&=[\gamma N + x'_2 \alpha^{-1/3}(1+\sqrt{\alpha})^{2/3} N^{2/3}],
		\end{split}
	\end{equation}	
where $[x]$ denotes the largest integer which is smaller than or equal to $x$.
	 Suppose $\pi$ is an up/left lattice path from $(m,n)$ to $(m',n')$. Then
	 \begin{equation}
	 	\label{eq:prelimit_geodesic}
	 	\lim_{N\to\infty} \prob
	 	\left(
	 	\begin{array}{l}
	 		\gd_{(1,1)}{(M,N)}
	 		 \text{ intersects $\pi$,} \\
	 		\mbox{and }L_{(1,1)}(\mathbf{p})\ge d((1,1),\mathbf{p}) + \mathrm{t}_1\cdot \alpha^{-1/6}(1+\sqrt{\alpha})^{4/3} N^{1/3}, \\
	 		\mbox{and } L_{\mathbf{p}_+}(M,N)\ge d(\mathbf{p}_+,(M,N)) + \mathrm{t}_2\cdot \alpha^{-1/6}(1+\sqrt{\alpha})^{4/3} N^{1/3},\\
	 		\mbox{where $\mathbf{p}$ denotes the exit point of $\gd_{(1,1)}{(M,N)}$ on $\pi$}\\
	   \end{array}
	 	\right)
	 \end{equation}
 exists and is independent of the choice of $\pi$. 
  The limit equals to
 \begin{equation}
 	\label{eq:limit_geodesic_distribution}
 	\int_{x_2-x_1}^{x'_2-x'_1} \int_{\mathrm{t}_1}^\infty\int_{\mathrm{t}_2}^\infty
 	\limp(\mathrm{s}_1,\mathrm{s}_2,\mathrm{x};\gamma)
 	\mathrm{d}\mathrm{s}_2\mathrm{d}\mathrm{s}_1\mathrm{d}\mathrm{x},
  \end{equation}
where the joint  probability density function $\limp(\mathrm{s}_1,\mathrm{s}_2,\mathrm{x};\gamma)$ is defined in~\eqref{eq:def_limiting_density}.
\end{thm}

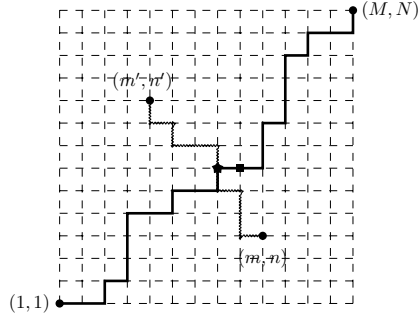
\begin{figure}[h]
	\centering\scalebox{1}{	\begin{tikzpicture}
			\draw[step=0.3cm, dashed,line width= 0.1pt] (0,0) grid (3.9,3.9);
			\draw[line width= 1pt] (0,0) -- (0.6,0) -- (0.6,0.3) -- (0.9,0.3) -- (0.9,1.2) -- (1.5,1.2) -- (1.5,1.5) --(2.1,1.5)--(2.1,1.8);
			\draw[line width= 1pt] (2.1,1.8)--(2.7,1.8) -- (2.7,2.4) -- (3,2.4) -- (3,3.3) -- (3.3, 3.3) -- (3.3, 3.6) -- (3.9,3.6) -- (3.9,3.9);
			\node at (2.1,1.8)[star,fill,inner sep=1.1pt]{};
			\node at (2.4,1.8)[rectangle,fill,inner sep=1.4pt]{};								
			\draw[decoration={aspect=0.03, segment length=0.4mm, amplitude=0.12mm,coil},decorate] (1.2,2.7)--(1.2,2.4)--(1.5,2.4)--(1.5,2.1)--(2.1,2.1)--(2.1,1.5)--(2.4,1.5)--(2.4,0.9)--(2.7,0.9);
			\node at (2.7,0.9)[circle,fill,inner sep=1.1pt]{};
			\node at (0,0)[circle,fill,inner sep=1.1pt]{};
			\node at (3.9,3.9)[circle,fill,inner sep=1.1pt]{};
			\node at (1.2,2.7)[circle,fill,inner sep=1.1pt]{};
			  \node[scale=0.7] at (2.7,0.6)  {$(m,n)$};
			  \node[scale=0.7] at (1.1,2.95)  {$(m',n')$};
			  \node[scale=0.7] at (-0.4,0)  {$(1,1)$};
			  \node[scale=0.7] at (4.4,3.9)  {$(M,N)$};
	\end{tikzpicture}}
\caption{The thick path denotes the geodesic $\gd_{(1,1)}(M,N)$. The spring-shaped lattice path denotes $\pi$. The star-shaped point is the exit point of $\gd_{(1,1)}(M,N)$ on $\pi$, and the square-shaped point is the next point on $\gd_{(1,1)}(M,N)$ after the exit point.}\label{fig:limit_theorem}
\end{figure}

See Figure~\ref{fig:limit_theorem} for an illustration. The proof of Theorem~\ref{thm:limiting_geodesic} is provided in Section~\ref{sec:proof_thm2}.

We expect that the geodesic is around a straight line from $(1,1)$ to $(M,N)$. The line is of slope $\alpha^{-1}\approx N/M$. Then $x_2-x_1$ and $x'_2-x'_1$ can be viewed as (after appropriate scaling) the shifts of moving $(m,n)$ and $(m',n')$ to the line. Similarly, in the density function $\limp(\mathrm{s}_1,\mathrm{s}_2,\mathrm{x};\gamma)$, $\mathrm{x}$ can be viewed as the shift of moving the exit point $\mathbf{p}$ to the line. See Figure~\ref{fig:limit_theorem_proof} at the beginning of Section~\ref{sec:proof_thm2} for an illustration.

It might look surprising at a first glance that the limiting distribution is independent of $\pi$, but only depends on the locations of the endpoints. Here we provide an intuitive explanation.  Suppose we have a different up/left  lattice path $\pi'$ from $(m,n)$ to $(m',n')$. For any point $\mathbf{q}\in\pi$, we can find a unique point $\mathbf{q'}\in\pi'$ such that $\mathbf{q}-\mathbf{q}'\in\{(\alpha y,y):y\in\realR\}$. Note that the distance between $\mathbf{q}$ and $\mathbf{q}'$ is at most of order $O(N^{2/3})\ll o(N)$. By the uniform slow decorrelation of the directed last passage percolation \cite{Corwin-Ferrari-Peche12,Corwin-Liu-Wang16}, $N^{-1/3}({L_{(1,1)}(\mathbf{q})-d((1,1),\mathbf{q})})-N^{-1/3}({L_{(1,1)}(\mathbf{q}')-d((1,1),\mathbf{q}')})$ converges to $0$ in probability as $N\to\infty$. Moreover, with appropriate scaling, the limiting process of the last passage times from $(1,1)$ (and from $(M,N)$ similarly) to the points of $\pi$ has the same law as that  to the points of $\pi'$. Therefore we expect the limit of~\eqref{eq:prelimit_geodesic} is independent of $\pi$. This probabilistic argument is heuristic but it might be possible to make it rigorous. In this paper, we will use an analytical way to show this independence instead. See the argument after Proposition~\ref{prop:asymptotics_density} in Section~\ref{sec:proof_thm2}.

Note that the geodesic $\gd_{(1,1)}{(M,N)}$ intersects a rectangle with vertices $(m,n), (m,n'), (m',n')$ and $(m',n)$ if and only if $\gd_{(1,1)}{(M,N)}$ intersects a lattice path from  $(m,n)$ to $(m',n')$.  Thus by setting $\mathrm{t}_1,\mathrm{t}_2\to-\infty$ we immediately have
	\begin{equation}
		\begin{split}
		&\lim_{N\to\infty} \prob
		\left(\gd_{(1,1)}{(M,N)} \text{ intersects the rectangle with vertices $(m,n), (m,n'), (m',n')$ and $(m',n)$}\right)\\
		&=\int_{x_2-x_1}^{x'_2-x'_1} \int_{-\infty}^\infty\int_{-\infty}^\infty
		\limp(\mathrm{s}_1,\mathrm{s}_2,\mathrm{x};\gamma)
		\mathrm{d}\mathrm{s}_2\mathrm{d}\mathrm{s}_1\mathrm{d}\mathrm{x}.
		\end{split}
	\end{equation}

\bigskip

Now we discuss an application of Theorem~\ref{thm:limiting_geodesic}. 

\begin{cor}
	\label{cor:argmax_Airys}
	Let $\mathcal{A}^{(1)}$ and $\mathcal{A}^{(2)}$ be two independent Airy$_2$ processes. Denote the parabolic Airy$_2$ processes $\mathcal{\hat A}^{(\ell)}(\mathrm{x})=\mathcal{A}^{(\ell)}(\mathrm{x})-\mathrm{x}^2$, $\ell=1,2$. Suppose $\gamma\in(0,1)$ is a fixed constant. Denote
	\begin{equation*}
		\mathcal{T}=\mathrm{argmax}_{\mathrm{x}}\left(\gamma^{1/3}\mathcal{\hat A}^{(1)}\left(\frac{\mathrm{x}}{2\gamma^{2/3}}\right)+(1-\gamma)^{1/3}\mathcal{\hat A}^{(2)}\left(\frac{\mathrm{x}}{2(1-\gamma)^{2/3}}\right)\right).
	\end{equation*}
Then $\limp(\mathrm{s}_1,\mathrm{s}_2,\mathrm{x};\gamma)$ is the joint probability density function of $\gamma^{1/3}\mathcal{ A}^{(1)}\left(\frac{\mathcal{T}}{2\gamma^{2/3}}\right),(1-\gamma)^{1/3}\mathcal{A}^{(2)}\left(\frac{\mathcal{T}}{2(1-\gamma)^{2/3}}\right)$ and $\mathcal{T}$.
\end{cor}

\begin{proof}
	Denote $\pi$ the line $\{(x,y):x+y=2\gamma N\}$.
	It is known \cite{Johansson03} that the processes of the last passage times from $(1,1)$ (or $(N,N)$) to the points on $\pi$ after appropriate scaling converge to two independent parabolic Airy$_2$ processes as $N\to\infty$. More explicitly, for any constant $K$,
	\begin{equation}
		\label{eq:2021_12_28_01}
		\frac{L_{(1,1)}(\gamma N - 2^{-1/3}\mathrm{x} N^{2/3},\gamma N + 2^{-1/3}\mathrm{x} N^{2/3})-4\gamma N}{2^{4/3}N^{1/3}} \to \gamma^{1/3}\mathcal{\hat A}^{(1)}\left(\frac{\mathrm{x}}{2\gamma^{2/3}}\right),\quad |\mathrm{x}|\le K
	\end{equation}
and
\begin{equation}
	\label{eq:2021_12_28_02}
	\frac{L_{(\gamma N - 2^{-1/3}\mathrm{x} N^{2/3},\gamma N + 2^{-1/3}\mathrm{x} N^{2/3})}(N,N)-4(1-\gamma) N}{2^{4/3}N^{1/3}} \to (1-\gamma)^{1/3}\mathcal{\hat A}^{(2)}\left(\frac{\mathrm{x}}{2(1-\gamma^{2/3})}\right),\quad |\mathrm{x}|\le K
\end{equation}
as $N\to\infty$. Both processes are tight in the space of continuous functions on $[-K,K]$ (see  \cite[Theorem 2.3]{Ferrari-Occelli18} for example). Note that the geodesic passes through a point $\mathbf{q}$ on the line $\pi$ if and only if $L_{(1,1)}(\mathbf{q})+L_{\mathbf{q}}(N,N)$ reaches the maximum. And the probability that this intersection point $\mathbf{q}$ lies outside of $\{(\gamma N - 2^{-1/3}\mathrm{x} N^{2/3},\gamma N + 2^{-1/3}\mathrm{x} N^{2/3}): |\mathrm{x}|\le K\}$ decays exponentially as $N\to\infty$ and $K$ becomes large (see \cite[Proposition 2.1]{Baik-Liu16b} for example). Also note that the argmax $\mathcal{T}$ is unique since it represents the geodesic location in the limiting directed landscape and the geodesic is unique (see \cite{Dauvergne-Virag21}). Using the above facts, we conclude that the location of the intersection of $\gd_{(1,1)}(N,N)$ and $\pi$, the argmax of the left hand side of~\eqref{eq:2021_12_28_01}$+$\eqref{eq:2021_12_28_02}, converges to $\mathcal{T}$. Now we apply Theorem~\ref{thm:limiting_geodesic} with $\alpha=1$ and use the facts that 
\begin{equation*}
	d_{(1,1)}(\gamma N - 2^{-1/3}\mathrm{x} N^{2/3},\gamma N + 2^{-1/3}\mathrm{x} N^{2/3})=4\gamma N+\frac{\mathrm{x}^2}{2^{2/3}\gamma}N^{1/3}+o(1)
\end{equation*}
and
\begin{equation*}
	d_{(\gamma N - 2^{-1/3}\mathrm{x} N^{2/3},\gamma N + 2^{-1/3}\mathrm{x} N^{2/3})}(N,N)=4(1-\gamma) N+\frac{\mathrm{x}^2}{2^{2/3}(1-\gamma)}N^{1/3}+o(1).
\end{equation*}
Corollary~\ref{cor:argmax_Airys} follows immediately.
\end{proof}
The explicit distribution of $\mathcal{T}$ was an interesting open problem in the community before, see \cite[Problem 14.4(a)]{Dauvergne-Ortmann-Virag18} for example. Our result above resolves this problem. It is also possible to apply this result and the formula of $\limp(\mathrm{s}_1,\mathrm{s}_2,\mathrm{x};\gamma)$ to obtain some properties of the directed landscape, the limiting four-parameter random field of the directed last passage percolation. For example, in a follow-up paper \cite{Liu21} we proved that when the height of the directed landscape at a point is sufficiently large, the geodesic to this point is rigid and the location has a Gaussian distribution under appropriate scaling.

We remark that the density function $\limp(\mathrm{s}_1,\mathrm{s}_2,\mathrm{x};\gamma)$ can be related to the well-known GUE Tracy-Widom distribution. Note that the max of $\gamma^{1/3}\mathcal{\hat A}^{(1)}\left(\frac{\mathrm{x}}{2\gamma^{2/3}}\right)+(1-\gamma)^{1/3}\mathcal{\hat A}^{(2)}\left(\frac{\mathrm{x}}{2(1-\gamma)^{2/3}}\right)$ satisfies 
\begin{equation*}
	\prob\left(\max_{\mathrm{x}\in\realR} \left\{\gamma^{1/3}\mathcal{\hat A}^{(1)}\left(\frac{\mathrm{x}}{2\gamma^{2/3}}\right)+(1-\gamma)^{1/3}\mathcal{\hat A}^{(2)}\left(\frac{\mathrm{x}}{2(1-\gamma)^{2/3}}\right)\right\}\le \mathrm{s}\right) = F_{GUE}(\mathrm{s}),
\end{equation*}
where $F_{GUE}(\mathrm{s})$ is the GUE Tracy-Widom distribution. See \cite{Baik-Liu14,Baik-Liu13} for more details. By applying the Corollary~\ref{cor:argmax_Airys} and noting $\mathcal{\hat A}^{(\ell)}(\mathrm{x})=\mathcal{A}^{(\ell)}(\mathrm{x})-\mathrm{x}^2$, we have 
\begin{equation}\int_\realR\mathrm{d}\mathrm{x} \iint_{\mathrm{s}_1+\mathrm{s_2}\le\mathrm{s}}\mathrm{ds_1ds_2} \limp\left(\mathrm{s}_1+\frac{\mathrm{x}^2}{4\gamma},\mathrm{s}_2+\frac{\mathrm{x}^2}{4(1-\gamma)},\mathrm{x};\gamma\right) =F_{GUE}(\mathrm{s}).
	\end{equation}

\bigskip

One might be able to obtain the tail estimates for the geodesic using the formula~\eqref{eq:main_thm01}.  After a preliminary calculation, we have the following conjecture.
\begin{conj}
	Let $M,N$ and $m,n$ be numbers satisfying the scaling~\eqref{eq:scaling} in Theorem~\ref{thm:limiting_geodesic}, then
	\begin{equation}
		\lim_{N\to\infty}\prob\left(\gd_{(1,1)}(M,N) \text{ lies above } (m,n)\right) = e^{-c\mathrm{x}^3+o(\mathrm{x})} \quad \text{with}\quad c=\frac1{6(\gamma(1-\gamma))^{3/2}},
	\end{equation}
	when $\mathrm{x}=x_2-x_1$ becomes large.
\end{conj}
It also might be possible to obtain a more accurate estimate from this formula. We leave it as a future project.

\subsection{The limiting density function $\limp(\mathrm{s}_1,\mathrm{s}_2,\mathrm{x};\gamma)$}

The limiting density function $\limp(\mathrm{s}_1,\mathrm{s}_2,\mathrm{x};\gamma)$ has a similar structure as the finite time probability density function $p(s_1,s_2;m,n,M,N)$. Before we write down the formula, we introduce some contours. Suppose $\Gamma_{\LL,\inn},\Gamma_{\LL}$ and $\Gamma_{\LL,\out}$ are three disjoint contours on the left half plane each of which starts from $e^{-2\pi\ii/3}\infty$ and ends to $e^{2\pi\ii/3}\infty$. Here $\Gamma_{\LL,\inn}$ is the leftmost contour and $\Gamma_{\LL,\out}$ is the rightmost contour. The index ``$\inn$'' and ``$\out$'' refer to the relative location compared with $-\infty$. Similarly, suppose  $\Gamma_{\RR,\inn},\Gamma_{\RR}$ and $\Gamma_{\RR,\out}$ are three disjoint contours on the right half plane each of which starts from $e^{-\pi\ii/3}\infty$ and ends to $e^{\pi\ii/3}\infty$. Here the index ``$\inn$'' and ``$\out$'' refer to the relative location compared with $+\infty$, hence $\Gamma_{\RR,\inn}$ is the rightmost contour and $\Gamma_{\RR,\out}$ is the leftmost contour. See Figure~\ref{fig:contours_limit} for an illustration of these contours.

\begin{figure}[t]
	\centering
	\begin{tikzpicture}[scale=1]
		\draw [line width=0.4mm,lightgray] (-3,0)--(3,0) node [pos=1,right,black] {$\realR$};
		\draw [line width=0.4mm,lightgray] (0,-1)--(0,1) node [pos=1,above,black] {$\mathrm{i}\realR$};
		\fill (0,0) circle[radius=2pt] node [below,shift={(0pt,-3pt)}] {$0$};

		\path [draw=black,thick,postaction={mid3 arrow={black,scale=1.2}}]	(3,-5.1/6) 
		to [out=160,in=-90] (1.5,0)
		to [out=90,in=-160] (3,5.1/6);
		\path [draw=black,thick,postaction={mid3 arrow={black,scale=1.2}}]	(2.6,-6.1/6) 
		to [out=160,in=-90] (1,0)
		to [out=90,in=-160] (2.6,6.1/6);
		\path [draw=black,thick,postaction={mid3 arrow={black,scale=1.2}}]	(2.3,-7.1/6) 
		to [out=160,in=-90] (0.5,0)
		to [out=90,in=-160] (2.3,7.1/6);
		
		\path [draw=black,thick,postaction={mid3 arrow={black,scale=1.5}}]	(-3,-5.1/6) 
		to [out=20,in=-90] (-1.5,0)
		to [out=90,in=-20] (-3,5.1/6);
		\path [draw=black,thick,postaction={mid3 arrow={black,scale=1.5}}]	(-2.6,-6.1/6) 
		to [out=20,in=-90] (-1,0)
		to [out=90,in=-20] (-2.6,6.1/6);
		\path [draw=black,thick,postaction={mid3 arrow={black,scale=1.5}}]	(-2.3,-7.1/6) 
		to [out=20,in=-90] (-0.5,0)
		to [out=90,in=-20] (-2.3,7.1/6);
	\end{tikzpicture}
	\caption{Illustration of the contours: The three contours in the left half plane from left to right are $\Gamma_{\LL,\inn}, \Gamma_\LL$ and $\Gamma_{\LL,\out}$ respectively, and the three contours in the right half plane from left to right are $\Gamma_{\RR,\out}, \Gamma_{\RR}$ and $\Gamma_{\RR,\inn}$ respectively.}\label{fig:contours_limit}
\end{figure}
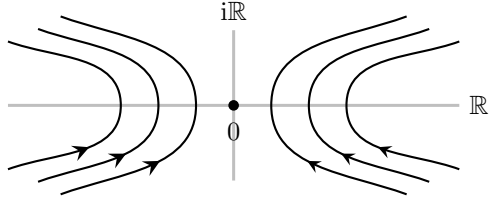

\bigskip

The probability density function $\limp(\mathrm{s}_1,\mathrm{s}_2,\mathrm{x};\gamma)$ is defined to be
\begin{equation}
	\label{eq:def_limiting_density}
	\limp(\mathrm{s}_1,\mathrm{s}_2,\mathrm{x};\gamma)
	:=\oint_0\ddbar{\mathrm{z}}{\mathrm{(1-z)^2}}\sum_{k_1,k_2\ge 1}
	\frac{1}{(k_1!k_2!)^2}
	\mathrm{T}_{k_1,k_2}(\mathrm{z};\mathrm{s}_1,\mathrm{s}_2,\mathrm{x};\gamma)
\end{equation}
with
\begin{equation}
	\label{eq:def_limit_T}
	\begin{split}
		&\mathrm{T}_{k_1,k_2}(\mathrm{z};\mathrm{s}_1,\mathrm{s}_2,\mathrm{x};\gamma)\\
		&:=(-1)^{k_1+k_2}
		\prod_{i_1=1}^{k_1}     
		\left(  \frac{1}{1-\mathrm{z}} \int_{\Gamma_{\LL,\inn}} 	    
		\ddbar{\xi^{(1)}_{i_1}}{}   
		-\frac{\mathrm{z}}{1-\mathrm{z}} \int_{\Gamma_{\LL,\out}} \ddbar{\xi^{(1)}_{i_1}}{}
		\right)
		\left(  \frac{1}{1-\mathrm{z}}  \int_{\Gamma_{\RR,\inn}}  
		\ddbar{\eta^{(1)}_{i_1}}{}
		-\frac{\mathrm{z}}{1-\mathrm{z}}   \int_{\Gamma_{\RR,\out}}
		\ddbar{\eta^{(1)}_{i_1}}{}
		\right)
		\\
		&\quad \cdot 
		\prod_{i_2=1}^{k_2}  
		\int_{\Gamma_\LL} \ddbar{\xi^{(2)}_{i_2}}{} 
		\int_{\Gamma_\RR} \ddbar{\eta^{(2)}_{i_2}}{} 
		\cdot    
		\left(1-\mathrm{z}\right)^{k_2}
		\left(1-\frac{1}{\mathrm{z}}\right)^{k_1}
		\cdot 
		\frac
		   {\mathrm{f}_1(\boldsymbol{\xi}^{(1)};\mathrm{s}_1)
			\mathrm{f}_2(\boldsymbol{\xi}^{(2)};\mathrm{s}_2)}
		   {\mathrm{f}_1(\boldsymbol{\eta}^{(1)};\mathrm{s}_1)
		   	\mathrm{f}_2(\boldsymbol{\eta}^{(2)};\mathrm{s}_2)}
		\cdot
		\mathrm{H}(\boldsymbol{\xi}^{(1)},\boldsymbol{\eta}^{(1)};
		           \boldsymbol{\xi}^{(2)},\boldsymbol{\eta}^{(2)})\\
		&\quad \cdot 
		\prod_{\ell=1}^2
		\frac{\left( \Delta(\boldsymbol{\xi}^{(\ell)}) \right)^2
			  \left( \Delta(\boldsymbol{\eta}^{(\ell)}) \right)^2
		     }
		     {\left( \Delta(\boldsymbol{\xi}^{(\ell)}; \boldsymbol{\eta}^{(\ell)}) \right)^2}
		\cdot 
		\frac{\Delta(\boldsymbol{\xi}^{(1)};\boldsymbol{\eta}^{(2)})
			  \Delta(\boldsymbol{\eta}^{(1)};\boldsymbol{\xi}^{(2)})
		     }
		     {\Delta(\boldsymbol{\xi}^{(1)};\boldsymbol{\xi}^{(2)})
		      \Delta(\boldsymbol{\eta}^{(1)};\boldsymbol{\eta}^{(2)})
	         },
	\end{split}
\end{equation}
where the vectors $\boldsymbol{\xi}^{(\ell)}=(\xi_1^{(\ell)},\cdots,\xi_{i_\ell}^{(\ell)})$ and $\boldsymbol{\eta}^{(\ell)}=(\eta_1^{(\ell)},\cdots,\eta^{(\ell)}_{i_\ell})$ for $\ell\in\{1,2\}$, the functions $\mathrm{f}_1,\mathrm{f}_2$ are defined by
\begin{equation}
	\label{eq:def_rmf}
	\begin{split}
	\mathrm{f}_1(\zeta;\mathrm{s})&:=\exp\left(-\frac{\gamma}{3}\zeta^3-\frac12\mathrm{x}\zeta^2+\left(\mathrm{s}-\frac{\mathrm{x}^2}{4\gamma}\right)\zeta\right),\\
	 \mathrm{f}_2(\zeta;\mathrm{s})& := \exp\left(-\frac{(1-\gamma)}{3}\zeta^3+\frac12\mathrm{x}\zeta^2+\left(\mathrm{s}-\frac{\mathrm{x}^2}{4(1-\gamma)}\right)\zeta\right),
	 \end{split}
\end{equation}
and the function $\mathrm{H}$ is defined by
\begin{equation}
	\label{eq:def_rmH}
	\mathrm{H}(\boldsymbol{\xi}^{(1)},\boldsymbol{\eta}^{(1)};
	\boldsymbol{\xi}^{(2)},\boldsymbol{\eta}^{(2)})
	=\frac{1}{12}\mathrm{S}_1^4+\frac14\mathrm{S}_2^2-\frac13\mathrm{S}_1\mathrm{S}_3
\end{equation}
with
\begin{equation}
	\label{eq:def_S}
	\mathrm{S}_\ell
	=\mathrm{S}_\ell(\boldsymbol{\xi}^{(1)},\boldsymbol{\eta}^{(1)};
	\boldsymbol{\xi}^{(2)},\boldsymbol{\eta}^{(2)})
	=\sum_{i_1=1}^{k_1}
	      \left( \left(\xi_{i_1}^{(1)}\right)^{\ell}
	            -\left(\eta_{i_1}^{(1)}\right)^{\ell}
	      \right)
	 -\sum_{i_2=1}^{k_2}
	      \left( \left(\xi_{i_2}^{(2)}\right)^{\ell}
	             -\left(\eta_{i_2}^{(2)}\right)^{\ell}
	      \right).
\end{equation}

\begin{rmk}
	\label{rmk:symmetry_T}
	It can be directly verified that $\mathrm{T}$ is symmetric on $\mathrm{x}$, i.e., it satisfies $\mathrm{T}_{k_1,k_2}(\mathrm{z};\mathrm{s}_1,\mathrm{s}_2,\mathrm{x};\gamma)=\mathrm{T}_{k_1,k_2}(\mathrm{z};\mathrm{s}_1,\mathrm{s}_2,-\mathrm{x};\gamma)$. In fact, one can see it clearly by changing variables $\xi_{i_\ell}^{(\ell)}=-\tilde\eta_{i_\ell}^{(\ell)}$ and $\eta_{i_\ell}^{(\ell)}=-\tilde\xi_{i_\ell}^{(\ell)}$ for $1\le i_\ell\le k_\ell$ and $\ell=1,2$.
\end{rmk}

One can prove that the summation is absolutely convergent in~\eqref{eq:def_limiting_density} due to the super-exponential decay of $\mathrm{f}_\ell$ along the integral contours. The proof is similar to that of Lemma~\ref{lm:hat_rt_uniform_bounds} so we omit it.

\section*{Acknowledgments}
We would like to thank Jinho Baik, Ivan Corwin, Duncan Dauvergne, Patrik Ferrari, Kurt Johansson, Daniel Remenik and B\'alint Vir\'ag for the comments and suggestions. The work was supported by the University of Kansas Start Up Grant, the University of Kansas New Faculty General Research Fund, Simons Collaboration Grant No. 637861, and  NSF grant
DMS-1953687.

\section{Finite time formulas and proof of Theorem~\ref{thm:finite_time_geodesic1}}
\label{sec:proof_thm1}
\subsection{Outline of the proof}

Theorem~\ref{thm:finite_time_geodesic1} states two formulas for different locations of $\mathbf{r}'$. The equation~\eqref{eq:main_thm01} holds when $\mathbf{r}'=(m+1,n)$, i.e., when $\mathbf{r}'$ is at the same row as $\mathbf{r}$. The case when $\mathbf{r}'$ is at the same column as $\mathbf{r}$ follows by switching the rows and columns of the model. Thus it is sufficient to show the equation~\eqref{eq:main_thm01} with $\mathbf{r}'=(m+1,n)$.

The proof involves a few computations and identities. We would like to split the proof into three steps, each of which ends with an identity about the probability density function $p(s_1,s_2;m,n,M,N)$. We will outline the steps and state these main identities in this subsection and leave their proofs in subsequent subsections.

In the first step, we obtain a formula for $p(s_1,s_2;m,n,M,N)$. The main idea is to convert the desired probability to a sum of the product of two transition probabilities, and evaluate the sum explicitly. There are two types of transition probabilities for the exponential directed last passage percolation. One is the transition probability by viewing its equivalent model, the so-called TASEP, as a Markov process with respect to time \cite{Schutz97}. The second one is the transition probability by viewing the model as a Markov chain along one dimension on the space \cite{Johansson10}. It turns out that only the later one can be used to find an exact formula for $p(s_1,s_2;m,n,M,N)$. If one uses the transition probabilities of TASEP instead, there will be an $\bigO(1)$ error on the finite time formulas but the resulting limit probability densities $\mathrm{p}(\mathrm{s}_1,\mathrm{s}_2,\mathrm{x};\gamma)$ is the same. We will consider this approach in a follow-up paper.

Using the transition probability formula of \cite{Johansson10} and an summation identity for the product of two eigenfunctions, we obtain the following proposition.
\begin{prop}
	\label{prop:p_step1}
	We have the following formula for $p(s_1,s_2;m,n,M,N)$
	\begin{equation}
	\label{eq:finite_time_formula01}
	\begin{split}
	   & p(s_1,s_2;m,n,M,N) \\
	                      & = \frac{(-1)^{N(N-1)/2}}
	                               {(N!)^2}
	                          \oint_0\ddbar{\rz}{\rz^{n}}
	                             \prod_{i_1=1}^N \int_{|w_{i_1}^{(1)}|=R_1} \ddbar{w_{i_1}^{(1)}}{}
	                             \prod_{i_2=1}^N \int_{|w_{i_2}^{(2)}|=R_2} \ddbar{w_{i_2}^{(2)}}{}
	                                \Delta\left(W^{(1)}\right)
	                                \Delta\left(W^{(2)}\right)\\
	                       &\quad \cdot
	                                \tilde f_1\left( W^{(1)} \right)
	                                \tilde f_2\left( W^{(2)} \right)
	                              \cdot \sum_{\ell_1,\ell_2=1}^N (-1)^{\ell_1+\ell_2}
	                                       \frac{ \left(w_{\ell_1}^{(1)}\right)^n e^{s_1 w_{\ell_1}^{(1)}} }
	                                            { \left(w_{\ell_2}^{(2)}\right)^{n-1} e^{s_1 w_{\ell_2}^{(2)}} }
	                                    \det\left[
	                                          C_\rz\left(w_{i_1}^{(1)},w_{i_2}^{(2)}\right)
	                                          +
	                                          D_\rz\left(w_{i_1}^{(1)},w_{i_2}^{(2)}\right)
	                                        \right]_{\substack{i_1\ne \ell_1,\\
	                                                           i_2\ne \ell_2}
                                                    }.
	\end{split}   
	\end{equation}
	Here the radii of the contours satisfy $R_1>R_2>1$. The vectors $W^{(1)}$ and $W^{(2)}$ are defined by
	\begin{equation*}
	W^{(1)}=(w_1^{(1)},\cdots,w_N^{(1)}),
	\quad
	W^{(2)}=(w_1^{(2)},\cdots,w_N^{(2)}).
	\end{equation*}
    Recall our conventions $\Delta(W)$ and $f(W)$ as in~\eqref{eq:def_Delta1} and~\eqref{eq:def_f_vector}. The functions $\tilde f_1$ and $\tilde f_2$ are defined by
	\begin{equation}
		\label{eq:tilde_f}
	\tilde f_1 (w) := w^{-N} (w+1)^{-m},\quad \tilde f_2 (w) := (w+1)^{-M+m} e^{(s_1+s_2) w}.
	\end{equation}
	The functions $C$ and $D$ appearing in the determinant are defined by
	\begin{equation}
		\label{eq:C}
	C_\rz(w_1,w_2):=   \frac{ \rz } {w_1-w_2}
	                   \frac{ w_1^{n-1} e^{s_1 w_1} }
	                        { w_2^{n-1} e^{s_1 w_2} }
	                +  \frac{ 1 } {-w_1+w_2}
	                   \frac{ w_1^{n+1} e^{s_1 w_1} }
	                        { w_2^{n+1} e^{s_1 w_2} },
	\end{equation}
	and
	\begin{equation}
		\label{eq:D}
	D_\rz(w_1,w_2):=   \frac{ \rz } {-w_1+w_2}
	                   \frac{ w_1 } { w_2 }
	                +  \frac{ 1 } {w_1-w_2}
	                   \frac{ w_1^{N} e^{(s_1+s_2) w_1} }
	                        { w_2^{N} e^{(s_1+s_2) w_2} }.
	\end{equation}
\end{prop}

\bigskip

The proof of Proposition~\ref{prop:p_step1} is provided in the next subsection~\ref{sec:proof_step1}.

\bigskip

It seems that the formula~\eqref{eq:finite_time_formula01} is not suitable for asymptotic analysis by the following two reasons. The first reason is that this formula involves some unneeded information. Note that the two terms in $D_\rz(w_1,w_2)$ have factors $(w_1/w_2)^1$ and $(w_1/w_2)^N$ whose exponents $1$ and $N$ indeed represent the bounds of the possible locations of the geodesic. However, we expect that the geodesic only fluctuates of order $N^{2/3}$ around its expected location. In other words, changing the far endpoints $1$ and $N$ will not affect the asymptotics. Therefore, $D_\rz(w_1,w_2)$ should not appear in the limit and we need to reformulate~\eqref{eq:finite_time_formula01} and remove the term $D_\rz(w_1,w_2)$. The second reason is that the formula~\eqref{eq:finite_time_formula01} contains some determinants of size $\bigO(N)$, such as the Vandermonde determinants $\Delta\left(W^{(1)}\right)$ and $\Delta\left(W^{(2)}\right)$, and the determinant $\det(C_\rz+D_\rz)$. It is typically hard to find the asymptotics of these determinants when the size $N\to\infty$. We will need to rewrite it to a formula which is more suitable for asymptotic analysis.

\bigskip

In the second step, we take the term  $D_\rz(w_1,w_2)$ away at the cost of changing the integral contours, and then evaluate the summation over $\ell_1,\ell_2$. We obtain
\begin{prop}
	\label{prop:p_step2}
	The equation~\eqref{eq:finite_time_formula01} is equivalent to
	\begin{equation}
		\label{eq:finite_time_formula02}
		\begin{split}
		&p(s_1,s_2;m,n,M,N)=\frac{1}{(N!)^2} 
	 	 \oint_0  \frac{ (1-\rz)^{N-2} \mathrm{d}\rz }
		               { 2\pi\mathrm{i} \rz^n }
		          \prod_{i_1=1}^N \left(
		                              \frac{-\rz}{1-\rz} \int_{\Sigma_{\out}}
		                                                 \ddbar{w_{i_1}^{(1)}}{}
		                             +\frac{1}{1-\rz}  \int_{\Sigma_{\inn}}
		                                               \ddbar{w_{i_1}^{(1)}}{}
		                          \right)
		          \prod_{i_2=1}^N \int_{\Sigma} \ddbar{w_{i_2}^{(2)}}{}\\
		 &\quad 
		        \hat f_1\left(W^{(1)}\right)
		        \hat f_2\left(W^{(2)}\right)
		        \frac{ \left( \Delta\left(W^{(1)}\right) \right)^2 
		 	           \left( \Delta\left(W^{(2)}\right) \right)^2
	 	             }
		             { \Delta\left(W^{(2)};W^{(1)}\right)
		             }
	             \cdot 
	             \left(
	               \hat H \left(W^{(1)};W^{(2)}\right)
	                + \rz \frac{\prod_{i_2=1}^N w_{i_2}^{(2)}}
	                           {\prod_{i_1=1}^N w_{i_1}^{(1)}}
	                      \hat H \left(W^{(2)}; W^{(1)}\right)
	             \right),
		\end{split}
	\end{equation}
where the contours $\Sigma_\out$, $\Sigma$, and $\Sigma_{\inn}$ are three nested closed contours, from outside to inside, all of which enclose both $0$ and $-1$. The vectors $W^{(1)}:=(w_1^{(1)},\cdots,w_N^{(1)})$ and $W^{(2)}:=(w_1^{(2)},\cdots,w_N^{(2)})$. The functions
\begin{equation}
	\label{eq:hat_f}
	\hat f_1(w):= (w+1)^{-m} w^{-N+n} e^{s_1 w},\quad 
	\hat f_2(w):= (w+1)^{-M+m} w^{-n} e^{s_2 w},
\end{equation}
and
\begin{equation}
	\label{eq:hat_H}
	\hat H\left(W;W'\right):= 
	   \frac{1}{2} \left(
	                   \sum_i w_i -\sum_{i'} w'_{i'}  
	               \right)^2
	   -\frac{1}{2} \left(
	                  \sum_i w_i^2 -\sum_{i'} \left(w'_{i'}\right)^2  
	                \right)
\end{equation}
for any vectors $W=(\cdots,w_i,\cdots)$ and $W'=(\cdots,w'_{i'},\cdots)$ of finite sizes.
\end{prop}

We remark that the idea of changing the integral contours is constructive. It results in a compact formula which effectively removes the terms including the information of the geodesic bounds. Formulas from similar summations (for product of two eigenfunctions in TASEP as we did in the proof of Proposition~\ref{prop:p_step1}) without including the information of the summation bounds were also obtained in the periodic version of the directed last passage percolation \cite{Baik-Liu16,Baik-Liu17,Baik-Liu21} and its large period limit \cite{Liu19}. Heuristically, in the periodic model it turned out that the upper bound (in the previous period) cancels out the lower bound (in the current period) in the summation. While in this paper, we construct contours $\Sigma_\inn$ and $\Sigma_\out$ which play similar roles as different periods: integral of the terms involving the upper bound along one contour cancels that involving the lower bound along the other contour.

The proof of Proposition~\ref{prop:p_step2} is provided in subsection~\ref{sec:proof_step2}.

\bigskip

In the last step, we rewrite the formula~\eqref{eq:finite_time_formula02} in the form with a structure similar to a Fredholm determinant expansion, which is the formula~\eqref{eq:def_finite_time_density}.

\begin{prop}
	\label{prop:p_step3}
	The formula~\eqref{eq:finite_time_formula02} is equivalent to~\eqref{eq:def_finite_time_density}.
\end{prop} 

The proof of Proposition~\ref{prop:p_step3}  is provided in subsection~\ref{sec:proof_step3}. It involves an extension of a Cauchy-type summation formula in \cite{Liu19}. We first convert the integral into discrete summations over a so-called Bethe roots, then reformulate the summation as a Fredholm-determinant-like expansion, and finally convert the discrete summation back into integrals. It would be nice to see a more direct proof for Proposition~\ref{prop:p_step3} but it seems quite complicated considering the differences between the two formulas.

\subsection{Proof of Proposition~\ref{prop:p_step1}}
\label{sec:proof_step1}

As we mentioned in the previous subsection, we need a transition probability formula by viewing the directed last passage percolation as a Markov chain. Such a formula was obtained in \cite{Johansson10} for the geometric directed last passage percolation, which is a discrete version of the model we are considering in this paper. We will introduce the model below. Then we will show how to compute an analogous probability for the geodesic in the geometric model, and take the limit to get the results for exponential directed last passage percolation.

The geometric last passage percolation model is defined as follows. We assign to each site $\mathbf{p}\in \intZ^2$ an i.i.d. geometric random variables $\tilde w(\mathbf{p})$ with parameter $q\in(0,1)$
\begin{equation}
	\prob\left(\tilde w(\mathbf{p}) = i \right) = (1-q) q^i,\qquad i=0,1,2,\cdots
\end{equation}
for each integer site $\mathbf{p}$. Note that if we take $q=1-\epsilon$ and let $\epsilon\to 0$,  $\epsilon\tilde w(\mathbf{p})$ converges to an exponential random variable.

Similar to~\eqref{eq:def_dlpp}, if a lattice point $\mathbf{q}$ lies in the upper right direction of another lattice point $\mathbf{p}$, we define the last passage time from  $\mathbf{p}$ to $\mathbf{q}$ as
\begin{equation}
	\label{eq:def_dlpp2}
	G_{\mathbf{p}}(\mathbf{q}):=\max_{\pi} \sum_{\mathbf{r}\in\pi} \tilde w(\mathbf{r}),
\end{equation}
where the maximum is over all possible up/right lattice paths from $\mathbf{p}$ to $\mathbf{q}$. We remark that the maximal path is not necessary unique in this model. We call these maximal paths the geodesics from $\mathbf{p}$ to $\mathbf{q}$.

We consider the following event
\begin{equation}
	\label{eq:2021_12_28_03}
	A= \left\{
	\begin{array}{l}
		G_{(1,1)}(m,n)+G_{(m+1,n)}(M,N) = G_{(1,1)} (M,N),\\
		G_{(1,1)}(m,n) = x,\\
		G_{(m+1,n)}(M,N) = y.
	\end{array}
	\right\}.
\end{equation}
Here $x$ and $y$ are nonnegative integers. As we mentioned before, there may be more than one geodesic. The event $A$ means that there is one geodesic that passes through the two points $(m,n)$ and $(m+1,n)$, and these two points split the last passage time $G_{(1,1)}(M,N)$ into two parts $G_{(1,1)}(m,n)=x$ and $G_{(m+1,n)}(M,N)=y$. Later we will show
\begin{lm}
	\label{lem:prob_A}
	We have
	\begin{equation}
		\label{eq:pA_formula}
		\begin{split}
			&\prob\left(A\right)\\
			&= (-1)^{N(N-1)/2}\frac{(1-q)^{MN}}{(N!)^2}
			\oint_0 	\ddbar{\rz}{\rz^n}
			\prod_{i_1=1}^N 
			\int_{|w_{i_1}^{(1)}|=R_1} \ddbar{w_{i_1}^{(1)}}{}
			\prod_{i_2=1}^N
			\int_{|w_{i_2}^{(2)}|=R_2} \ddbar{w_{i_2}^{(2)}}{}
			\Delta\left(W^{(1)}\right)
			\Delta\left(W^{(2)}\right)\\
			& \quad 
			\tilde F_1\left( W^{(1)} \right)
			\tilde F_2\left( W^{(2)} \right)
			\sum_{\ell_1,\ell_2= 1}^N
			(-1)^{\ell_1+\ell_2}
			\frac{\left(w_{\ell_1}^{(1)} + 1\right)^x
				\left(w_{\ell_1}^{(1)}\right)^n               	     
			}
			{\left(w_{\ell_2}^{(2)} + 1\right)^{x+1}
				\left(w_{\ell_2}^{(2)}\right)^{n-1}
			}
			\det\left[ \mathcal{C}_\rz(w_{i_1}^{(1)}, w_{i_2}^{(2)})
			+\mathcal{D}_\rz(w_{i_1}^{(1)}, w_{i_2}^{(2)})
			\right]_{\substack{i_1\ne \ell_1\\
					i_2\ne \ell_2}},
		\end{split}
	\end{equation}
where the radii $R_1$ and $R_2$ are distinct and both larger than $1$. The functions $\tilde F_1$ and $\tilde F_2$ are defined by
\begin{equation}
	\label{eq:tilde_F}
	\tilde F_1(w):= (w+1)^{m-1} w^{-N} (w+1-q)^{-m},\quad 
	\tilde F_2(w):= (w+1)^{x+y+M-m} (w+1-q)^{-M+m}.
\end{equation}
Recall the conventions $\tilde F_\ell\left( W^{(\ell)} \right)$ and $\Delta\left(W^{(\ell)}\right) = \prod_{i>j}\left(w_i^{(\ell)} -w_j^{(\ell)}\right)=\det\left[\left(w_i^{(\ell)}\right)^{j-1}\right]_{i,j=1}^N$ as introduced in~\eqref{eq:def_Delta1} and~\eqref{eq:def_f_vector}. Finally, the functions $\mathcal{C}_\rz$ and $\mathcal{D}_\rz$ are given by
\begin{equation}
	\mathcal{C}_\rz(w_1,w_2):= \frac{\rz}{w_1-w_2}\cdot \frac{w_1^{n-1}(w_1+1)^{x+1}}{w_2^{n-1}(w_2+1)^x}
	+\frac{1}{-w_1+ w_2}\cdot \frac{w_1^{n+1} (w_1+1)^{x}}
	{w_2^{n+1} (w_2+1)^{x-1}}
\end{equation}
and
\begin{equation}
\mathcal{D}_\rz(w_1,w_2):= \frac{\rz}{-w_1+w_2}\cdot \frac{w_1(w_2+1)}{w_2}
+\frac{1}{w_1- w_2}\cdot \frac{w_1^N (w_1+1)^{x+y+1}}
{w_2^N (w_2+1)^{x+y}}.
\end{equation}
\end{lm}

We postpone the proof of this lemma later in this subsection. Assuming Lemma~\ref{lem:prob_A}, we are ready to prove Proposition~\ref{prop:p_step1}. Below we write $A$ as $A(x,y)$ in~\eqref{eq:2021_12_28_03} to emphasize the parameters $x$ and $y$. As we mentioned before, if we take $q=1-\epsilon$ and let $\epsilon\to 0$, the geometric directed last passage percolation becomes an exponential one. More explicitly, $\epsilon \tilde w(\mathbf{p})$ converges to an exponential random variable in distribution as $\epsilon\to 0$. Moreover, for any fixed interval $I_1=[t_1,t_1+\epsilon_1]$ and $I_2=[t_2,t_2+\epsilon_2]$, we have
\begin{equation}
	\label{eq:2021_12_28_04}
	\prob\left(\bigcup_{s_1\in I_1}\bigcup_{s_2\in I_2}  A\left(\frac{s_1}{\epsilon},\frac{s_2}{\epsilon}\right)\right)
	=\prob\left\{
	\begin{array}{l}
		G_{(1,1)}(m,n)+G_{(m+1,n)}(M,N) = G_{(1,1)} (M,N),\\
		\epsilon G_{(1,1)}(m,n)\in I_1,\\
		\epsilon G_{(m+1,n)}(M,N) \in I_2.
	\end{array}
	\right\}
\end{equation}
converges as $\epsilon\to 0$ to the analogous probability that in the exponential directed last passage percolation, the geodesic $\gd_{(1,1)}(M,N)$ passes through two points $(m,n)$ and $(m+1,n)$, and the analogous last passage times satisfy $L_{(1,1)}(m,n)\in I_1$ and $L_{(m+1,n)}(M,N)\in I_2$. In other words, the limit of~\eqref{eq:2021_12_28_04} is the left hand side of~\eqref{eq:main_thm01}. We remark that although it is possible that there are more than one geodesics in the geometric last passage percolation, after taking the small $\epsilon$ limit the chance of getting more geodesics becomes zero. 

Now we evaluate the limit of~\eqref{eq:2021_12_28_04}. The left hand side of~\eqref{eq:2021_12_28_04} is
\begin{equation}
	\sum_{i\epsilon\in I_1,j\epsilon\in I_2}\prob(A(i,j))=\int_{I_1}\int_{I_2} \frac{1}{\epsilon^2}\prob\left(A\left(\frac{s_1}{\epsilon},\frac{s_2}{\epsilon}\right)\right)\mathrm{d}\mu_\epsilon(s_2)\mathrm{d}\mu_\epsilon(s_1),
\end{equation}
where $\mathrm{d}\mu_\epsilon(s)=\epsilon\delta_{\frac{s}{\epsilon}\in\intZ}$. We will prove
\begin{equation}
	\label{eq:gDLPP_to_eDLPP}
	\lim_{\epsilon\to 0}\frac{1}{\epsilon^2}\prob(A(s_1/\epsilon,s_2/\epsilon))= p(s_1,s_2;m,n,M,N)
\end{equation}
uniformly on $I_1\times I_2$, with $p(s_1,s_2;m,n,M,N)$ defined in~\eqref{eq:finite_time_formula01}. Then by using the continuity of the function $p(s_1,s_2;m,n,M,N)$ we immediately obtain that the limit of~\eqref{eq:2021_12_28_04} equals to $\int_{I_1}\int_{I_2}p(s_1,s_2;m,n,M,N)\mathrm{d}s_2\mathrm{d}s_1$. Hence we prove Proposition~\ref{prop:p_step1}.

Now we prove~\eqref{eq:gDLPP_to_eDLPP}. We insert $q=1-\epsilon$, $x=s_1/\epsilon$, and $y=s_2/\epsilon$ in~\eqref{eq:pA_formula}. Note that all other parameters are fixed, and $s_1\in I_1,s_2\in I_2$ are nonnegative. We observe that the exponents of $(w_{i_1}^{(1)}+1)$ for each $1\le i_1\le N$ in the integrand are at least $m-1 +\min\{x,1\}\ge m-1\ge 0$, and the exponents of $(w_{i_2}^{(2)}+1)$ for each $1\le i_2\le N$ are at least $x+y+M-m-\max\{x+1,x+y\}\ge M-m-1\ge 0$. Therefore the integrand is analytic at $-1$ for each $w_{i_1}^{(1)}$ and $w_{i_2}^{(2)}$. There are possible poles at $0$ and $q-1=-\epsilon$ both of which are close to $0$ as $\epsilon\to 0$. We hence can deform the contours sufficiently close to the origin. More precisely, we replace $R_1$ and $R_2$ by $\epsilon \hat R_1$ and $\epsilon \hat R_2$ where $\hat R_1,\hat R_2$ are distinct constants and both larger than $1$, and change variables $w_{i_1}^{(1)} =\epsilon \hat w_{i_1}^{(1)}$ and $w_{i_2}^{(2)} =\epsilon \hat w_{i_2}^{(2)}$. Then
\begin{equation}
	\label{eq:2021_12_29_01}
	\begin{split}
	&\Delta\left(W^{(1)}\right)\Delta\left(W^{(2)}\right) = \epsilon^{N(N-1)} \Delta\left(\hat W^{(1)}\right)\Delta\left(\hat W^{(2)}\right),\\
	&\tilde F_1(w) = \epsilon^{-N-m}(\hat w^{-N} (\hat w+1)^{-m}+\bigO(\epsilon))
	              =\epsilon^{-N-m} (\tilde f_1(\hat w)+\bigO(\epsilon)),\\
	&\tilde F_2(w) = \epsilon^{-M+m} ((\hat w +1)^{-M+m} e^{(s_1+s_2)\hat w}+\bigO(\epsilon))
	              =\epsilon^{-M+m} (\tilde f_2(\hat w)+\bigO(\epsilon)),\\
	& \left(w+1\right)^x w^n = 
	             \epsilon^n (\hat w^n e^{s_1\hat w}+\bigO(\epsilon)),
	  \quad \left(w+1\right)^x w^{n-1} =
	             \epsilon^{n-1} (\hat w^{n-1} e^{s_1\hat w}+\bigO(\epsilon)),\\	
	&\mathcal{C}_\rz(w_1,w_2)= 
	       \frac{\rz}{\epsilon(\hat w_1-\hat w_2)}\cdot     
	       \frac{\hat w_1^{n-1}e^{s_1\hat w_1}}{\hat w_2^{n-1}e^{s_1\hat w_2}}
	      +\frac{1}{\epsilon(-\hat w_1+ \hat w_2)}\cdot 
	       \frac{\hat w_1^{n+1}e^{s_1 \hat w_1}}{\hat w_2^{n+1} e^{s_1\hat w_2}}+\bigO(1)
	       =\epsilon^{-1}(C_\rz(\hat w_1,\hat w_2)+\bigO(\epsilon)),\\
	&\mathcal{D}_{\rz}(w_1,w_2)= 
	      \frac{\rz}{\epsilon(-\hat w_1+\hat w_2)}\cdot     
	      \frac{\hat w_1 }{\hat w_2}
         +\frac{1}{\epsilon(\hat w_1- \hat w_2)}\cdot 
	      \frac{\hat w_1^{N}e^{(s_1+s_2)\hat w_1}}
	           {\hat w_2^{N} e^{(s_1+s_2)\hat w_2}}+\bigO(1)
	    =\epsilon^{-1}(D_\rz(\hat w_1,\hat w_2)+\bigO(\epsilon)).
	\end{split}
\end{equation}
We remind that $\mathrm{d}w=\epsilon \mathrm{d}\hat w$. Therefore by inserting these leading terms, we heuristically obtain that
\begin{equation}
	\label{eq:2021_12_28_05}
	\lim_{\epsilon\to 0}\frac{1}{\epsilon^2}\prob(A(s_1/\epsilon,s_2/\epsilon)) = \text{ the right hand side of~\eqref{eq:finite_time_formula01}}.
\end{equation}
On the other hand, since all other parameters are fixed and the contours $|\hat w_{i_1}^{(1)}|=\hat R_1$ and $|\hat w_{i_2}^{(2)}|=\hat R_2$ are of finite size, if we insert the above estimates~\eqref{eq:2021_12_29_01} with the error terms into~\eqref{eq:pA_formula}, all the terms involving $\bigO(\epsilon)$ are uniformly bounded by $C\epsilon$ for some constant $C$, and there are only finitely many such terms. Therefore the equation~\eqref{eq:2021_12_28_05} holds uniformly. This proves~\eqref{eq:gDLPP_to_eDLPP}.
\bigskip

The remaining part of this subsection is to prove Lemma~\ref{lem:prob_A}.

Denote
\begin{equation}
	G(m) = \left(G_{(1,1)}(m,1),\cdots,G_{(1,1)}(m,N)\right)
\end{equation}
the vector of the last passage times from the site $(1,1)$ to $(m,i)$, $1\le i\le N$.

Our starting point is the following remarkable formula for the distribution of $G(m)$. 

\begin{thm}\cite[Theorem 2.1]{Johansson10}
	\label{thm:Joh}
	Suppose $X=(x_1,\cdots,x_N)\in\intZ_{\ge 0}^N$ satisfies $x_1\le x_2\le\cdots\le x_N$, then 
	\begin{equation*}
		\prob\left(G(m)=X\right)=\det\left[(1-q)^m\int_{|w|=R}(w+1)^{x_j+m-1}w^{j-i}\ddbar{w}{(w+1-q)^m}\right]_{i,j=1}^N,
	\end{equation*}
	where $R>1$ is any constant.
\end{thm}

Note that the contour is of radius $R>1$ in the above theorem. This restriction will be kept throughout the proof of Lemma~\ref{lem:prob_A} and finally lead to the requirements $R_1>1$ and $R_2>1$.

The original theorem of \cite[Theorem 2.1]{Johansson10} considered the finite-step transition probabilities from any column to another, and for any $x_1\le \cdots\le x_N$ without assuming $x_1\ge 0$. For our purpose we only need this simpler version. The assumption that $x_1\ge 0$ comes from the fact that all random variables $\tilde w(\mathbf{p})$ are nonnegative. Moreover, we use the contour integral formula in the above determinant for later computations. This formula is equivalent to the original version by combining the equations (9) and (25) in \cite{Johansson10}.

Denote
\begin{equation*}
	\tilde G(m+1) = \left( G_{(m+1,1)}(M,N),\cdots,G_{(m+1,N)}(M,N)\right).
\end{equation*}
Note that, by flipping the sites $(i,j)\to (-i,-j)$ and shifting the site $(-M,-N)$ to $(1,1)$, $\tilde G(m+1)$ has the same distribution as 
\begin{equation*}
	\left(G_{(1,1)}(M-m,N),G_{(1,1)}(M-m,N-1),\cdots,G_{(1,1)}(M-m,1)\right).
\end{equation*}
Therefore, by applying Theorem~\ref{thm:Joh} we have
		\begin{equation*}
		\prob\left( \tilde G(m+1)=Y
		     \right)
		 =  \det
		      \left[
		           (1-q)^{M-m} \int_{|w|=R}
		                       (w+1)^{y_{N+1-j}+M-m-1}w^{j-i}
		                       \ddbar{w}{(w+1-q)^{M-m}}
		      \right]_{i,j=1}^N
	\end{equation*}
for any $Y=(y_1,\cdots,y_N)\in\intZ^N$ satisfying $y_1\ge y_2\ge\cdots\ge y_N\ge 0$.

Note that $G(M)$ and $\tilde G(m+1)$ are independent since they are defined on the lattices $\intZ_{\le m}\times\intZ$ and $\intZ_{\ge m+1}\times \intZ$ respectively. Also note the event $A$ is equivalent to the event that $G_{(1,1)}(m,n)=x$, $G_{(m+1,n)(M,N)}=y$, and $G_{(1,1)}(m,i)+G_{(m+1,i)}(M,N)\le G_{(1,1)}(M,N)=x+y$ for all other $i$'s. Thus by combining Theorem~\ref{thm:Joh} and the above formula for $\tilde G(m+1)$, we obtain
\begin{equation}
	\label{eq:PA_01}
	\begin{split}
      \prob(A)
		&= \sum \prob\left(G(m) =X \right) \prob\left( \tilde G(m+1) = Y\right)\\
		&= (1-q)^{MN} \sum   \det 
		                      \left[
		                         \int_{|w|=R}
		                           (w+1)^{x_j+m-1}w^{j-i}(w+1-q)^{-m}\ddbar{w}{}
		                      \right]_{i,j=1}^N\\
		                      &\qquad\qquad\qquad \quad \cdot
		                      \det 
		                      \left[
		                      \int_{|w|=R}
		                      (w+1)^{y_{N+1-j}+M-m-1}w^{j-i}(w+1-q)^{-M+m}\ddbar{w}{}
		                      \right]_{i,j=1}^N,
	\end{split}
\end{equation}
where the summation is running over all possible $X=(x_1,\cdots,x_N)\in\intZ^N$ and $Y=(y_1,\cdots,y_N)\in\intZ^N$ satisfying
\begin{equation}
	\label{eq:XY_bounds}
	\begin{split}
	&0\le x_1\le\cdots\le x_N, \quad	y_1\ge\cdots\ge y_N\ge 0,\\			
	&x_i+y_i \le x+y, \quad \text{for all } i=1,\cdots,N,\\
	&\text{and }		x_n=x,\ y_n=y.
   \end{split}
\end{equation}

We will consider the above summation in two steps. First, we fix $X$ satisfying $0\le x_1\le\cdots\le x_N \le x+y$ and $x_n=x$, and take the sum over $Y$ satisfying~\eqref{eq:XY_bounds}. Note that only the last determinant in~\eqref{eq:PA_01} contains $Y$. We formulate such a summation in the following lemma.

\begin{lm}
	\label{lem:sum_Y}
	Suppose $0\le x_1\le\cdots\le x_N\le x+y$, $x_n=x$ and $y_n=y$. Assume that $F(w)$ is a function which is analytic on $|w|\ge R$ and satisfies $|F(w)|\to 0$ uniformly as $|w|\to\infty$. Then
	\begin{equation*}
		\begin{split}
		&\sum_{y_N=0}^{x+y-x_N}
		 \sum_{y_{N-1}=y_N}^{x+y-x_{N-1}}
		 \cdots 
		 \sum_{y_{n+1}=y_{n+2}}^{x+y-x_{n+1}}
		 \sum_{y_{n-1}=y}^{x+y-x_{n-1}}
		 \cdots
		 \sum_{y_1=y_2}^{x+y-x_1}
		 \det\left[ \int_{|w|=R} 
		            (w+1)^{y_j}w^{-j+i}F(w)
		            \ddbar{w}{}
		     \right]_{i,j=1}^N\\
		&= \det\left[
		          \int_{|w|=R}
		                        (w+1)^{x+y-x_j+1_{j\ne n}}w^{-j+i-1_{j\ne n}}F(w)
		                    \ddbar{w}{}
		       \right]_{i,j=1}^N.
		\end{split}
	\end{equation*}
\end{lm} 
\begin{proof}[Proof of Lemma~\ref{lem:sum_Y}]
	Due to the linearity of determinant, we can take the summation of the columns inside the determinant. For each $j=1,\cdots,n-1,n+1,\cdots,N-1$, we have
	\begin{equation*}
		\begin{split}
		&\sum_{y_{j}=y_{j+1}}^{x+y-x_j} 
	  	  \int_{|w|=R} 
		  (w+1)^{y_j}w^{-j+i}F(w)
		  \ddbar{w}{}\\
		  &
		 =\int_{|w|=R} 
		  (w+1)^{x+y-x_j+1}w^{-j-1+i}F(w)
		  \ddbar{w}{}
		  -\int_{|w|=R} 
		  (w+1)^{y_{j+1}}w^{-j-1+i}F(w)
		  \ddbar{w}{},
		\end{split}
	\end{equation*}
where the second term matches the corresponding entry in the $(j+1)$-th column. Therefore we can remove this term without changing the determinant. For the summation over $y_N$, we have a similar identity where the second term becomes
\begin{equation*}
	\int_{|w|=R} 
	w^{-N-1+i}F(w)
	\ddbar{w}{}=0
\end{equation*}
by deforming the contour to infinity. We complete the proof by combining the above summations.
\end{proof}

\bigskip
Now we come back to~\eqref{eq:PA_01}. We reorder the rows and columns in the second determinant by replacing $i\to N+1-i$ and $j\to N+1-j$, and apply Lemma~\ref{lem:sum_Y} with $F(w) = (w+1)^{M-m-1}(w+1-q)^{-M+m}$. We have
\begin{equation}
	\label{eq:PA_02}
	\begin{split}
		\prob(A)
		&= (1-q)^{MN} \sum   \det 
		\left[
		\int_{|w|=R}
		(w+1)^{x_j+m-1}w^{j-i}(w+1-q)^{-m}\ddbar{w}{}
		\right]_{i,j=1}^N\\
		&\qquad\qquad\qquad \quad \cdot
		 \det\left[
	           \int_{|w|=R}
		          (w+1)^{-x_j+x+y+M-m-1_{j=N}}w^{-j-1+i+1_{j=n}}(w+1-q)^{-M+m}
		       \ddbar{w}{}
		     \right]_{i,j=1}^N,
	\end{split}
\end{equation}
where the summation is over all $0\le x_1\le\cdots\le x_N\le x+y$ with $x_n=x$.
\bigskip

In the next step, we consider the sum over $X$ in~\eqref{eq:PA_02}. We first apply the following Cauchy-Binet/Andreief’s formula in~\eqref{eq:PA_02}
\begin{equation*}
	\det  \left[
	           \int f_i(z)g_j(z) \mathrm{d}\mu(z)
	      \right]_{i,j=1}^N
	=\frac{1}{N!}
	    \int\cdots\int 
	       \det  \left[ f_i(z_j) \right]_{i,j=1}^N
	       \det  \left[ g_i(z_j) \right]_{i,j=1}^N
	       \mathrm{d}\mu(z_1)\cdots \mathrm{d}\mu(z_N).
\end{equation*}
We also relabel the variables to avoid confusions. Recall the functions $\tilde F_1$ and $\tilde F_2$ defined in~\eqref{eq:tilde_F}. We have
\begin{equation*}
\begin{split}
	&\det 
	\left[
	\int_{|w|=R}
	(w+1)^{x_j+m-1}w^{j-i}(w+1-q)^{-m}\ddbar{w}{}
	\right]_{i,j=1}^N\\
	&=\frac{1}{N!}\prod_{i=1}^N 
	\int_{|w_{i}^{(1)}|=R_1} \ddbar{w_{i}^{(1)}}{}
	\tilde F_1\left( W^{(1)} \right) \det\left[ 
	\left( w_i^{(1)}+1\right)^{x_j}
	\left(w_i^{(1)}\right)^j
	\right]_{i,j=1}^N\det\left[\left(w_i^{(1)}\right)^{N-j}\right]_{i,j=1}^N
\end{split}
\end{equation*}
and
\begin{equation*}
	\begin{split}
		&\det\left[
		\int_{|w|=R}
		(w+1)^{-x_j+x+y+M-m-1_{j=N}}w^{-j-1+i+1_{j=n}}(w+1-q)^{-M+m}
		\ddbar{w}{}
		\right]_{i,j=1}^N\\
		&=\frac{1}{N!}\prod_{i=1}^N 
		\int_{|w_{i}^{(2)}|=R_2} \ddbar{w_{i}^{(2)}}{}
		\tilde F_2\left( W^{(2)} \right) \det\left[ 
		\left( w_i^{(1)}+1\right)^{-x_j-1_{j=N}}
		\left(w_i^{(1)}\right)^{-j+1_{j=N}}
		\right]_{i,j=1}^N\det\left[\left(w_i^{(2)}\right)^{j-1}\right]_{i,j=1}^N.
	\end{split}
\end{equation*}
Thus we write
\begin{equation}
	\label{eq:PA_03}
	\begin{split}
		\prob(A) & =(-1)^{N(N-1)/2} \frac{(1-q)^{MN}}{(N!)^2}
		              \prod_{i=1}^N 
		                    \int_{|w_{i}^{(1)}|=R_1} \ddbar{w_{i}^{(1)}}{}		              
		                    \int_{|w_{i}^{(2)}|=R_2} \ddbar{w_{i}^{(2)}}{}
		              \tilde F_1\left( W^{(1)} \right)
		              \tilde F_2\left( W^{(2)} \right)\\
		         & \quad \cdot \Delta\left(W^{(1)}\right)
		                       \Delta\left(W^{(2)}\right)
		                \cdot S\left(W^{(1)}, W^{(2)}\right),
	\end{split}
\end{equation}
where $W^{(1)}=\left(w_1^{(1)},\cdots,w_N^{(1)}\right)$, $W^{(2)}=\left(w_1^{(2)},\cdots,w_N^{(2)}\right)$.  We also rewrote $\det\left[\left(w_{i}^{(\ell)}\right)^{j-1}\right]_{i,j=1}^N = \Delta\left(W^{(\ell)}\right)$ for both $\ell=1,2$. Finally, the function
\begin{equation}
	\label{eq:SW}
	S\left(W^{(1)};W^{(2)}\right)
	:=\sum_{X} 
	     \det\left[ 
               	\left( w_i^{(1)}+1\right)^{x_j}
             	\left(w_i^{(1)}\right)^j
	         \right]_{i,j=1}^N
	     \det\left[ 
	            \left( w_i^{(2)}+1\right)^{-x_j-1_{j=n}}
	            \left(w_i^{(2)}\right)^{-j+1_{j=n}}
	         \right]_{i,j=1}^N,
\end{equation}
where the summation is over all $0\le x_1\le\cdots\le x_N\le x+y$ with fixed $x_n=x$.

Note that the summation over $X$ only appears in the function $S\left(W^{(1)};W^{(2)}\right)$. Our goal in this step is to evaluate this summation explicitly. We remark that this summation without the extra $1_{j=n}$ in the exponents can be simplified to a compact formula if all the coordinates of $W^{(\ell)}$ satisfy a so-called Bethe equation, see \cite[Proposition 5.2]{Baik-Liu17}. However, here we do not have the Bethe roots structure for the coordinates and the resulting formulas are more complicated.

To proceed, we need an identity to expand the determinants in~\eqref{eq:SW}. By using the Laplace expansion of the determinant along the $n$-th column and the Cauchy-Binet formula for the cofactors, we have the identity
\begin{equation*}
	\det\left[A_{i,j}\right]_{i,j=1}^N 
	  = \sum_{\ell}(-1)^{\ell+n}A_{\ell,n}
	    \sum_{
	    	  \substack{ I_1\cup I_2=\{1,\cdots,N\}\setminus\{\ell\}\\
	    	  	         |I_1|=n-1, |I_2|=N-n
	    	           }
    	     }
         (-1)^{\#(I,J)}
         \det\left[A_{i,j}\right]_{\substack{i\in I_1 \\ 1\le j\le n-1}}
         \det\left[A_{i,j}\right]_{\substack{i\in I_2 \\ n+1\le j\le N}},
\end{equation*}
where
\begin{equation}
	\#(I_1,I_2):=\text{ the number of pairs $(i_1,i_2)\in I_1\times I_2$ such that $i_1>i_2$}.
\end{equation}
We apply the above identity in~\eqref{eq:SW} and change the order of summations. This leads to
\begin{equation}
	\label{eq:SW2}
	\begin{split}
		&S\left(W^{(1)};W^{(2)}\right)\\
		&=\sum_{\ell_1,\ell_2\ge1} (-1)^{\ell_1+\ell_2}
		               \frac{\left(w_{\ell_1}^{(1)} + 1\right)^x
		               	     \left(w_{\ell_1}^{(1)}\right)^n               	     
		                    }
		                    {\left(w_{\ell_2}^{(2)} + 1\right)^{x+1}
		                     \left(w_{\ell_2}^{(2)}\right)^{n-1}
		                    }
	       \sum_{\substack{
	     		                 I_1^{(1)} \cup I_2^{(1)}
	     		                  = \{1,\cdots,N\} \setminus\{\ell_1 \}\\
	     		                 I_1^{(2)} \cup I_2^{(2)}
	     		                  = \{1,\cdots,N\} \setminus\{\ell_2 \}\\
	     		                 |I_1^{(1)}| = |I_1^{(2)}| =n-1\\
	     		                 |I_2^{(1)}| = |I_2^{(2)}| =N-n
	     	                    }
	                 }
              (-1)^{\#(I_1^{(1)},I_2^{(1)}) +\#(I_1^{(2)},I_2^{(2)})}\\
&\quad           \frac{\prod_{i\in I_2^{(1)}} \left(w_i^{(1)}\right)^n}
                      {\prod_{i\in I_2^{(2)}} \left(w_i^{(2)}\right)^n}
                 \hat S_{0,x}\left( W^{(1)}_{I_1^{(1)}}, W^{(2)}_{I_1^{(2)}}\right)
                 \hat S_{x,x+y} \left( W^{(1)}_{I_2^{(1)}}, W^{(2)}_{I_2^{(2)}} \right),
	\end{split}
\end{equation}
where for simplification we use the notation $W_I$ for the vector with coordinates $w_i$'s satisfying $i\in I$. More explicitly, $W_I=(w_{i_1},w_{i_2},\cdots,w_{i_k})$ for any $I=(i_1,\cdots,i_k)$. The function
\begin{equation*}
	\hat S_{a,b} \left(W,W'\right)
	  := \sum_{a\le x_1\le\cdots\le x_k\le b} 
	        \det\left[ (w_i+1)^{x_j} w_i^j\right]_{1\le i,j\le k}
	        \det\left[ (w'_i+1)^{-x_j} (w'_i)^{-j}\right]_{1\le i,j\le k}
\end{equation*}
for any $a\le b$ and vectors $W$ and $W'$ of the same size. Here $k$ is the size of $W$ and $W'$, $w_i$'s and $w'_i$'s are the coordinates of $W$ and $W'$ respectively.

We have the following identity to simplify $\hat S_{a,b} \left(W,W'\right)$.
\begin{lm}\cite{Baik-Liu17}
	\label{lem:double_sum}
	We have
	\begin{equation}
		\label{eq:S_ab}
		\hat S_{a,b}(W,W')
		=\det\left[
		          \frac{1}{-w_i+w'_{i'}}
		          \cdot
		          \frac{w_i (w_i+1)^a} {w'_{i'} (w'_{i'}+1)^{a-1}}
		          +
		          \frac{1}{w_i-w'_{i'}}
		          \cdot
		          \frac{w_i^k (w_i+1)^{b+1}} {(w'_{i'})^k (w'_{i'} +1)^b}
		     \right]_{i,i'=1}^k.
	\end{equation}
\end{lm}
\begin{proof}[Proof of Lemma~\ref{lem:double_sum}]
	The main technical part of the summation was included in \cite{Baik-Liu17}. Here we simply mention how to arrive~\eqref{eq:S_ab} using the known results in \cite{Baik-Liu17}.
	
	In \cite{Baik-Liu17}, the authors introduced a similar sum $H_a(W;W')$, where $W$ and $W'$ both are of size $N$. See equation (5.6) in \cite{Baik-Liu17}. It reads
	\begin{equation*}
		H_a(W;W') = \sum_{a-1= x_1\le\cdots\le x_N\le a+L-N-1} 
		               \det\left[ (w'_i+1)^{x_j} (w'_i)^j\right]_{1\le i,j\le N}
		               \det\left[ (w_i+1)^{-x_j} w_i^{-j}\right]_{1\le i,j\le N}.
	\end{equation*}
    Here we emphasize that $x_1=a-1$ is fixed in this summation. We also remark that the original definition of $H_a(W;W')$ assumes that the coordinates of $W$ and $W'$ are  roots of the so-called Bethe equation, but we will only cite the identities in \S5.1-5.3 in \cite{Baik-Liu17} where the Bethe roots properties are not used. 
	
	The equation (5.44) in \cite{Baik-Liu17} can be viewed as a difference of two terms. We apply Lemma 5.9 of \cite{Baik-Liu17} for each term and rewrite the equation as
	\begin{equation*}
		\begin{split}
			H_a(W,W')&= \det\left[
			                     \frac{1}{w_i-w'_{i'}}
			                     \cdot
			                     \frac{w'_{i'} (w'_{i'}+1)^{a-1}}{w_i (w_i+1)^{a-2}}
			                    +\frac{1}{-w_i+w'_{i'}}
			                     \cdot
			                     \frac{(w'_{i'})^N (w'_{i'}+1)^{a+L-N}}
			                          {w_i^N (w_i+1)^{a+L-N-1}}
			                \right]_{i,i'=1}^N\\
			      &\quad -\det\left[
			                 \frac{1}{w_i-w'_{i'}}
			                 \cdot
			                 \frac{w'_{i'} (w'_{i'}+1)^{a}}{w_i (w_i+1)^{a-1}}
			                 +\frac{1}{-w_i+w'_{i'}}
			                 \cdot
			                 \frac{(w'_{i'})^N (w'_{i'}+1)^{a+L-N}}
			                      {w_i^N (w_i+1)^{a+L-N-1}}
			                \right]_{i,i'=1}^N.
		\end{split}
	\end{equation*}
We replace $a+L-N-1$ by $b$, and then $a-1$ by $a$, and get
\begin{equation*}
	\begin{split}
		&\sum_{a= x_1\le\cdots\le x_N\le b} 
		   \det\left[ (w'_i+1)^{x_j} (w'_i)^j\right]_{1\le i,j\le N}
		   \det\left[ (w_i+1)^{-x_j} w_i^{-j}\right]_{1\le i,j\le N}\\
		&= \det\left[
		\frac{1}{w_i-w'_{i'}}
		\cdot
		\frac{w'_{i'} (w'_{i'}+1)^{a}}{w_i (w_i+1)^{a-1}}
		+\frac{1}{-w_i+w'_{i'}}
		\cdot
		\frac{(w'_{i'})^N (w'_{i'}+1)^{b+1}}
		{w_i^N (w_i+1)^{b}}
		\right]_{i,i'=1}^N\\
		&\quad -\det\left[
		\frac{1}{w_i-w'_{i'}}
		\cdot
		\frac{w'_{i'} (w'_{i'}+1)^{a+1}}{w_i (w_i+1)^{a}}
		+\frac{1}{-w_i+w'_{i'}}
		\cdot
		\frac{(w'_{i'})^N (w'_{i'}+1)^{b+1}}
		{w_i^N (w_i+1)^{b}}
		\right]_{i,i'=1}^N.  
	\end{split}
\end{equation*}
So far $x_1=a$ is fixed. Now by summing the above identity for all $x_1$ from $a$ to $b$, we get
\begin{equation*}
	\begin{split}
		\hat S_{a,b}(W',W)
		&=\det\left[
		        \frac{1}{w_i-w'_{i'}}
		        \cdot
	         	\frac{w'_{i'} (w'_{i'}+1)^{a}}{w_i (w_i+1)^{a-1}}
	          + \frac{1}{-w_i+w'_{i'}}
		        \cdot
		        \frac{(w'_{i'})^N (w'_{i'}+1)^{b+1}}
		             {w_i^N (w_i+1)^{b}}
		      \right]_{i,i'=1}^N\\
		 &\quad -\det\left[
		                 \frac{1}{w_i-w'_{i'}}
		                 \cdot
		                 \frac{w'_{i'} (w'_{i'}+1)^{b+1}}{w_i (w_i+1)^{b}}
		               + \frac{1}{-w_i+w'_{i'}}
		                 \cdot
		                 \frac{(w'_{i'})^N (w'_{i'}+1)^{b+1}}
	                          {w_i^N (w_i+1)^{b}}
		             \right]_{i,i'=1}^N.
	\end{split}
\end{equation*}
It is easy to see that the second determinant is zero. 
Therefore we obtain a formula for $\hat S_{a,b}(W',W)$ with a single determinant. By switching $W$ and $W'$, and replace the size $N$ by $k$, we obtain~\eqref{eq:S_ab}.
\end{proof}

Now we apply Lemma~\ref{lem:double_sum} to~\eqref{eq:SW2}. We also use the identity
\begin{equation*}
	\begin{split}
	&\sum_{\substack{I_1\cup J_1 = \{1,\cdots,L\}\\
	                 J_1\cup J_2 = \{1,\cdots,L\}\\
                     |I_1|=|I_2|=n-1\\
                     |J_1|=|J_2|=L-n+1}}
                 (-1)^{\#(I_1,J_1)+\#(I_2,J_2)}
                 \det\left[A(i_1,j_1)\right]_{\substack{i_1\in I_1\\ j_1\in J_1}}
                 \det\left[B(i_2,j_2)\right]_{\substack{i_2\in I_2\\ j_2\in J_2}}\\
    &=\oint_0 \det\left[\rz A(i,j) +B(i,j)\right]_{i,j=1}^L \ddbar{\rz}{\rz^n},
    \end{split}
\end{equation*}
which follows from the multilinearity of the determinant on the rows and the Cauchy-Binet formula. It can also be derived from Lemma 5.9 of \cite{Baik-Liu17}. Then we arrive at
\begin{equation}
	\label{eq:SW3}
	\begin{split}
		&S\left(W^{(1)};W^{(2)}\right)
		=\sum_{\ell_1,\ell_2\ge1} (-1)^{\ell_1+\ell_2}
		\frac{\left(w_{\ell_1}^{(1)} + 1\right)^x
			\left(w_{\ell_1}^{(1)}\right)^n               	     
		}
		{\left(w_{\ell_2}^{(2)} + 1\right)^{x+1}
			\left(w_{\ell_2}^{(2)}\right)^{n-1}
		}\\
	    &\oint_0 \frac{\mathrm{dz}}{2\pi\mathrm{i}\mathrm{z}^n}
	     \det\left[
	               \frac{\rz}{-w_{i_1}^{(1)}+w_{i_2}^{(2)}}
	               \cdot
	               \frac{w_{i_1}^{(1)}\left(w_{i_2}^{(2)}+1\right)}
	                    {w_{i_2}^{(2)}}
	              +\frac{\rz}{w_{i_1}^{(1)}-w_{i_2}^{(2)}}
	               \cdot
	               \frac{\left(w_{i_1}^{(1)}\right)^{n-1}
	               	     \left(w_{i_1}^{(1)}+1\right)^{x+1}
               	        }
	                    {\left(w_{i_2}^{(2)}\right)^{n-1}
	                     \left(w_{i_2}^{(2)}+1\right)^{x}
                        }\right.\\
        &\qquad\quad\left.
                  +\frac{1}{-w_{i_1}^{(1)}+w_{i_2}^{(2)}}
                   \cdot
                   \frac{\left(w_{i_1}^{(1)}\right)^{n+1}
                   	     \left(w_{i_1}^{(1)}+1\right)^x
                        }
                        {\left(w_{i_2}^{(2)}\right)^{n+1}
                         \left(w_{i_2}^{(2)}+1\right)^{x-1}
                        }
                  +\frac{1}{w_{i_1}^{(1)}-w_{i_2}^{(2)}}
                   \cdot
                   \frac{\left(w_{i_1}^{(1)}\right)^{N}
                  	     \left(w_{i_1}^{(1)}+1\right)^{x+y+1}
                        }
                        {\left(w_{i_2}^{(2)}\right)^{N}
                  	     \left(w_{i_2}^{(2)}+1\right)^{x+y}
                        }
	         \right]_{\substack{i_1\ne \ell_1\\
	                            i_2\ne \ell_2}}.
	\end{split}
\end{equation}
By inserting this formula to~\eqref{eq:PA_03}, we obtain Lemma~\ref{lem:prob_A}.

\subsection{Proof of Proposition~\ref{prop:p_step2}}
\label{sec:proof_step2}

In this subsection, we prove Proposition~\ref{prop:p_step2}. There are two main steps in the proof. In the first step we will deform the contours and get rid of term $D_\rz$ in~\eqref{eq:finite_time_formula01}. In the second step we will evaluate the summation over $\ell_1$ and $\ell_2$.

\subsubsection{Step 1: Deforming the contours}

We first realize that
\begin{equation*}
	C_\rz(w_1,w_2)+D_\rz(w_1,w_2)
	  = 
        \frac{ \rz } {w_1-w_2}
	    \frac{ w_1^{n-1} e^{s_1 w_1} }
             { w_2^{n-1} e^{s_1 w_2} }
	 +  \frac{ 1 } {-w_1+w_2}
	    \frac{ w_1^{n+1} e^{s_1 w_1} }
	         { w_2^{n+1} e^{s_1 w_2} }
	 +  \frac{ \rz } {-w_1+w_2}
        \frac{ w_1 } { w_2 }
     +  \frac{ 1 } {w_1-w_2}
        \frac{ w_1^{N} e^{(s_1+s_2) w_1} }
             { w_2^{N} e^{(s_1+s_2) w_2} }
\end{equation*}
does not have a pole at $w_1=w_2$. Hence the integrand in~\eqref{eq:finite_time_formula01} only has poles at $0$ and $-1$. Furthermore, we can rewrite the $w_{i_1}^{(1)}$ integrals as
\begin{equation}
	\label{eq:contour_repalce01}
	\int_{|w_{i_1}^{(1)}|=R_1} \ddbar{w_{i_1}^{(1)}}{}
	=\frac{-\rz}{1-\rz} \int_{\Sigma_{\out}}
	  \ddbar{w_{i_1}^{(1)}}{}
	+\frac{1}{1-\rz}  \int_{\Sigma_{\inn}}
	  \ddbar{w_{i_1}^{(1)}}{}
\end{equation}
and the $w_{i_2}^{(2)}$ integrals as
\begin{equation}
		\label{eq:contour_repalce02}
	\int_{|w_{i_2}^{(2)}|=R_2} \ddbar{w_{i_2}^{(2)}}{}=\int_{\Sigma} \ddbar{w_{i_2}^{(2)}}{}
\end{equation}
without changing the value of~\eqref{eq:finite_time_formula01}. After we change the order of summation and integrals, we have
	\begin{equation}
	\label{eq:finite_time_formula03}
	\begin{split}
		& p(s_1,s_2;m,n,M,N) \\
		& = \frac{(-1)^{N(N-1)/2}}
		{(N!)^2}
		\oint_0\ddbar{\rz}{\rz^{n}}
		\sum_{\ell_1,\ell_2=1}^N (-1)^{\ell_1+\ell_2}
		\left( \frac{-\rz}{1-\rz} \int_{\Sigma_{\out}}
		                          \ddbar{w_{\ell_1}^{(1)}}{}
		      +\frac{1}{1-\rz}  \int_{\Sigma_{\inn}}
		                          \ddbar{w_{\ell_1}^{(1)}}{}
		\right)
		\int_{\Sigma} \ddbar{w_{\ell_2}^{(2)}}{}
		   \frac{ \left(w_{\ell_1}^{(1)}\right)^n e^{s_1 w_{\ell_1}^{(1)}} }
		        { \left(w_{\ell_2}^{(2)}\right)^{n-1} e^{s_1 w_{\ell_2}^{(2)}} }\\
		&
		\prod_{i_1\ne \ell_1} \left( \frac{-\rz}{1-\rz} \int_{\Sigma_{\out}}
		                                                \ddbar{w_{i_1}^{(1)}}{}
	                             	+\frac{1}{1-\rz}  \int_{\Sigma_{\inn}}
		                                              \ddbar{w_{i_1}^{(1)}}{}
		                      \right)
		\prod_{i_2\ne \ell_2} \int_{\Sigma} \ddbar{w_{i_2}^{(2)}}{}
		\\
		&\quad \cdot
		\Delta\left(W^{(1)}\right)
		\Delta\left(W^{(2)}\right)
		\tilde f_1\left( W^{(1)} \right)
		\tilde f_2\left( W^{(2)} \right)
		\det\left[
		C_\rz\left(w_{i_1}^{(1)},w_{i_2}^{(2)}\right)
		+
		D_\rz\left(w_{i_1}^{(1)},w_{i_2}^{(2)}\right)
		\right]_{\substack{i_1\ne \ell_1,\\
				i_2\ne \ell_2}
		}.
	\end{split}   
\end{equation}

 Although this rewriting seems simple, it turns out with these changes, we can drop the term $D_\rz$ in the integrand, following from the lemma below.

\begin{lm}
	\label{lem:remove_D}
	Suppose $\Sigma$ and $\Sigma'$ are contours on the complex plane, $\mathrm{d}\mu(w)$ and $\mathrm{d}\mu'(w')$ are two measures on these contours respectively. Suppose $C(w,w')$ and $D(w,w')$ are two complex-valued functions on $\Sigma\times\Sigma'$, and $B(w_1,\cdots,w_N;w_1',\cdots,w'_N)$ is a complex-valued function defined on $\Sigma^N\times(\Sigma')^N$. Assume that
	\begin{equation*}
		\int_{\Sigma^N}\int_{(\Sigma)^N} \left| B(w_1,\cdots,w_N;w_1',\cdots,w'_N) \right| \cdot \prod_{i=1}^N \left(\left|C(w_i,w'_{\sigma(i)})\right|+\left|D(w_i,w'_{\sigma(i)})\right|\right) \prod_{i=1}^N \left|\mathrm{d}\mu(w_i)\right|
		\prod_{i'=1}^N\left|\mathrm{d}\mu'(w'_{i'})\right|<\infty
	\end{equation*}
for each permutation $\sigma\in S_N$. We further assume that
\begin{equation}
	\label{eq:condition_remove_D}
	\int_\Sigma\int_{\Sigma'}  B(w_1,\cdots,w_N;w_1',\cdots,w'_N) D(w_i,w'_{i'})\mathrm{d}\mu(w_i) \mathrm{d}\mu'(w'_{i'})=0
\end{equation}
for any $1\le i,i'\le N$, and any $w_\ell\in\Sigma, \ell\ne i$, any $w'_{\ell'}\in\Sigma', \ell'\ne i'$. Then we have
\begin{equation}
	\label{eq:remove_D}
	\begin{split}
	&\int_{\Sigma^N}\int_{(\Sigma)^N}
	    B(w_1,\cdots,w_N;w_1',\cdots,w'_N) 
	   \cdot
	   \det\left[C(w_i,w'_{i'}) +D(w_i,w'_{i'})\right]_{i,i'=1}^N \prod_{i=1}^N \mathrm{d}\mu(w_i)\prod_{i'=1}^N\mathrm{d}\mu'(w'_i)\\
	&=\int_{\Sigma^N}\int_{(\Sigma)^N}
	B(w_1,\cdots,w_N;w_1',\cdots,w'_N) 
	\cdot
	\det\left[C(w_i,w'_{i'})\right]_{i,i'=1}^N \prod_{i=1}^N \mathrm{d}\mu(w_i)\prod_{i'=1}^N\mathrm{d}\mu'(w'_i).
	\end{split}
\end{equation}
\end{lm}
\begin{proof}[Proof of Lemma~\ref{lem:remove_D}]
	We expand the determinants on both sides of~\eqref{eq:remove_D}. It turns out all the terms that appear on the left side but not the right side have some factor $D(w_i,w'_{i'})$ in the integrand and hence these terms are zero by the assumption~\eqref{eq:condition_remove_D}. This proves the identity.
\end{proof}

In order to apply Lemma~\ref{lem:remove_D} in~\eqref{eq:finite_time_formula03}, we need to check the assumptions. All of these assumptions are obvious except for the assumption~\eqref{eq:condition_remove_D}, which we verify below. We need to show
\begin{equation*}
	\begin{split}
	\left( \frac{-\rz}{1-\rz} \int_{\Sigma_{\out}}
	\ddbar{w_{i_1}^{(1)}}{}
	+\frac{1}{1-\rz}  \int_{\Sigma_{\inn}}
	\ddbar{w_{i_1}^{(1)}}{}
	\right)
	\int_{\Sigma} \ddbar{w_{i_2}^{(2)}}{}
	\Delta\left(W^{(1)}\right)
	\Delta\left(W^{(2)}\right)
	\tilde f_1\left( W^{(1)} \right)
	\tilde f_2\left( W^{(2)} \right)
	D_\rz\left(w_{i_1}^{(1)},w_{i_2}^{(2)}\right)
	\end{split}
\end{equation*}
equals to zero. If we insert the formulas of $\tilde f_1$ and $\tilde f_2$ (see~\eqref{eq:tilde_f}) and $D_\rz$ (see~\eqref{eq:D}) in the above formula, we only need to prove
\begin{equation}
	\label{eq:check_condition}
	\begin{split}
	&\left( \frac{-\rz}{1-\rz} \int_{\Sigma_{\out}}
	\ddbar{w_{1}}{}
	+\frac{1}{1-\rz}  \int_{\Sigma_{\inn}}
	\ddbar{w_1}{}
	\right)
	\int_{\Sigma} \ddbar{w_2}{}\\
	&\quad 
	\tilde G_1(w_1)\tilde G_2(w_2)
	w_1^{-N}(w_1+1)^{-m}
	(w_2+1)^{-M+m}e^{(s_1+s_2)w_2}	
	\left(
	\frac{ \rz } {-w_1+w_2}
	\frac{ w_1 } { w_2 }
	+  \frac{ 1 } {w_1-w_2}
	\frac{ w_1^{N} e^{(s_1+s_2) w_1} }
	{ w_2^{N} e^{(s_1+s_2) w_2} }
	\right)=0
	\end{split}
\end{equation}
for some polynomials $\tilde G_1$ and $\tilde G_2$ of degree $N-1$. Using a simple residue computation, we have
\begin{equation*}
	\begin{split}
		&\int_{\Sigma_\out}\ddbar{w_1}{}\int_\Sigma\ddbar{w_2}{}	\tilde G_1(w_1)\tilde G_2(w_2)
		w_1^{-N}(w_1+1)^{-m}
		(w_2+1)^{-M+m}e^{(s_1+s_2)w_2}	\frac{ \rz } {-w_1+w_2}
		\frac{ w_1 } { w_2 }=0,\\
		&\int_{\Sigma_\inn}\ddbar{w_1}{}\int_\Sigma\ddbar{w_2}{}	\tilde G_1(w_1)\tilde G_2(w_2)
		w_1^{-N}(w_1+1)^{-m}
		(w_2+1)^{-M+m}e^{(s_1+s_2)w_2}	\frac{ \rz } {-w_1+w_2}
		\frac{ w_1 } { w_2 }\\
		&\quad = \rz \int_\Sigma \ddbar{w}{} \tilde G_1(w) \tilde G_2(w) w^{-N} (w+1)^{-M} e^{(s_1+s_2)w},\\
		&\int_{\Sigma_\out}\ddbar{w_1}{}\int_\Sigma\ddbar{w_2}{}	\tilde G_1(w_1)\tilde G_2(w_2)
		w_1^{-N}(w_1+1)^{-m}
		(w_2+1)^{-M+m}e^{(s_1+s_2)w_2}	\frac{ 1 } {w_1-w_2}
		\frac{ w_1^N e^{(s_1+s_2)w_1} } { w_2^Ne^{(s_1+s_2)w_2} }\\
		&\quad =  \int_\Sigma \ddbar{w}{} \tilde G_1(w) \tilde G_2(w) w^{-N} (w+1)^{-M} e^{(s_1+s_2)w},\\
		&\int_{\Sigma_\inn}\ddbar{w_1}{}\int_\Sigma\ddbar{w_2}{}	\tilde G_1(w_1)\tilde G_2(w_2)
		w_1^{-N}(w_1+1)^{-m}
		(w_2+1)^{-M+m}e^{(s_1+s_2)w_2}	\frac{ 1 } {w_1-w_2}
		\frac{ w_1^N e^{(s_1+s_2)w_1} } { w_2^Ne^{(s_1+s_2)w_2} }=0.
	\end{split}
\end{equation*}
~\eqref{eq:check_condition} follows immediately.

Thus we can apply Lemma~\ref{lem:remove_D} in~\eqref{eq:finite_time_formula03}. After we remove the term $D_\rz$, we exchange the integral and summation again and obtain
	\begin{equation}
	\label{eq:finite_time_formula04}
	\begin{split}
		& p(s_1,s_2;m,n,M,N) \\
		& = \frac{(-1)^{N(N-1)/2}}
		{(N!)^2}
		\oint_0\ddbar{\rz}{\rz^{n}}
		\prod_{i_1=1}^N \left( \frac{-\rz}{1-\rz} \int_{\Sigma_{\out}}
		\ddbar{w_{i_1}^{(1)}}{}
		+\frac{1}{1-\rz}  \int_{\Sigma_{\inn}}
		\ddbar{w_{i_1}^{(1)}}{}
		\right)
		\prod_{i_2=1}^N \int_{\Sigma} \ddbar{w_{i_2}^{(2)}}{}
		\\
		&		
		\quad\Delta\left(W^{(1)}\right)
		\Delta\left(W^{(2)}\right)		
		\tilde f_1\left( W^{(1)} \right)
		\tilde f_2\left( W^{(2)} \right)
		\sum_{\ell_1,\ell_2=1}^N (-1)^{\ell_1+\ell_2}
		\frac{ \left(w_{\ell_1}^{(1)}\right)^n e^{s_1 w_{\ell_1}^{(1)}} }
		{ \left(w_{\ell_2}^{(2)}\right)^{n-1} e^{s_1 w_{\ell_2}^{(2)}} }	
		\det\left[
		C_\rz\left(w_{i_1}^{(1)},w_{i_2}^{(2)}\right)
		\right]_{\substack{i_1\ne \ell_1,\\
				i_2\ne \ell_2}
	   	}.
	\end{split}   
\end{equation}

\subsubsection{Step 2: Evaluating the summation}

Recall the formula of $C_\rz$ in~\eqref{eq:C}. We can write
\begin{equation*}
	C_\rz(w_1,w_2)= \frac{w_1^ne^{s_1w_1}}{w_2^ne^{s_1w_2}} 
	                \cdot
	                \left(
	                \frac{\rz}{w_1-w_2}\frac{w_2}{w_1} 
	                +
	                \frac{1}{-w_1+w_2}\frac{w_1}{w_2}
	                \right).
\end{equation*}
We insert this formula in~\eqref{eq:finite_time_formula04}. Recall the formulas of $\tilde f_1, \tilde f_2$ in~\eqref{eq:tilde_f}, and $\hat f_1, \hat f_2$ in~\eqref{eq:hat_f}. We arrive at
	\begin{equation}
	\label{eq:finite_time_formula05}
	\begin{split}
		& p(s_1,s_2;m,n,M,N) \\
		& = \frac{(-1)^{N(N-1)/2}}
		{(N!)^2}
		\oint_0\ddbar{\rz}{\rz^{n}}
		\prod_{i_1=1}^N \left( \frac{-\rz}{1-\rz} \int_{\Sigma_{\out}}
		\ddbar{w_{i_1}^{(1)}}{}
		+\frac{1}{1-\rz}  \int_{\Sigma_{\inn}}
		\ddbar{w_{i_1}^{(1)}}{}
		\right)
		\prod_{i_2=1}^N \int_{\Sigma} \ddbar{w_{i_2}^{(2)}}{}\Delta\left(W^{(1)}\right)
		\Delta\left(W^{(2)}\right)
		\\
		&		
		\quad		
		\hat f_1\left( W^{(1)} \right)
		\hat f_2\left( W^{(2)} \right)
		\sum_{\ell_1,\ell_2=1}^N (-1)^{\ell_1+\ell_2}
		w_{\ell_2}^{(2)}
		\det\left[
		\frac{\rz}{w_{i_1}^{(1)}-w_{i_2}^{(2)}}\frac{w_{i_2}^{(2)}}{w_{i_1}^{(1)}} 
		+
		\frac{1}{-w_{i_1}^{(1)}+w_{i_2}^{(2)}}\frac{w_{i_1}^{(1)}}{w_{i_2}^{(2)}}
		\right]_{\substack{i_1\ne \ell_1,\\
				i_2\ne \ell_2}
		}.
	\end{split}   
\end{equation}

Compare the above formula with~\eqref{eq:finite_time_formula02}. Note the following Cauchy determinant formula
\begin{equation*}
	\frac{\Delta\left(W^{(1)}\right)
		\Delta\left(W^{(2)}\right)}
	     {\Delta\left(W^{(2)};W^{(1)}\right)}
	     =(-1)^{N(N-1)/2} \det\left[\frac{1}{w_{i_2}^{(2)} -w_{i_1}^{(1)}}\right]_{i_1,i_2=1}^N.
\end{equation*}
We see that~\eqref{eq:finite_time_formula02} follows from~\eqref{eq:finite_time_formula05} and Lemma~\ref{lem:Cauchy_generalization} below. This completes the proof of Proposition~\ref{prop:p_step2}.

The remaining part of this subsection is the next lemma and its proof.

\begin{lm}
	\label{lem:Cauchy_generalization}
	Suppose $X=(x_1,\cdots,x_N)$ and $Y=(y_1,\cdots,y_N)$ are two vectors in $\complexC^N$ satisfying $x_i\ne y_j$ for all $1\le i,j\le N$. Suppose $\rz$ is an arbitrary complex number. Then we have the following identity
	\begin{equation}
		\label{eq:Cauchy_generalization}
		\begin{split}
		&\sum_{a,b=1}^N (-1)^{a +b}
		   y_b \det\left[ \frac{\rz}{ x_i-y_j} \frac{y_j}{x_i}
		                        +\frac{1}{-x_i+y_j}\frac{x_i}{y_j}   
		                  \right]_{\substack{i\ne a\\ j\ne b}}\\
		                  &
		=(1-\rz)^{N-2}\left(\hat H(X;Y)+\rz\prod_{i=1}^N\frac{y_i}{x_i}\hat H(Y;X)\right)\det\left[\frac{1}{y_j-x_i}\right]_{i,j=1}^N,
		\end{split}
	\end{equation}
where $\hat H$ is defined in~\eqref{eq:hat_H}.
\end{lm}

\begin{proof}[Proof of Lemma~\ref{lem:Cauchy_generalization}]
	We first use the identity
	\begin{equation*}
		\frac{\rz}{ x-y} \frac{y}{x}
		+\frac{1}{-x+y}\frac{x}{y} 
		=(1-\rz)\frac{x}{y}
		    \cdot\left(\frac{1}{-x+y} 
		              -\frac{\rz}{1-\rz}\frac{1}{x}
		              -\frac{\rz}{1-\rz}\frac{y}{x^2}
		         \right)
	\end{equation*}
and write the left hand side of~\eqref{eq:Cauchy_generalization} as
\begin{equation}
	(1-\rz)^{N-1}\prod_{i=1}^N\frac{x_i}{y_i} \sum_{a,b=1}^N(-1)^{a+b}\frac{y_b^2}{x_a}
	\det\left[\frac{1}{-x_i+y_j} 
	-\frac{\rz}{1-\rz}\frac{1}{x_i}
	-\frac{\rz}{1-\rz}\frac{y_j}{x_i^2}\right]_{\substack{i\ne a \\ j\ne b}}.
\end{equation}
Thus the equation~\eqref{eq:Cauchy_generalization} is equivalent to, by setting $u=-\rz/(1-\rz)$,
\begin{equation}
	\label{eq:sum_identity_01}
	\begin{split}
	&\sum_{a,b=1}^N (-1)^{a+b} \frac{y_b^2}{x_a} 
	       \det\left[  \frac{1}{-x_i+y_j} +u \frac{1}{x_i} +u \frac{y_j}{x_i^2} \right]_{\substack{i\ne a\\ j\ne b}}\\
	       &
	       = -\left(
	        (u-1)\prod_{i=1}^N \frac{y_i}{x_i} \hat H(X;Y)
	        + u \prod_{i=1}^N\frac{y_i^2}{x_i^2}\hat H(Y; X)
	       \right)\cdot \det\left[\frac{1}{y_j-x_i}\right]_{i,j=1}^N.
	       \end{split}
\end{equation}

The proof of~\eqref{eq:sum_identity_01} is tedious while the strategy is quite straightforward. Below we will show the proof but omit some details which are direct to check. We remark that the strategy was applied to a much simpler identity in \cite[Lemma 5.5]{Baik-Liu17}, but this identity~\eqref{eq:sum_identity_01} is much more complicated.

Before we prove~\eqref{eq:sum_identity_01}, we need to prepare some easier identities. We denote 
\begin{equation*}
	X(w):= \prod_{i=1}^N (w-x_i),\quad Y(w):= \prod_{i=1}^N (w-y_i),
\end{equation*}
and introduce
\begin{equation*}
	C_{p,q} = \sum_{a,b=1}^n x_a^py_b^q\frac{Y(x_a)X(y_b)}{(x_a-y_b)X'(x_a)Y'(y_b)},
\end{equation*}
where $p,q$ are both integers. It is not hard to verify, by using the Cauchy determinant formula, that
\begin{equation}
	\label{eq:C_ab_relation_Cauchy}
	C_{p,q} = \sum_{a,b=1}^N (-1)^{a+b} x_a^p y_b^q \det\left[\frac{1}{-x_i+y_j}\right]_{\substack{i\ne a\\ j\ne b}}
	/\det\left[\frac{1}{-x_i+y_j}\right]_{i,j=1}^N.
\end{equation}

One can evaluate $C_{p,q}$ by converting the sum as a residue computation of an integral on the complex plane. As an illustration, we show how to obtain $C_{-1,2}$, then we will list all the $C_{p,q}$ values we will use later without providing proofs, see Table~\ref{table:C}.

We consider a double integral
\begin{equation*}
	\int_{|y|=R_2} \int_{|x|=R_1} \frac{y^2}{x} \frac{Y(x)X(y)}{(x-y)X(x)Y(y)} \ddbar{x}{} \ddbar{y}{},
\end{equation*}
where $R_1>R_2> \max_i\{|x_i|+|y_i|\}$. Note that we can deform the $x$-contour to infinity and the integral becomes zero. Hence the above double integral is zero. On the other hand, we can change the order of integrals and evaluate the $y$-integral first. It gives a sum over all roots of $Y(y)$:
\begin{equation*}
	0= \int_{|x|=R_1} \sum_{b=1}^N \frac{y_b^2}{x} \frac{Y(x)X(y_b)}{(x-y_b) X(x) Y'(y_b)} \ddbar{x}{}.
\end{equation*}
Then we exchange the summation and integral, and evaluate the $x$-integral by computing the residues within the contour. Note that $x=y_b$ is not a pole. We get
\begin{equation}
	\label{eq:aux_01}
	0=C_{-1,2} -\frac{Y(0)}{X(0)} \sum_{b=1}^N y_b\frac{X(y_b)}{Y'(y_b)}.
\end{equation}
We need to continue to evaluate the summation in~\eqref{eq:aux_01}. We have, by a residue computation,
\begin{equation*}
	\begin{split}
	\sum_{b=1}^N y_b\frac{X(y_b)}{Y'(y_b)}&= \int_{|y|=R_2}y\frac{X(y)}{Y(y)}\ddbar{y}{}\\
	&= \int_{|y|=R_2} y\left(1+\frac{1}{y}\sum_{i=1}^N(y_i-x_i)+\frac{1}{y^2}\hat H(X;Y) +O(y^{-3})\right)\ddbar{y}{}
	 = \hat H(X;Y),
\end{split}
\end{equation*}
where we evaluated the integral by expanding the integrand for large $y$. Here the function $\hat H$ is defined in~\eqref{eq:hat_H}. 

By inserting the above formula to~\eqref{eq:aux_01}, we obtain
\begin{equation*}
	C_{-1,2}=\prod_i\frac{y_i}{x_i}\hat H(X;Y).
\end{equation*}
Using similar calculations, we can find all $C_{p,q}$ for small $p,q$ values. In Table~\ref{table:C} we list some $C_{p,q}$ identities we will use in the proof of~\eqref{eq:sum_identity_01}. We remark that the proof of these identities are analogous to that of $C_{-1,2}$ without adding extra difficulties.

\begin{table}
\begin{tabular}{ |p{1.7cm}|p{2.5cm}|p{1.8cm}|p{7.5cm}|  }
    \hline
    \rule{0pt}{0.6cm}
	Expression& Value &Expression&Value\\
	&&&\\
	\hline
	\rule{0pt}{0.6cm}$C_{0,-1}$   & $\displaystyle 1-\prod_i \frac{x_i}{y_i}$    &
	$C_{-1,2}$ &$\displaystyle \prod_i\frac{y_i}{x_i} \hat H(X;Y)$
	\\
    \hline
	\rule{0pt}{0.6cm}
	$C_{-1,0}$& $\displaystyle-1+ \prod_i \frac{y_i}{x_i}$&$C_{-1,1}-C_{0,0}$ & $\displaystyle \left(1-\prod_i\frac{y_i}{x_i}\right)\sum_i(x_i-y_i)$
	\\
	\hline
		\rule{0pt}{0.6cm}
	$C_{1,0}$ & $\displaystyle -\hat H(Y;X)$
	& 	$C_{0,2}-C_{1,1}$
	&$\displaystyle -\sum_i(x_i-y_i) \hat H(X;Y)$
	\\
	\hline
			\rule{0pt}{0.6cm}
			$C_{0,1}$& $\displaystyle \hat H(X;Y)$&$C_{-2,1}$& $\displaystyle -1 + \prod_i\frac{y_i}{x_i}\left( 1- \sum_i\left(\frac{1}{x_i}-\frac{1}{y_i}\right)\sum_i\left(x_i-y_i\right)\right)$\\
	\hline
\end{tabular}
\caption{\label{table:C}Values of some $C_{p,q}$ expressions.}
\end{table}

We need to evaluate
\begin{equation*}
	\det\left[\frac{1}{-x_i+y_j} +u\frac{1}{x_i}\right]_{i,j=1}^N =\det\left[\frac{1}{-x_i+y_j}\right]_{i,j=1}^N + u\sum_{a,b=1}^N (-1)^{a+b}\frac{1}{x_a} \det\left[\frac{1}{-x_i+y_j}\right]_{\substack{i\ne a\\ j\ne b}}.
\end{equation*}
By applying~\eqref{eq:C_ab_relation_Cauchy} and finding the $C_{-1,0}$ value in Table~\ref{table:C}, we get
\begin{equation}
	\label{eq:aux_02}
	\begin{split}
	\det\left[\frac{1}{-x_i+y_j} +u\frac{1}{x_i}\right]_{i,j=1}^N 
	&=\det\left[\frac{1}{-x_i+y_j}\right]_{i,j=1}^N \left(1+uC_{-1,0}\right)\\
	&=\det\left[\frac{1}{-x_i+y_j}\right]_{i,j=1}^N  \left(1+ u\left(-1+\prod_{i=1}^N\frac{y_i}{x_i}\right)\right).
	\end{split}
\end{equation}
Then we evaluate
\begin{equation*}
	\det\left[  \frac{1}{-x_i+y_j} +u \frac{1}{x_i} +u \frac{y_j}{x_i^2} \right]_{i,j=1}^N = \det\left[\frac{1}{-x_i+y_j} +u\frac{1}{x_i}\right]_{i,j=1}^N
	+u\sum_{a,b=1}^N (-1)^{a+b}\frac{y_b}{x_a^2}\det\left[\frac{1}{-x_i+y_j} +u\frac{1}{x_i}\right]_{\substack{i\ne a\\ j\ne b}}.
\end{equation*}
We insert~\eqref{eq:aux_02} in the above equation and obtain
\begin{equation}
	\label{eq:aux_03}
	\begin{split}
		&\det\left[  \frac{1}{-x_i+y_j} +u \frac{1}{x_i} +u \frac{y_j}{x_i^2} \right]_{i,j=1}^N\\
		&= \det\left[\frac{1}{-x_i+y_j}\right]_{i,j=1}^N  \left(1- u+ u\prod_{i=1}^N\frac{y_i}{x_i}\right)
		+u(1-u)\sum_{a,b=1}^N (-1)^{a+b}\frac{y_b}{x_a^2}\det\left[\frac{1}{-x_i+y_j}\right]_{\substack{i\ne a\\ j\ne b}}\\
		&\quad +u^2\prod_{i=1}^N\frac{y_i}{x_i}\sum_{a,b=1}^N (-1)^{a+b}\frac{1}{x_a}\det\left[\frac{1}{-x_i+y_j}\right]_{\substack{i\ne a\\ j\ne b}}\\
		&=\det\left[\frac{1}{-x_i+y_j}\right]_{i,j=1}^N  \left(1- u+ u\prod_{i=1}^N\frac{y_i}{x_i}
		+u(1-u) C_{-2,1} + u^2\prod_{i=1}^N\frac{y_i}{x_i} C_{-1,0}\right).
	\end{split}
\end{equation}
By inserting the values of $C_{-2,1}$ and $C_{-1,0}$ and simplifying the expression, we obtain
\begin{equation}
	\label{eq:aux_04}
	\begin{split}
		&\det\left[  \frac{1}{-x_i+y_j} +u \frac{1}{x_i} +u \frac{y_j}{x_i^2} \right]_{i,j=1}^N =  \det\left[\frac{1}{-x_i+y_j}\right]_{i,j=1}^N\\
		&\quad \cdot \left[
		1 - 2u\left( 1-\prod_{i=1}^N \frac{y_i}{x_i}\right) -(u-u^2)\prod_{i=1}^N \frac{y_i}{x_i}\sum_{i=1}^N\left(\frac{1}{x_i}-\frac{1}{y_i}\right)\sum_{i=1}^N (x_i-y_i) + u^2 \left(\prod_{i=1}^N \frac{y_i}{x_i}-1\right)^2
		\right].
	\end{split}
\end{equation}
Finally we are ready to prove~\eqref{eq:sum_identity_01}. Inserting~\eqref{eq:aux_04}, we can write
\begin{equation*}
	\begin{split}
		&\sum_{a,b=1}^N (-1)^{a+b} \frac{y_b^2}{x_a} 
		\det\left[  \frac{1}{-x_i+y_j} +u \frac{1}{x_i} +u \frac{y_j}{x_i^2} \right]_{\substack{i\ne a\\ j\ne b}}\\
		&=\sum_{a,b=1}^N (-1)^{a+b} \frac{y_b^2}{x_a} \det\left[\frac{1}{-x_i+y_j}\right]_{\substack{i\ne a\\ j\ne b}}
		\cdot \left[
		1 - 2u\left( 1-\frac{x_a}{y_b}\prod_{i=1}^N \frac{y_i}{x_i}\right)
		 + u^2 \left(\frac{x_a}{y_b}\prod_{i=1}^N \frac{y_i}{x_i}-1\right)^2\right.\\
		&\quad\left. -(u-u^2)\frac{x_a}{y_b}\prod_{i=1}^N \frac{y_i}{x_i}
		\left(-\frac{1}{x_a}+\frac{1}{y_b}+
		\sum_{i=1}^N\left(\frac{1}{x_i}-\frac{1}{y_i}\right)\right)
		\left(-x_a+y_b+ \sum_{i=1}^N (x_i-y_i)\right)
		\right].
	\end{split}
\end{equation*}
We apply~\eqref{eq:C_ab_relation_Cauchy} and rewrite the above equation as
\begin{equation}
	\begin{split}
		&\sum_{a,b=1}^N 
		    (-1)^{a+b} \frac{y_b^2}{x_a} 
		        \det\left[
		               \frac{1}{-x_i+y_j} +u \frac{1}{x_i} +u \frac{y_j}{x_i^2} 
		            \right]_{\substack{i\ne a\\ j\ne b}}
		       /\det\left[
		              \frac{1}{-x_i+y_j} +u \frac{1}{x_i} +u \frac{y_j}{x_i^2} 
		            \right]_{i,j=1}^N\\
	  &=\left(
	         (1-u)^2 +u(1-u)\prod_i\frac{y_i}{x_i}
	    \right)C_{-1,2}
	    +u\prod_i\frac{y_i}{x_i}
	       \left(1-u+u\cdot\prod_i\frac{y_i}{x_i}\right)
	        C_{1,0}\\
	  &\quad-
		  u(1-u)\prod_i\frac{y_i}{x_i}
		      \left(\sum_i\frac{1}{x_i}-\sum_i\frac{1}{y_i}\right)
		      \left(\sum_ix_i-\sum_iy_i\right) C_{0,1}\\
	  &\quad -
	     u(1-u)\prod_i\frac{y_i}{x_i}
	         \left(\sum_ix_i-\sum_iy_i\right)C_{0,0}
	        +
	     u(1-u)\prod_i\frac{y_i}{x_i}
	         \left(\sum_ix_i-\sum_iy_i\right)C_{-1,1}\\
	  &\quad -
	    u(1-u)\prod_i\frac{y_i}{x_i}
	         \left(\sum_i\frac{1}{x_i}-\sum_i\frac{1}{y_i}\right)C_{0,2}
	         +
	    u(1-u)\prod_i\frac{y_i}{x_i}
	         \left(\sum_i\frac{1}{x_i}-\sum_i\frac{1}{y_i}\right)C_{1,1}.
	\end{split}
\end{equation}
By checking the values of Table~\ref{table:C}, and noting that $\left(\sum_i (x_i-y_i)\right)^2=\hat H(X;Y)+\hat H(Y;X)$, we can simplify the above expression. It turns out, after a careful but straightforward calculation, the $u^2$ term vanishes, and the remaining terms match the right hand side of~\eqref{eq:sum_identity_01}. We hence complete the proof.
\end{proof}

\subsection{Proof of Proposition~\ref{prop:p_step3}}
\label{sec:proof_step3}

In this subsection, we prove Proposition~\ref{prop:p_step3}. Note that the equation~\eqref{eq:finite_time_formula02} involves a Cauchy determinant factor 
\begin{equation*}
	\frac{\Delta\left(W^{(1)}\right) \Delta \left(W^{(2)}\right)}
	     {\Delta\left( W^{(2)}; W^{(1)}\right)}
	      = (-1)^{N(N-1)/2}\det\left[\frac{1}{w_{i_2}^{(2)} -w_{i_1}^{(1)}}\right]_{i_1,i_2=1}^N,
\end{equation*}
which is of size N, while the formula~\eqref{eq:def_finite_time_density} is analogous to a Fredholm determinant expansion. So Proposition~\ref{prop:p_step3} can be interpreted as an identity between a Cauchy determinant of large size and a Fredholm-determinant-like expansion. Our strategy contains three steps. First, we rewrite the formula~\eqref{eq:finite_time_formula02} to a summation on discrete spaces with summand having similar Cauchy determinant structures. This rewriting involves a generalized version of an identity in \cite{Liu19}. In the second step, we reformulate the summation to a Fredholm-determinant-like expansion on the same discrete space. We remark that similar calculation were considered in \cite{Baik-Liu16,Baik-Liu17} but our summand is more involved. Finally, we verify that the expansion indeed matches~\eqref{eq:def_finite_time_density} using the identity obtained in the first step.

Below we will first introduce a generalized version of an identity in \cite{Liu19}, the Proposition 4.3 of \cite{Liu19}. Then we prove Proposition~\ref{prop:p_step3} using the above strategy.

\subsubsection{A Cauchy-type summation identity}

We introduce a few concepts before we state the results. We will mainly follow \cite[Section 4]{Liu19} but add a small generalization.

Suppose $W=(w_1,\cdots,w_n)\in\complexC^n$ and $W'=(w'_1,\cdots,w'_{n'})\in\complexC^{n'}$ are two vectors without overlapping coordinates, i.e., they satisfy $w_i\ne w'_{i'}$ for all $i,i'$. We define
\begin{equation}
	\label{eq:2021_12_23_02}
	\Ch(W;W') = \frac{\Delta(W)\Delta(W')}{\Delta(W;W')}
\end{equation}
and call it a \emph{Cauchy-type factor}. Note that when $n=n'$, $\Ch(W;W')$ equals to a Cauchy determinant $\det\left[1/(w_i-w'_{i'})\right]_{i,i'=1}^n$ multiplied by a sign factor $(-1)^{n(n-1)/2}$. We remark that we allow empty product and view it as $1$ in the above definition. For example, when $n'=0$, we have $\Ch(W;W')=\Delta(W)$.

Similar as in~\eqref{eq:SW2}, we use the convention that $W_I=(w_{i_1},\cdots,w_{i_k})$ for 
any index set $I=\{i_1,\cdots,i_k\}$ where $1\le i_1<\cdots<i_k\le n$. In other words, $W_I$ is the vector formed by the coordinates with indices in $I$.

We denote 
\begin{equation*}
	\mathbb{D}(r) := \{z: |z|< r\}, \text{ and } \mathbb{D}_0(r) = \{z: 0<|z|<r\}.
\end{equation*}
And we omit $r$ when $r=1$, i.e., $\mathbb{D}=\mathbb{D}(1)$ and $\mathbb{D}_0=\mathbb{D}_0(1)$. 

Suppose $q(w)$ is a function which is analytic in a certain bounded region $\mathcal{D}$. Denote
\begin{equation}
	\mathcal{R}_z=\left\{
	              w\in \mathcal{D}: q(w)=z\right\}.
\end{equation}
Assume that $\mathcal{R}_0\ne \emptyset$. In other words, there is at least one root of $q(w)$ within $\mathcal{D}$. We also assume that $\rr$ is a positive constant such that $\cup_{z\in\mathbb{D}(\rr)}\mathcal{R}_z=\{w\in\mathcal{D}:|q(w)|\le\rr\}$ lies within a compact subset of $\mathcal{D}$, and $\{w\in\mathcal{D}:|q(w)|=r\}$ for all $0<|r|<\rr$ consists of $|\mathcal{R}_0|$ non-intersecting simply connected contours around the points in $\mathcal{R}_0$. It is easy to see that with these assumptions $q'(w)\ne 0$ for all $w\in \{w\in\mathcal{D}:|q(w)|<\rr\}$. We remark that in the original setting of \cite{Liu19}, they assumed $\mathcal{R}_0=\{0\}$ or $\{-1\}$. Here we drop this assumption.

We will consider a Cauchy-type summation, which involves an expression
\begin{equation}
	\label{eq:calH}
	\mathcal{H}\left(W^{(1)},\cdots,W^{(\ell)};z_0,\cdots,z_{\ell-1}\right)
	:= \left[\prod_{k=1}^{\ell-1} \Ch\left( W_{I^{(k)}}^{(k)};  W_{J^{(k+1)}}^{(k+1)}\right)\right] \cdot \mathcal{A}\left(W^{(1)},\cdots,W^{(\ell)};z_0,\cdots,z_{\ell-1}\right),
\end{equation}
where $W^{(k)}=\left(w^{(k)}_1,\cdots,w_{n_k}^{(k)}\right) \in \complexC^{n_k}$, $1\le k\le \ell$, such that $W^{(k)}$ and $W^{(k+1)}$ do not have overlapping coordinates for $1\le k\le \ell-1$. $I^{(k)}$ and $J^{(k)}$ are arbitrary subsets of $ \{1,\cdots,n_k\}$ for $1\le k\le \ell-1$ and $2\le k\le \ell$ respectively. The function $\mathcal{A}$ is analytic for all $w_{j_k}^{(k)}\in\mathcal{D}\setminus\mathcal{R}_0$, $1\le j_k\le n_k, 1\le k\le \ell$, and for all $(z_0,\cdots,z_{\ell-1})\in\mathbb{D}(\rr)\times \mathbb{D}^{\ell-1}$. Hence $\mathcal{H}$ is also analytic on $(\mathcal{D}\setminus \mathcal{R}_0)^{n_1+\cdots+n_\ell}\times \mathbb{D}(\rr)\times \mathbb{D}^{\ell-1}$, except for having possible poles at $w_{i_k}^{(k)}=w_{i_{k+1}}^{(k+1)}$ for some $i_k\in I^{(k)}$ and $i_{k+1}\in I^{(k+1)}$, which comes from the Cauchy-type factors. We remark that the function $\mathcal{H}$ also depends on the index sets $I^{(k)}, J^{(k+1)}$, $1\le k\le \ell-1$.

Now we introduce the summation. We consider
\begin{equation}\label{eq:calG}
	\mathcal{G}(z_0,\cdots,z_{\ell-1})=
	\sum_{W^{(1)}\in \mathcal{R}_{\hat z_1}^{n_1}}
	\cdots
	\sum_{W^{(\ell)}\in \mathcal{R}_{\hat z_\ell}^{n_\ell}}
                \left[\prod_{k=1}^\ell J(W^{(k)})\right]
                \mathcal{H}\left(W^{(1)},\cdots,W^{(\ell)};z_0,\cdots,z_{\ell-1}\right)
\end{equation}
for $(z_0,\cdots,z_\ell) \in \mathbb{D}_0(\rr)\times \mathbb{D}_0^{\ell-1}$,
where the function
\begin{equation}
	J(w):= \frac{q(w)}{q'(w)}.
\end{equation}
Recall our convention $J(W^{(k)})=\prod_{a=1}^{n_k}J(w_a^{(k)})$. The variables $\hat z_k$'s are defined by
\begin{equation}
	\hat z_k = z_0z_1\cdots z_{k-1},\quad k=1,\cdots,\ell.
\end{equation}

Note the identity
\begin{equation}
	\label{eq:2021_12_31_01}
	\sum_{w\in\mathcal{R}_z}f(w)H(w)= \left(\int_{|q(w)|=C_1}-\int_{|q(w)|=C_2}\right)\frac{f(w)q(w)}{q(w)-z}\ddbar{w}{},
\end{equation}
where $C_1$ and $C_2$ are two positive constants satisfying  $C_2<|z|<C_1$ such that the function $f(w)$ is analytic in $\{w: C_2< |q(w)| < C_1\}$. The right hand side is analytic as a function of $z$ within $C_2<|z|<|C_1|$. This identity implies that $\sum_{w\in\mathcal{R}_z}f(w)H(w)$ is also analytic as a function of $z$ within $C_2<|z|<|C_1|$. Using this fact we obtain that $\mathcal{G}(z_0,\cdots,z_{\ell-1})$
is analytic as a function of $\hat z_1,\cdots,\hat z_\ell$ within $0<|\hat z_\ell|<\cdots<|\hat z_1|<\rr$, and hence is analytic as a function of $z_0,\cdots,z_{\ell-1}$ in $\mathbb{D}_0(\rr)\times \mathbb{D}_0^{\ell-1}$. 
We remark that there are no poles from the Cauchy-type factor due to the order of $|\hat z_k|$.

Our goal is to analytically extend the function $\mathcal{G}$ to $\mathbb{D}(\rr)\times\mathbb{D}^{\ell-1}$ under certain assumption. Below we introduce two more concepts related the assumption, then we state the identity.

We call a sequence of variables $w_{i_k}^{(k)},w_{i_{k+1}}^{(k+1)},\cdots,w_{i_{k'}}^{(k')}$ \emph{a Cauchy chain} with respect to the vectors $W^{(1)},\cdots, W^{(\ell)}$ and index sets $I^{(1)},J^{(2)},I^{(2)},J^{(3)},\cdots, I^{(\ell-1)}, J^{(\ell)}$, if 
\begin{equation*}
	\left(w_{i_k}^{(k)} -w_{i_{k+1}}^{(k+1)}\right)\cdot
	\left(w_{i_{k+1}}^{(k+1)} -w_{i_{k+2}}^{(k+2)}\right)
	\cdot\cdots \cdot
	\left(w_{i_{k'-1}}^{(k'-1)} -w_{i_{k'}}^{(k')}\right)
\end{equation*}
appears as a factor of the denominator in $\prod_{k=1}^{\ell-1} \Ch\left( W_{I^{(k)}}^{(k)};  W_{J^{(k+1)}}^{(k+1)}\right)$. We allow any single variable $w_{i_k}^{(k)}$ to be a Cauchy chain as long as it is a coordinate of $W^{(k)}$.

We say $q(w)$ \emph{dominates} $\mathcal{H}\left(W^{(1)},\cdots,W^{(\ell)};z_0,\cdots,z_{\ell-1}\right)$ if and only if the following function of $w$
\begin{equation}
	\left.q(w)\cdot \mathcal{A}\left(W^{(1)},\cdots,W^{(\ell)};z_0,\cdots,z_{\ell-1}\right)\right|_{w_{i_k}^{(k)}=w_{i_{k+1}}^{(k+1)}=\cdots=w_{i_{k'}}^{(k')}=w}
\end{equation}
is analytic at any $w\in\mathcal{R}_0$ when all other variables are fixed, here $w_{i_k}^{(k)},w_{i_{k+1}}^{(k+1)},\cdots,w_{i_{k'}}^{(k')}$ is an arbitrary Cauchy chain with respect to $W^{(1)},\cdots, W^{(\ell)}$ and $I^{(1)},J^{(2)},I^{(2)},J^{(3)},\cdots, I^{(\ell-1)}, J^{(\ell)}$. We remark that in \cite{Liu19}, this concept was only defined when $\mathcal{R}_0$ contains one single point. Here we dropped this assumption.

\begin{prop}
	\label{prop:Cauchy_summation_identity}
	If $q(w)$ dominates $\mathcal{H}\left(W^{(1)},\cdots,W^{(\ell)};z_0,\cdots,z_{\ell-1}\right)$, then the function $\mathcal{G}(z_0,\cdots,z_{\ell-1})$ can be analytically extended to $\mathbb{D}(\rr)\times\mathbb{D}^{\ell-1}$. Moreover, $\mathcal{G}(z_0=0,z_1\cdots,z_{\ell-1})$ is independent of $q(w)$, and it equals to
	\begin{equation*}
		\prod_{k=2}^{\ell}\prod_{i_k=1}^{n_k} \left[\frac{1}{1-z_{k-1}} \int_{\Sigma_{\inn}^{(k)}}\ddbar{w_{i_k}^{(k)}}{} -\frac{z_{k-1}}{1-z_{k-1}}\int_{\Sigma^{(k)}_{\out}}\ddbar{w_{i_k}^{(k)}}{}\right]\prod_{i_1=1}^{n_1} \int_{\Sigma^{(1)}}\ddbar{w_{i_1}^{(1)}}{}\mathcal{H}\left(W^{(1)},\cdots,W^{(\ell)};0,z_1,\cdots,z_{\ell-1}\right),
	\end{equation*}
where $\Sigma_{\out}^{(\ell)},\cdots,\Sigma_\out^{(2)},\Sigma^{(1)},\Sigma_\inn^{(2)},\cdots,\Sigma_\inn^{(\ell)}$ are $2\ell-1$ nested contours in $\mathcal{D}$ each of which encloses all the points in $\mathcal{R}_0$.
\end{prop}
\begin{proof}[Proof of Proposition~\ref{prop:Cauchy_summation_identity}]
	When $\mathcal{R}_0=\{0\}$, this is exactly the same as \cite[Proposition 4.3]{Liu19}. On the other hand, their proof does not use the fact $\mathcal{R}_0=\{0\}$, see \cite[Section 6]{Liu19}. Hence Proposition~\ref{prop:Cauchy_summation_identity} follows from the same argument.
\end{proof}

One can similarly consider a two-region version of the above result. Assume that $\mathcal{D}_\LL$ and $\mathcal{D}_\RR$ are two disjoint bounded regions on the complex plane. Let $q(w)$ be a function analytic in $\mathcal{D}_\LL\cup\mathcal{D}_\RR$ and define
\begin{equation*}
	\mathcal{R}_{z,\LL}=\{u\in\mathcal{D}_\LL: q(u)=z\},\quad \text{and }
	\mathcal{R}_{z,\RR}=\{v\in\mathcal{D}_\RR: q(v)=z\}.
\end{equation*}
Assume that both $\mathcal{R}_{0,\LL}$ and $\mathcal{R}_{0,\RR}$ are nonempty. The analog of~\eqref{eq:calH} is
\begin{equation*}
	\begin{split}
		&	\mathcal{H}\left(U^{(1)},\cdots,U^{(\ell)};V^{(1)},\cdots,V^{(\ell)};z_0,\cdots,z_{\ell-1}\right)\\
		&
		:= \left[\prod_{k=1}^{\ell-1} \Ch\left( U_{I_\LL^{(k)}}^{(k)};  U_{J_\LL^{(k+1)}}^{(k+1)}\right)
		 \Ch\left( V_{I_\RR^{(k)}}^{(k)};  V_{J_\RR^{(k+1)}}^{(k+1)}\right)\right] \cdot \mathcal{A}\left(U^{(1)},\cdots,U^{(\ell)};V^{(1)},\cdots,V^{(\ell)};z_0,\cdots,z_{\ell-1}\right),
	\end{split}
\end{equation*}
where $\mathcal{A}$ is analytic in $\mathcal{D}_\LL\setminus \mathcal{R}_{0,\LL}$ for each coordinate of $U^{(k)}$, and in $\mathcal{D}_\RR\setminus \mathcal{R}_{0,\RR}$ for each coordinate of $V^{(k)}$, $1\le k\le \ell$, and analytic for all $(z_0,\cdots,z_\ell)\in\mathbb{D}(\rr)\times \mathbb{D}^{\ell-1}$. The analog of~\eqref{eq:calG} is
\begin{equation*}
	\mathcal{G}(z_0,\cdots,z_{\ell-1})=
	\sum_{\substack{U^{(1)}\in \mathcal{R}_{\hat z_1,\LL}^{n_{1,\LL}}\\
	 V^{(1)}\in \mathcal{R}_{\hat z_1,\RR}^{n_{1,\RR}}}}
	\cdots
	\sum_{\substack{U^{(\ell)}\in \mathcal{R}_{\hat z_\ell,\LL}^{n_{\ell,\LL}}\\
			V^{(\ell)}\in \mathcal{R}_{\hat z_\ell,\RR}^{n_{\ell,\RR}}}}
	\left[\prod_{k=1}^\ell J(U^{(k)})J(V^{(k)})\right]
	\mathcal{H}\left(U^{(1)},\cdots,U^{(\ell)};V^{(1)},\cdots,V^{(\ell)};z_0,\cdots,z_{\ell-1}\right)
\end{equation*}
for $(z_0,\cdots,z_\ell) \in \mathbb{D}_0(\rr)\times \mathbb{D}_0^{\ell-1}$. We can similarly define Cauchy chains in $\mathcal{D}_\LL$ and in $\mathcal{D}_\RR$. We say $q(w)$ dominates $\mathcal{H}\left(U^{(1)},\cdots,U^{(\ell)};V^{(1)},\cdots,V^{(\ell)};z_0,\cdots,z_{\ell-1}\right)$ if
\begin{equation*}
	\left.q(u)\cdot \mathcal{A}\left(U^{(1)},\cdots,U^{(\ell)};V^{(1)},\cdots,V^{(\ell)};z_0,\cdots,z_{\ell-1}\right)\right|_{u_{i_k}^{(k)}=u_{i_{k+1}}^{(k+1)}=\cdots=u_{i_{k'}}^{(k')}=u}
\end{equation*}
is analytic at any $u\in \mathcal{R}_{0,\LL}$ for any Cauchy chain in $\mathcal{D}_\LL$, and
\begin{equation*}
	\left.q(v)\cdot \mathcal{A}\left(U^{(1)},\cdots,U^{(\ell)};V^{(1)},\cdots,V^{(\ell)};z_0,\cdots,z_{\ell-1}\right)\right|_{v_{i_k}^{(k)}=v_{i_{k+1}}^{(k+1)}=\cdots=v_{i_{k'}}^{(k')}=v}
\end{equation*}
is analytic at any $v\in \mathcal{R}_{0,\RR}$ for any Cauchy chain in $\mathcal{D}_\RR$. The analog of Proposition~\ref{prop:Cauchy_summation_identity} is as follows.
\begin{prop}
	\label{prop:Cauchy_summation_identity2}
	If $q(w)$ dominates $\mathcal{H}\left(U^{(1)},\cdots,U^{(\ell)};V^{(1)},\cdots,V^{(\ell)};z_0,\cdots,z_{\ell-1}\right)$, then the function $\mathcal{G}(z_0,\cdots,z_{\ell-1})$ can be analytically extended to $\mathbb{D}(\rr)\times\mathbb{D}^{\ell-1}$. Moreover, $\mathcal{G}(z_0=0,z_1\cdots,z_{\ell-1})$ is independent of $q(w)$, and it equals to
	\begin{equation*}
		\begin{split}
		&\prod_{k=2}^{\ell}\prod_{i_k=1}^{n_{k,\LL}} \left[\frac{1}{1-z_{k-1}} \int_{\Sigma_{\inn,\LL}^{(k)}}\ddbar{u_{i_k}^{(k)}}{} -\frac{z_{k-1}}{1-z_{k-1}}\int_{\Sigma^{(k)}_{\out,\LL}}\ddbar{u_{i_k}^{(k)}}{}\right]\prod_{i_1=1}^{n_{1,\LL}} \int_{\Sigma^{(1)}_\LL}\ddbar{u_{i_1}^{(1)}}{}\\
		&\prod_{k=2}^{\ell}\prod_{i_k=1}^{n_{k,\RR}} \left[\frac{1}{1-z_{k-1}} \int_{\Sigma_{\inn,\RR}^{(k)}}\ddbar{v_{i_k}^{(k)}}{} -\frac{z_{k-1}}{1-z_{k-1}}\int_{\Sigma^{(k)}_{\out,\RR}}\ddbar{v_{i_k}^{(k)}}{}\right]\prod_{i_1=1}^{n_{1,\RR}} \int_{\Sigma^{(1)}_\RR}\ddbar{v_{i_1}^{(1)}}{}\\
		&\mathcal{H}\left(U^{(1)},\cdots,U^{(\ell)};V^{(1)},\cdots,V^{(\ell)};0,z_1\cdots,z_{\ell-1}\right),
		\end{split}
	\end{equation*}
where $\Sigma_{\out,\LL}^{(\ell)},\cdots,\Sigma_{\out,\LL}^{(2)},\Sigma^{(1)}_\LL,\Sigma_{\inn,\LL}^{(2)},\cdots,\Sigma_{\inn,\LL}^{(\ell)}$ are $2\ell-1$ nested contours in $\mathcal{D}_\LL$ each of which encloses all the points in $\mathcal{R}_{0,\LL}$, and $\Sigma_{\out,\RR}^{(\ell)},\cdots,\Sigma_{\out,\RR}^{(2)},\Sigma^{(1)}_\RR,\Sigma_{\inn,\RR}^{(2)},\cdots,\Sigma_{\inn,\RR}^{(\ell)}$ are $2\ell-1$ nested contours in $\mathcal{D}_\RR$ each of which encloses all the points in $\mathcal{R}_{0,\RR}$.
\end{prop}
\begin{proof}[Proof of Proposition~\ref{prop:Cauchy_summation_identity2}]
	The case when $\mathcal{R}_{0,\LL}=\{-1\}$ and $\mathcal{R}_{0,\RR}=\{0\}$ was the same as \cite[Proposition 4.4]{Liu19}. The proof for the more general case is also the same as the proof of \cite[Proposition 4.4]{Liu19}, except that we apply Proposition~\ref{prop:Cauchy_summation_identity} in this paper instead of \cite[Proposition 4.2]{Liu19}.
\end{proof}

\subsubsection{Rewriting~\eqref{eq:finite_time_formula02}}
\label{sec:rewriting_3_5}

Now we want to apply Proposition~\ref{prop:Cauchy_summation_identity} to equation~\eqref{eq:finite_time_formula02} and rewrite the formula.

We first choose $q(w)=w^N(w+1)^{L-N}$, where $L$ is any fixed integer satisfying $L\ge M+N$. Recall the formula~\eqref{eq:finite_time_formula02}. Let $\mathcal{H}\left(W^{(2)},W^{(1)};z_1,z_0=\rz\right)$ be a slight modification of the integrand in~\eqref{eq:finite_time_formula02}. More precisely, let
\begin{equation}
	\label{eq:calH2}
	\mathcal{H} \left(W^{(2)},W^{(1)};z_1,z_0\right)=\Ch\left(W^{(2)};W^{(1)}\right)\mathcal{A} \left(W^{(2)},W^{(1)};z_1,z_0\right),
\end{equation}
where
\begin{equation}
	\label{eq:calA2}
	\begin{split}
	&\mathcal{A} \left(W^{(2)},W^{(1)};z_1,z_0\right)\\
	&:=\Delta\left(W^{(1)}\right)\Delta\left(W^{(2)}\right)\hat f_1\left(W^{(1)}\right) \hat f_2\left(W^{(2)}\right)\left[\hat H\left(W^{(1)};W^{(2)}\right) +z_0\prod_{i=1}^N\frac{w_{i}^{(2)}}{w_i^{(1)}}\hat H\left(W^{(1)};W^{(2)}\right)\right].
	\end{split}
\end{equation}
Note that when $z_0=\rz$, $\mathcal{H}\left(W^{(2)},W^{(1)};z_1,z_0\right)$ is exactly the integrand of~\eqref{eq:finite_time_formula02}. Assume $\mathcal{D}$ is a bounded region enclosing both $0$ and $-1$. It is obvious that the function $\mathcal{A}$ is well defined and analytic for all $w_{i}^{(1)}, w_i^{(2)}\in \mathcal{D}\setminus\{\-1,0\}$, $1\le i\le N$, and for all $(z_1,z_0)\in \mathbb{D}(\rr)\times\mathbb{D}$, here we choose
\begin{equation}
	\label{eq:rr_value}
	\rr= N^N(L-N)^{L-N}/L^L.
\end{equation}

We remark that we have a different ordering of indices compared to the original formulas~\eqref{eq:calH} and~\eqref{eq:calG}. This is because we want to make the indices of $\hat f_1$ and $\hat f_2$ more natural by using $1$ to label the parameters appearing in the first part of the last passage time and using $2$ to label the parameters appearing in the second part of the last passage time. On the other hand, we also want to make our indices in Propositions~\ref{prop:Cauchy_summation_identity} and~\ref{prop:Cauchy_summation_identity2} consistent with \cite{Liu19} so the readers can compare the results easily. These different orderings might be confusing but they only appear in this technical proof. We will keep reminding readers if needed.

The sum we are considering is
\begin{equation}
	\label{eq:calG2}
	\mathcal{G}(z_1,z_0) = \sum_{W^{(2)}\in \mathcal{R}_{\hat z_2}^N} \sum_{W^{(1)} \in \mathcal{R}_{\hat z_1}^{N}} J\left(W^{(1)}\right)J\left(W^{(2)}\right) \mathcal{H}\left(W^{(2)},W^{(1)};z_1,z_0\right),
\end{equation}
where $\hat z_2=z_1$ and $\hat z_1=z_1z_0$. We assume that $z_1\in\mathbb{D}_0(\rr)$ and $z_0\in\mathbb{D}_0$ hence $0<|\hat z_2|< |\hat z_1|<\rr$.

We need to verify that Proposition~\ref{prop:Cauchy_summation_identity} is applicable for this function~\eqref{eq:calG2}. All other assumptions are trivial, except for the one that $q(w)$ dominates $\mathcal{H} \left(W^{(2)},W^{(1)};z_1,z_0\right)$. We verify it below.

There are only three types of Cauchy chains. The chains of single element $w_{i_1}^{(1)}$ or $w_{i_2}^{(2)}$, and the chain of two elements $w_{i_2}^{(2)}, w_{i_1}^{(1)}$. For the first type of chains, we need to verify $q(w_{i_1}^{(1)})\mathcal{A} \left(W^{(2)},W^{(1)};z_1,z_0\right)$ is analytic at $0$ and $-1$. This follows from the fact that $\hat f_1(w)q(w)w^{-1}=(w+1)^{L-N-m}w^{n-1}e^{s_1w}$ is an entire function. Similarly we can verify it for the second type of Cauchy chains. Finally, for the chain of two elements $w_{i_2}^{(2)},w_{i_1}^{(1)}$, we need to show $\left.q(w)\mathcal{A} \left(W^{(2)},W^{(1)};z_1,z_0\right)\right|_{w_{i_2}^{(2)}=w_{i_1}^{(1)}=w}$ is analytic at $-1$ and $0$. It follows from the fact that $\hat f_1(w)\hat f_2(w)q(w)=(w+1)^{L-N-M}e^{(s_1+s_2)w}$ is entire.

So we can apply Proposition~\ref{prop:Cauchy_summation_identity}, and obtain
\begin{equation*}
	\begin{split}
	&\mathcal{G}(0,z_0=\rz)=\prod_{i_1=1}^N \left(
	\frac{-\rz}{1-\rz} \int_{\Sigma_{\out}}
	\ddbar{w_{i_1}^{(1)}}{}
	+\frac{1}{1-\rz}  \int_{\Sigma_{\inn}}
	\ddbar{w_{i_1}^{(1)}}{}
	\right)
	\prod_{i_2=1}^N \int_{\Sigma} \ddbar{w_{i_2}^{(2)}}{}\\
	&\quad 
	\hat f_1\left(W^{(1)}\right)
	\hat f_2\left(W^{(2)}\right)
	\frac{ \left( \Delta\left(W^{(1)}\right) \right)^2 
		\left( \Delta\left(W^{(2)}\right) \right)^2
	}
	{ \Delta\left(W^{(2)};W^{(1)}\right)
	}
	\cdot 
	\left(
	\hat H \left(W^{(1)};W^{(2)}\right)
	+ \rz \frac{\prod_{i_2=1}^N w_{i_2}^{(2)}}
	{\prod_{i_1=1}^N w_{i_1}^{(1)}}
	\hat H \left(W^{(2)}; W^{(1)}\right)
	\right).
	\end{split}
\end{equation*}
Hence we have an alternate expression for~\eqref{eq:finite_time_formula02}
\begin{equation}
	\label{eq:aux_09}
	p(s_1,s_2;m,n,M,N)=\frac{1}{(N!)^2}\oint_0 \mathcal{G}(0,z_0=\rz) \frac{ (1-\rz)^{N-2} \mathrm{d}\rz }
	{ 2\pi\mathrm{i} \rz^n }.
\end{equation}

\subsubsection{Reformulation to a Fredholm-determinant-like expansion}

In this subsubsection, we want to evaluate the summation~\eqref{eq:calG2} in a different way. Recall $q(w)=w^N(w+1)^{L-N}$ and $\mathcal{R}_z$ are the roots of $q(w)=z$. This equation is called the Bethe equation, and its roots are called the Bethe roots. It is known \cite{Baik-Liu16} that when $|z|< \rr= N^N (L-N)^{L-N}/L^L$, the set $\mathcal{R}_z$ can be split into two different subsets $\mathcal{R}_{z,\LL}$ and $\mathcal{R}_{z,\RR}$ satisfying $|\mathcal{R}_{z,\LL}|=L-N$ and $|\mathcal{R}_{z,\RR}|=N$. Intuitively, each root in $\mathcal{R}_{z,\LL}$ ($\mathcal{R}_{z,\RR}$, respectively) can be viewed as an continuous function of $z$ starting from $-1$ ($0$, respectively) when $z=0$. We denote
\begin{equation}
	\mathcal{D}_{\LL} = \cup_{|z|<\rr} \mathcal{R}_{z,\LL},\quad \text{and}\quad
	\mathcal{D}_{\RR} = \cup_{|z|<\rr} \mathcal{R}_{z,\RR},
\end{equation}
and
\begin{equation}
	q_{z,\LL} (w)= \prod_{u\in\mathcal{R}_{z,\LL}} (w-u),\quad \text{and}\quad 
	q_{z,\RR} (w)= \prod_{v\in\mathcal{R}_{z,\LL}} (w-v)
\end{equation}
which will be used in later computations. Note that $\mathcal{D}_{\LL}$ and $\mathcal{D}_{\RR}$ are two disjoint bounded regions, and $q_{z,\LL}(w)q_{z,\RR}(w)=q(w)-z$.

We will rewrite the summation~\eqref{eq:calG2} by treating $w_{i_k}^{(k)}\in\mathcal{R}_{\hat z_k,\LL}$ and $w_{i_k}^{(k)}\in\mathcal{R}_{\hat z_k,\RR}$ separately. We first observe that, by checking the formulas~\eqref{eq:calH2} and~\eqref{eq:calA2}, the summand is invariant when we permute the coordinates of $W^{(k)}$, $k=1,2$. We also observe that the summand is zero if any two coordinates of $W^{(k)}$ are equal due to the Cauchy-type factor. Therefore we only need to consider the summation for $W^{(k)}$ with different coordinates.

Assume that $n_k$ coordinates in $W^{(k)}$ are chosen from $\mathcal{R}_{\hat z_k,\LL}$. Then the other $N-n_k$ coordinates are chosen from $\mathcal{R}_{\hat z_k,\RR}$. Note that $\mathcal{R}_{\hat z_k,\RR}$ has exactly $N$ elements, hence there are $n_k$ elements which do not appear in $W^{(k)}$. We denote $V^{(k)}=(v_1^{(k)},\cdots,v_{n_k}^{(k)})$ the vector formed by these elements with any given order. We also denote $U^{(k)}=(u_1^{(k)},\cdots,u_{n_k}^{(k)})$ the vector formed by the coordinates of $W^{(k)}$ in $\mathcal{R}_{\hat z_k,\LL}$. Note the invariance property we observed above. We write
\begin{equation}
	\label{eq:sum_rewriting}
	\sum_{W^{(2)} \in \mathcal{R}_{\hat z_2}^N}
	\sum_{W^{(1)} \in \mathcal{R}_{\hat z_1}^N} =(N!)^2\sum_{n_1,n_2=0}^N \frac{1}{(n_1!)^2(n_2!)^2}\sum_{\substack{U^{(2)} \in \mathcal{R}_{\hat z_2,\LL}^{n_2}\\ V^{(2)} \in \mathcal{R}_{\hat z_2,\RR}^{n_2}}}\sum_{\substack{U^{(1)} \in \mathcal{R}_{\hat z_1,\LL}^{n_1}\\ V^{(1)} \in \mathcal{R}_{\hat z_1,\RR}^{n_1}}},
\end{equation}
where the factors $N!$, $n_k!$ come from the number of ways to permute the coordinates of $W^{(k)}$, $U^{(k)}$ (and $V^{(k)}$) respectively. Now we need to rewrite the summand in terms of $U^{(k)}$ and $V^{(k)}$, $k=1,2$. Such a rewriting was mostly done in \cite{Baik-Liu16,Baik-Liu17} except for one extra factor. We will write down the formulas without proofs except for the one involving the extra factor.

Recall the notation conventions~\eqref{eq:def_Delta1},~\eqref{eq:def_Delta2} and~\eqref{eq:def_f_vector}. We write, by simply inserting the coordinates,
\begin{equation*}
	\hat f_k\left(W^{(k)}\right) = \frac{\hat f_k\left(U^{(k)}\right)}{\hat f_k\left(V^{(k)}\right)} \cdot \hat f_k\left(\mathcal{R}_{\hat z_k,\RR}\right),
	\quad 
	J\left(W^{(k)}\right) = \frac{J\left(U^{(k)}\right)}{J\left(V^{(k)}\right)} \cdot  J\left(\mathcal{R}_{\hat z_k,\RR}\right), \quad k=1,2.
\end{equation*}
We also have (see equation (4.43) of \cite{Baik-Liu17})
\begin{equation*}
	\Delta\left(W^{(k)}\right)^2 = (-1)^{N(N-1)/2} 
	                    \frac{\Delta\left(U^{(k)}\right)^2\Delta\left(V^{(k)}\right)^2}
	                         {\Delta\left(U^{(k)};V^{(k)}\right)^2}
	                    \frac{q_{\hat z_k,\RR}^2\left(U^{(k)}\right)}
	                         {\left(q'_{\hat z_k,\RR}\left(V^{(k)}\right)\right)^2}
	                    \cdot q'_{\hat z_k,\RR}\left(\mathcal{R}_{\hat z_k,\RR}\right)
\end{equation*}
and (see equation (4.44) of \cite{Baik-Liu17})
\begin{equation*}
	\Delta\left(W^{(2)};W^{(1)}\right)=\Delta\left(\mathcal{R}_{\hat z_2,\RR};\mathcal{R}_{\hat z_1,\RR}\right)
	\frac{\Delta\left(U^{(2)};U^{(1)}\right)\Delta\left(V^{(2)};V^{(1)}\right)}{\Delta\left(U^{(2)};V^{(1)}\right)\Delta\left(V^{(2)};U^{(1)}\right)}
	\cdot
	\frac{q_{\hat z_1,\RR}\left(U^{(2)}\right)
	q_{\hat z_2,\RR}\left(U^{(1)}\right)}{
q_{\hat z_1,\RR}\left(V^{(2)}\right)
q_{\hat z_2,\RR}\left(V^{(1)}\right)}.
\end{equation*}
We need to further rewrite the above expressions so that we can apply Proposition~\ref{prop:Cauchy_summation_identity2} later.
Denote
\begin{equation*}
	\mathfrak{h}(w;z):= \begin{dcases}
		q_{z,\RR}(w)/w^N, & w\in \mathcal{D}_\LL,\\
		q_{z,\LL}(w)/(w+1)^{L-N}, & w\in\mathcal{D}_\RR.
	\end{dcases}
\end{equation*}
It is easy to check that $\mathfrak{h}(w;z)$ is analytic and nonzero for $w\in\mathcal{D}_\LL\cup\mathcal{D}_\RR$ and for $z\in\mathbb{D}(\rr)$. Especially we have $\mathfrak{h}(w;0)=1$ for all $w\in\mathcal{D}_\LL\cup\mathcal{D}_\RR$. See equation (5.5) in \cite{Liu19} and the discussions below.

One can write (see equation (4.51) of \cite{Baik-Liu17})
\begin{equation*}
	q'_{z,\RR}(v) = \frac{v^N}{J(v)\mathfrak{h}(v;z)},\quad v\in\mathcal{R}_{z,\RR}.
\end{equation*}
and (see (4.49) of \cite{Baik-Liu17})
\begin{equation*}
	q_{z,\RR}(v') = \frac{z'-z}{q_{z,\LL}(v')} =\frac{z'-z}{(v'+1)^{L-N}\mathfrak{h}(v';z)},\quad v'\in\mathcal{R}_{z',\RR}.
\end{equation*}
Note that $\Delta\left(\mathcal{R}_{\hat z_2,\RR};\mathcal{R}_{\hat z_1,\RR}\right) = q_{\hat z_1,\RR}\left(\mathcal{R}_{\hat z_2,\RR}\right)$. After inserting all these formulas and simplifying the expression, we end up with
\begin{equation}
	\label{eq:aux_05}
	\begin{split}
&	J\left(W^{(1)}\right)J\left(W^{(2)}\right)
	\hat f_1\left(W^{(1)}\right)
	\hat f_2\left(W^{(2)}\right)
	\frac{ \left( \Delta\left(W^{(1)}\right) \right)^2 
		\left( \Delta\left(W^{(2)}\right) \right)^2
	}
	{ \Delta\left(W^{(2)};W^{(1)}\right)
	}=\mathcal{K}(\hat z_2,\hat z_1)\cdot \frac{\hat z_1^n \hat z_2^{N-n}}{(\hat z_2-\hat z_1)^N}\\
&\quad\cdot \left[  
       \prod_{k=1}^2
            \frac{(\Delta(U^{(k)}))^2(\Delta(V^{(k)}))^2}
                 {(\Delta(U^{(k)};V^{(k)}))^2}
            \cdot
            \frac{f_k(U^{(k)};s_k)}{f_k(V^{(k)};s_k)}
            \cdot
              \left( \mathfrak{h}(U^{(k)};\hat z_k)
              \right)^2
            \cdot
              \left( \mathfrak{h}(V^{(k)};\hat z_k)
              \right)^2
            \cdot 
              J(U^{(k)}) J(V^{(k)})
\right]\\
&\quad \cdot
       \left[
           \frac{
           	     \Delta(U^{(2)};V^{(1)}) \Delta (V^{(2)}; U^{(1)})
                }                 
                {
                 \Delta(U^{(2)};U^{(1)}) \Delta (V^{(2)}; V^{(1)})
                }
            \cdot
            \frac{
            	  (1-\hat z_{2}/\hat z_1)^{n_1}
            	  (1-\hat z_{1}/\hat z_2)^{n_2}
                 }
                 {
                  \mathfrak{h}(U^{(2)};\hat z_1)
                  \mathfrak{h}(V^{(2)};\hat z_1)	
                  \mathfrak{h}(U^{(1)};\hat z_2)
                  \mathfrak{h}(V^{(1)};\hat z_2)
                 }
       \right],
\end{split}
\end{equation}
where the functions $f_k(w;s_k)=\hat f_k(w) w^N$, $k=1,2$, are defined in~\eqref{eq:def_f}, and
\begin{equation}
	\begin{split}
	\mathcal{K}(\hat z_2,\hat z_1)&= \frac{1}{\hat z_1^{n}}\prod_{v\in\mathcal{R}_{\hat z_1,\RR}}\frac{(v+1)^{-m}v^{n}e^{s_1v}}{\mathfrak{h}(v;\hat z_1)}\cdot
	\frac{1}{\hat z_2^{N-n}}\prod_{v\in\mathcal{R}_{\hat z_2,\RR}}\frac{(v+1)^{-M+m+L-N}v^{N-n}e^{s_2v}}{\mathfrak{h}(v;\hat z_2)/\mathfrak{h}(v;\hat z_1)}\\
	&=(-1)^{N(L-1)}\prod_{v\in\mathcal{R}_{\hat z_1,\RR}}\frac{(v+1)^{-m}e^{s_1v}}{\mathfrak{h}(v;\hat z_1)}
	\prod_{u\in\mathcal{R}_{\hat z_1,\LL}}\frac{1}{u^n}
	\prod_{v\in\mathcal{R}_{\hat z_2,\RR}}\frac{(v+1)^{-M+m+L-N}e^{s_2v}}{\mathfrak{h}(v;\hat z_2)/\mathfrak{h}(v;\hat z_1)}
	\prod_{u\in\mathcal{R}_{\hat z_2,\LL}}\frac{1}{u^{N-n}}.
	\end{split}
\end{equation}
We observe that $\mathcal{K}(\hat z_2,\hat z_1)$ is analytic for both $\hat z_2\in\mathbb{D}_\rr$ and $\hat z_1\in\mathbb{D}_\rr$ since $\mathfrak{h}$ is analytic and nonzero, and $z^{-1}\prod_{v\in\mathcal{R}_{z,\RR}}v = (-1)^{L-1} \prod_{u\in\mathcal{R}_{z,\LL}} u^{-1}$ is analytic for $z\in\mathbb{D}_\rr$. Moreover, we have $\mathcal{K}(0,0)=1$.

As we mentioned before, there is an extra factor in the summand of~\eqref{eq:sum_rewriting} which comes from~\eqref{eq:calA2},
\begin{equation*}
	\hat H\left(W^{(1)};W^{(2)}\right) +z_0\prod_{i=1}^N\frac{w_{i}^{(2)}}{w_i^{(1)}}\hat H\left(W^{(2)};W^{(1)}\right).
\end{equation*}
Here $\hat H$ is defined in~\eqref{eq:hat_H}. Recall that $\{w_i^{(k)}:1\le i\le N\}=\mathcal{R}_{\hat z_k,\RR}\cup \{u_i^{(k)}:1\le i\le n_k\} \setminus \{v_i^{(k)}:1\le i\le n_k\}$. We write, for each $k,k'\in\{1,2\}$,
\begin{equation*}
	\sum_{i=1}^N \left(w_i^{(k)} -w_i^{(k')}\right)= \sum_{i_k=1}^{n_k} \left(u_{i_k}^{(k)} -v_{i_k}^{(k)}\right) -\sum_{i'_{k'}=1}^{n_{k'}} \left(u_{i'_{k'}}^{(k')} -v_{i'_{k'}}^{(k')}\right) +\mathcal{S}_1(\hat z_k) -\mathcal{S}_1(\hat z_{k'})
\end{equation*}
and
\begin{equation*}
	\sum_{i=1}^N \left((w_i^{(k)})^2 -(w_i^{(k')})^2\right)= \sum_{i_k=1}^{n_k} \left((u_{i_k}^{(k)})^2 -(v_{i_k}^{(k)})^2\right) -\sum_{i'_{k'}=1}^{n_{k'}} \left((u_{i'_{k'}}^{(k')})^2 -(v_{i'_{k'}}^{(k')})^2\right) +\mathcal{S}_2(\hat z_k) -\mathcal{S}_2(\hat z_{k'}),
\end{equation*}
where
\begin{equation}
	\mathcal{S}_k(\hat z):=\sum_{v\in\mathcal{R}_{\hat z,\RR}}v^k,\quad k=1,2
\end{equation}
is analytic in $\hat z\in\mathbb{D}_\rr$. Moreover, it is easy to see that $\mathcal{S}_k(0)=0$ for both $k=1,2$. We also write
\begin{equation}
	z_0\prod_{i=1}^N\frac{w_i^{(2)}}{w_i^{(1)}}=\frac{\hat z_1}{\hat z_2}\prod_{i=1}^N\frac{w_i^{(2)}}{w_i^{(1)}}= \prod_{i_1=1}^{n_1}\frac{v_{i_1}^{(1)}}{u_{i_1}^{(1)}}\prod_{i_2=1}^{n_2} \frac{u_{i_2}^{(2)}}{ v_{i_2}^{(2)}}\cdot \frac{\pi (\hat z_2)}{\pi (\hat z_1)},
\end{equation}
where
\begin{equation*}
	\pi(\hat z):= \frac{1}{\hat z}\prod_{v\in\mathcal{R}_{\hat z,\RR}} v = \frac{(-1)^{L-1}}{\prod_{u\in\mathcal{R}_{\hat z,\LL}} u}
\end{equation*}
is analytic in $\mathbb{D}_\rr$. Moreover, it is easy to see that $\pi(0) =(-1)^{N-1}$.

Combing the above calculations we have
\begin{equation}
	\label{eq:aux_06}
	\hat H\left(W^{(1)};W^{(2)}\right) +z_0\prod_{i=1}^N\frac{w_{i}^{(2)}}{w_i^{(1)}}\hat H\left(W^{(1)};W^{(2)}\right)
	=\tilde H(U^{(1)},U^{(2)};V^{(1)},V^{(2)};\hat z_1,\hat z_2)
\end{equation}
for some function $\tilde H$ which is analytic for all $u_{i_1}^{(1)},u_{i_2}^{(2)}\in\mathcal{D}_\LL$, $v_{i_1}^{(1)},v_{i_2}^{(2)}\in\mathcal{D}_\RR\setminus\{0\}$, $1\le i_1\le n_1, 1\le i_2\le n_2$, and for $\hat z_1,\hat z_2\in\mathbb{D}_\rr$. Moreover, we have
\begin{equation}
	\tilde H(U^{(1)},U^{(2)};V^{(1)},V^{(2)};0,0) = H(U^{(1)},U^{(2)};V^{(1)},V^{(2)}),
\end{equation}
where $H$ is defined in~\eqref{eq:def_H}.

\bigskip
Now we combine~\eqref{eq:aux_05} and~\eqref{eq:aux_06}, and note~\eqref{eq:sum_rewriting}. Note $\hat z_1/\hat z_2 =z_0$. We have
\begin{equation}
	\label{eq:aux_07}
	\begin{split}
		&\frac{1}{(N!)^2}\mathcal{G}(z_1,z_0)\frac{(1-z_0)^N}{z_0^n}\\
		&=\mathcal{K}(\hat z_2,\hat z_1)\sum_{n_1,n_2=0}^N \frac{(1-z_0^{-1})^{n_1}
			(1-z_0)^{n_2}}{(n_1!)^2(n_2!)^2}\sum_{\substack{U^{(2)} \in \mathcal{R}_{\hat z_2,\LL}^{n_2}\\ V^{(2)} \in \mathcal{R}_{\hat z_2,\RR}^{n_2}}}\sum_{\substack{U^{(1)} \in \mathcal{R}_{\hat z_1,\LL}^{n_1}\\ V^{(1)} \in \mathcal{R}_{\hat z_1,\RR}^{n_1}}}
		\Ch(U^{(2)};U^{(1)})\Ch(V^{(2)};V^{(1)})\\
		&\quad\cdot \left[  
		\prod_{k=1}^2
		\frac{(\Delta(U^{(k)}))(\Delta(V^{(k)}))}
		{(\Delta(U^{(k)};V^{(k)}))^2}
		\cdot
		\frac{f_k(U^{(k)};s_k)}{f_k(V^{(k)};s_k)}
		\cdot
		\left( \mathfrak{h}(U^{(k)};\hat z_k)
		\right)^2
		\cdot
		\left( \mathfrak{h}(V^{(k)};\hat z_k)
		\right)^2
		\cdot 
		J(U^{(k)}) J(V^{(k)})
		\right]\\
		&\quad \cdot 
		\left[
		{
			\Delta(U^{(2)};V^{(1)}) \Delta (V^{(2)}; U^{(1)})
		}        
		\cdot
		\frac{
			\tilde H(U^{(1)},U^{(2)};V^{(1)},V^{(2)};\hat z_1,\hat z_2)
		}
		{
			\mathfrak{h}(U^{(2)};\hat z_1)
			\mathfrak{h}(V^{(2)};\hat z_1)	
			\mathfrak{h}(U^{(1)};\hat z_2)
			\mathfrak{h}(V^{(1)};\hat z_2)
		}
		\right].
	\end{split}
\end{equation}

\subsubsection{Completing the proof}

Now we are ready to complete the proof. We will take $z_1\to 0$ on both sides of~\eqref{eq:aux_07}. Recall that we have already proven that $\mathcal{G}(z_1,z_0)$ is analytic for $(z_1,z_0)\in\mathbb{D}_\rr\times\mathbb{D}$ and $\mathcal{G}(0,z_0)$ is well defined. For the right hand side, recall $\hat z_2=z_1$ and $\hat z_1=z_1z_0$. When $z_1\to 0$, both $\hat z_1$ and $\hat z_2$ go to $0$. We also recall $\mathcal{K}(0,0)=1$.

For the summand over $U^{(2)},V^{(2)},U^{(1)},V^{(1)}$, it is a Cauchy type summation as we discussed in Proposition~\ref{prop:Cauchy_summation_identity2}. Our previous discussions on the functions $\mathfrak{h}$ and $\tilde H$ implies that this summand satisfies the analyticity assumption. The proof that $q(w)$ dominates the corresponding factor in this summand is also similar to the previous case discussed in Section~\ref{sec:rewriting_3_5}. The only minor difference is that we have a factor $\prod_{i_1}v_{i_1^{(1)}}\prod_{i_2}(v_{i_2}^{(2)})^{-1}$ in $\tilde H$ but the proof does not change even with this factor. Hence we know that this summation is also analytic for $(z_1,z_0)\in\mathbb{D}_\rr\times\mathbb{D}$. Moreover, by inserting $z_1=0$ in the equation, we obtain
\begin{equation}
	\label{eq:aux_08}
	\begin{split}
		&\frac{1}{(N!)^2}\mathcal{G}(0,z_0)\frac{(1-z_0)^N}{z_0^n}=\sum_{n_1,n_2=0}^N \frac{(1-z_0^{-1})^{n_1}
			(1-z_0)^{n_2}}{(n_1!)^2(n_2!)^2}\\
		&\quad \prod_{i_1=1}^{n_1}     
		\left(  \frac{1}{1-{z_0}} \int_{\Sigma_{\LL,\inn}} 	    
	\ddbar{u^{(1)}_{i_1}}{}   
-\frac{{z_0}}{1-{z_0}} \int_{\Sigma_{\LL,\out}} \ddbar{u^{(1)}_{i_1}}{}
\right)
\left(  \frac{1}{1-{z_0}}  \int_{\Sigma_{\RR,\inn}}  
\ddbar{v^{(1)}_{i_1}}{}
-\frac{{z_0}}{1-{z_0}}  \int_{\Sigma_{\RR,\out}}
\ddbar{v^{(1)}_{i_1}}{}
\right)\\
&\quad \cdot 
\prod_{i_2=1}^{n_2}  
\int_{\Sigma_\LL} \ddbar{u^{(2)}_{i_2}}{} 
\int_{\Sigma_\LL} \ddbar{v^{(2)}_{i_2}}{} 
\Ch(U^{(2)};U^{(1)})\Ch(V^{(2)};V^{(1)}) \prod_{k=1}^2
\frac{(\Delta(U^{(k)}))(\Delta(V^{(k)}))}
{(\Delta(U^{(k)};V^{(k)}))^2}
\cdot
\frac{f_k(U^{(k)};s_k)}{f_k(V^{(k)};s_k)}  \\
			&\quad \cdot 
		{
			\Delta(U^{(2)};V^{(1)}) \Delta (V^{(2)}; U^{(1)})
		}        
		\cdot
		{
			 H(U^{(1)},U^{(2)};V^{(1)},V^{(2)}).
		}		
	\end{split}
\end{equation}
Inserting it in~\eqref{eq:aux_09} and replacing $n_0,n_1$ by $k_1,k_2$, we obtain
\begin{equation}
	\label{eq:aux_30}
	\begin{split}
		&p(s_1,s_2;m,n,M,N)
		=\oint_0\ddbar{\mathrm{z}}{\mathrm{(1-z)^2}}\sum_{k_1,k_2\ge 0}
		\frac{1}{(k_1!k_2!)^2}\\
		&\quad
		\prod_{i_1=1}^{k_1}     
		\left(  \frac{1}{1-\mathrm{z}} \int_{\Sigma_{\LL,\inn}} 	    
		\ddbar{u^{(1)}_{i_1}}{}   
		-\frac{\mathrm{z}}{1-\mathrm{z}} \int_{\Sigma_{\LL,\out}} \ddbar{u^{(1)}_{i_1}}{}
		\right)
		\left(  \frac{1}{1-\mathrm{z}}  \int_{\Sigma_{\RR,\inn}}  
		\ddbar{v^{(1)}_{i_1}}{}
		-\frac{\mathrm{z}}{1-\mathrm{z}}   \int_{\Sigma_{\RR,\out}}
		\ddbar{v^{(1)}_{i_1}}{}
		\right)
		\\
		&\quad \cdot 
		\prod_{i_2=1}^{k_2}  
		\int_{\Sigma_\LL} \ddbar{u^{(2)}_{i_2}}{} 
		\int_{\Sigma_\RR} \ddbar{v^{(2)}_{i_2}}{} 
		\cdot    
		\left(1-\mathrm{z}\right)^{k_2}
		\left(1-\frac{1}{\mathrm{z}}\right)^{k_1}
		\cdot 
		\frac{f_1(U^{(1)};s_1)f_2(U^{(2)};s_2)}{f_1(V^{(1)};s_1)f_2(V^{(2)};s_2)}
		\cdot
		H(U^{(1)},U^{(2)};V^{(1)},V^{(2)})\\
		&\quad \cdot 
		\prod_{\ell=1}^2
		\frac{\left( \Delta(U^{(\ell)}) \right)^2
			\left( \Delta(V^{(\ell)}) \right)^2
		}
		{\left( \Delta(U^{(\ell)}; V^{(\ell)}) \right)^2}
		\cdot 
		\frac{\Delta(U^{(1)};V^{(2)})\Delta(V^{(1)};U^{(2)})}
		{\Delta(U^{(1)};U^{(2)})\Delta(V^{(1)};V^{(2)})}.
	\end{split}
\end{equation}
Note that when $k_1=0$, the summand is analytic for $\mathrm{z}=0$ hence the integral of $\mathrm{z}$ vanishes. When $k_2=0$, there is no $u^{(2)}_{i_2}$ or $v_{i_2}^{(2)}$ variable, hence the $u_{i_1}^{(1)}$ and $v_{i_1}^{(1)}$ contours can be deformed to $\Sigma_\LL$ and $\Sigma_\RR$ respectively. As a result, the $\mathrm{z}$ integral can be separately written as
\begin{equation*}
	\oint_0\ddbar{\mathrm{z}}{\mathrm{(1-z)^2}}\left(1-\frac{1}{\mathrm{z}}\right)^{k_1}=\begin{dcases}
		-1,& k_1=1,\\
		0,&k_1=0, \text{or }k_1\ge 2.
	\end{dcases}
\end{equation*}
However, it is direct to check that $H(U^{(1)},U^{(2)};V^{(1)},V^{(2)})=0$ when $k_1=1$ and $k_2=0$. Therefore the summand when $k_2=0$ also vanishes. Thus we can replace the sum $\sum_{k_1,k_2\ge 0}$ by $\sum_{k_1,k_2\ge 1}$, and
arrive at the formula~\eqref{eq:def_finite_time_density}.

\section{Asymptotic analysis and proof of Theorem~\ref{thm:limiting_geodesic}}
\label{sec:proof_thm2}
 
 In this section, we will perform asymptotic analysis for the formulas obtained in Theorem~\ref{thm:finite_time_geodesic1} and prove Theorem~\ref{thm:limiting_geodesic}. 
The main technical result of this section is as follows.
\begin{prop}
	\label{prop:asymptotics_density}Suppose $\alpha>0, \gamma\in(0,1)$ are fixed constants. Assume that
	\begin{equation}
		\label{eq:scaling}
		\begin{split}
			M&= [\alpha N],\\
			m&=[\gamma\alpha N + x_1 \alpha^{2/3}(1+\sqrt{\alpha})^{2/3} N^{2/3}],\\
			n&=[\gamma N + x_2 \alpha^{-1/3}(1+\sqrt{\alpha})^{2/3} N^{2/3}],\\
			t_1&=d((1,1),(m,n)) + \mathrm{t}_1\cdot \alpha^{-1/6}(1+\sqrt{\alpha})^{4/3}N^{1/3},\\
			t_2&=d((m+1,n),(M,N)) +\mathrm{t}_2 \cdot 
			\alpha^{-1/6}(1+\sqrt{\alpha})^{4/3}N^{1/3},\\
			t_2'&=d((m,n+1),(M,N)) +\mathrm{t}_2 \cdot 
			\alpha^{-1/6}(1+\sqrt{\alpha})^{4/3}N^{1/3},
		\end{split}
	\end{equation}
for some real numbers $x_1,x_2$. Then
\begin{equation}
	\label{eq:estimate_vertical}
	\begin{split}
		&\prob
		\left(
			(m,n),(m+1,n)\in \gd_{(1,1)}(M,N), L_{(1,1)}(m,n)\ge t_1,
			L_{(m+1,n)}(M,N)\ge t_2
		\right)\\
		&=\alpha^{1/3}(1+\sqrt{\alpha})^{-2/3}N^{-2/3}\int_{\mathrm{t}_1}^\infty \int_{\mathrm{t}_2}^\infty\limp(\mathrm{s}_1,\mathrm{s}_2,\mathrm{x}=x_2-x_1;\gamma)\mathrm{d}{\mathrm{s}_2}\mathrm{d}{\mathrm{s}_2}+\bigO(N^{-1}(\log N)^5),
		\end{split}
\end{equation}
and similarly
\begin{equation}
	\label{eq:estimate_horizontal}
	\begin{split}
		&\prob
		\left(
			(m,n), (m,n+1)\in \gd_{(1,1)}(M,N), L_{(1,1)}(m,n)\ge t_1,  L_{(m,n+1)}(M,N)\ge t'_2
		\right)\\
		&=\alpha^{-2/3}(1+\sqrt{\alpha})^{-2/3}N^{-2/3}\int_{\mathrm{t}_1}^\infty \int_{\mathrm{t}_2}^\infty\limp(\mathrm{s}_1,\mathrm{s}_2,\mathrm{x}=x_2-x_1;\gamma)\mathrm{d}{\mathrm{s}_2}\mathrm{d}{\mathrm{s}_1}+\bigO(N^{-1}(\log N)^5)
	\end{split}
\end{equation}
as $N$ becomes large, and the $\bigO(N^{-1}(\log N)^5)$ errors are uniformly for $x_1,x_2$ in any given compact set and for $\mathrm{t}_1,\mathrm{t}_2$ in any given set with a finite lower bound.	
\end{prop}


The proof of Proposition will be provided later in this section. 
Below we prove Theorem~\ref{thm:limiting_geodesic} assuming Proposition~\ref{prop:asymptotics_density}.

Recall that $\pi$ is an up/left lattice path from $(m,n)$ to $(m',n')$. See Figure~\ref{fig:limit_theorem_proof} for an illustration. We first realize that there are different types of lattice points $(a,b)\in\pi$ depending on whether $(a+1,b)$ and $(a,b+1)$ are on $\pi$ or not. We call $(a,b)\in\pi$ is a horizontal point if $(a,b+1)\notin\pi$, and a vertical point if $(a+1,b)\notin\pi$. Note there are outer corners which are both horizontal and vertical points, and inner corners which are neither horizontal nor vertical points. We also note that an exit point $\mathbf{p}$ must be a horizontal point $\mathbf{p}=(a,b)$ with $\mathbf{p}_+=(a,b+1)$, or a vertical point $\mathbf{p}=(a,b)$ with $\mathbf{p}_+=(a+1,b)$.  We write
\begin{equation}
	\label{eq:riemann_sum}
	\begin{split}
		&\prob
		\left(
		\begin{array}{l}
			\gd_{(1,1)}(M,N)
			\text{ intersects $\pi$, and exits $\pi$ at some point $\mathbf{p}=( a, b)$,} \\
			\mbox{and }L_{(1,1)}(\mathbf{p})\ge t_1=d((1,1),\mathbf{p}) + \mathrm{t}_1\cdot \alpha^{-1/6}(1+\sqrt{\alpha})^{4/3} N^{1/3}, \\
			\mbox{and } L_{\mathbf{p}_+}(M,N)\ge t_2=d(\mathbf{p}_+,(M,N)) + \mathrm{t}_2\cdot \alpha^{-1/6}(1+\sqrt{\alpha})^{4/3} N^{1/3}\\
		\end{array}
		\right)\\
		&=\sum_{(a,b)\in \pi \text{ is a vertical point }}
		\prob
		\left(
		\begin{array}{l}
			(a,b)\in \gd_{(1,1)}(M,N)
			\text{ and } (a+1,b)\in \gd_{(1,1)}{(M,N)},\\
			\mbox{and }L_{(1,1)}(a,b)\ge d((1,1),(a,b)) + \mathrm{t}_1\cdot \alpha^{-1/6}(1+\sqrt{\alpha})^{4/3} N^{1/3}, \\
			\mbox{and } L_{(a+1,b)}(M,N)\ge d((a+1,b),(M,N)) + \mathrm{t}_2\cdot \alpha^{-1/6}(1+\sqrt{\alpha})^{4/3} N^{1/3}\\
		\end{array}
		\right)\\	
		&
		+\sum_{(a,b)\in \pi \text{ is a horizontal point }}
		\prob
		\left(
		\begin{array}{l}
			(a,b)\in \gd_{(1,1)}(M,N)
			\text{ and } (a,b+1)\in \gd_{(1,1)}{(M,N)},\\
			\mbox{and }L_{(1,1)}(a,b)\ge d((1,1),(a,b)) + \mathrm{t}_1\cdot \alpha^{-1/6}(1+\sqrt{\alpha})^{4/3} N^{1/3}, \\
			\mbox{and } L_{(a,b+1)}(M,N)\ge d((a,b+1),(M,N)) + \mathrm{t}_2\cdot \alpha^{-1/6}(1+\sqrt{\alpha})^{4/3} N^{1/3}\\
		\end{array}
		\right).	
	\end{split}
\end{equation}

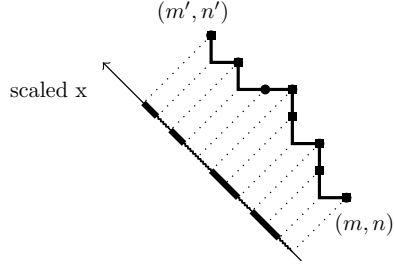
\begin{figure}[h]
	\centering\scalebox{1.2}{	\begin{tikzpicture}
			\draw[line width=1pt] (1.2,2.7)--(1.2,2.4)--(1.5,2.4)--(1.5,2.1)--(2.1,2.1)--(2.1,1.5)--(2.4,1.5)--(2.4,0.9)--(2.7,0.9);
			\draw[<-] (0,2.4)--(2.2,0.2); 
			\draw[line width=2pt] (0.45,1.95)--(0.6,1.8);
			\draw[decoration={aspect=0.03, segment length=0.4mm, amplitude=0.12mm,coil},decorate] (0.6,1.8)--(0.75,1.65);
			\draw[line width=2pt] (0.75,1.65)--(0.9,1.5);
			\draw[decoration={aspect=0.03, segment length=0.4mm, amplitude=0.12mm,coil},decorate] (0.9,1.5)--(1.2,1.2);
			\draw[line width=2pt] (1.2,1.2)--(1.5,0.9);
			\draw[decoration={aspect=0.03, segment length=0.4mm, amplitude=0.12mm,coil},decorate] (1.5,0.9)--(1.65,0.75);
			\draw[line width=2pt] (1.65,0.75)--(1.95,0.45);
			\draw[decoration={aspect=0.03, segment length=0.4mm, amplitude=0.12mm,coil},decorate] (1.95,0.45)--(2.1,0.3);
			\draw[dotted] (1.2,2.7)--(0.45,1.95);
			\draw[dotted] (1.2,2.4)--(0.6,1.8);
			\draw[dotted] (1.5,2.4)--(0.75,1.65);
			\draw[dotted] (1.5,2.1)--(0.9,1.5);
			\draw[dotted] (1.8,2.1)--(1.05,1.35);
			\draw[dotted] (2.1,2.1)--(1.2,1.2);
			\draw[dotted] (2.1,1.8)--(1.35,1.05);
			\draw[dotted] (2.1,1.5)--(1.5,0.9);
			\draw[dotted] (2.4,1.5)--(1.65,0.75);
			\draw[dotted] (2.4,1.2)--(1.8,0.6);
			\draw[dotted] (2.4,0.9)--(1.95,0.45);
			\draw[dotted] (2.7,0.9)--(2.1,0.3);
			\node at (2.7,0.9)[circle,fill,inner sep=1.1pt]{};
			\node at (2.7,0.9)[rectangle,fill,inner sep=1.34pt]{};
			\node at (2.4,1.2)[rectangle,fill,inner sep=1.34pt]{};
			\node at (2.4,1.5)[rectangle,fill,inner sep=1.34pt]{};
			\node at (2.4,1.5)[circle,fill,inner sep=1.1pt]{};
			\node at (2.1,1.8)[rectangle,fill,inner sep=1.34pt]{};
			\node at (2.1,2.1)[rectangle,fill,inner sep=1.34pt]{};
			\node at (2.1,2.1)[circle,fill,inner sep=1.1pt]{};
			\node at (1.8,2.1)[circle,fill,inner sep=1.1pt]{};
			\node at (1.5,2.4)[circle,fill,inner sep=1.1pt]{};
			\node at (1.5,2.4)[rectangle,fill,inner sep=1.34pt]{};
			\node at (1.2,2.7)[rectangle,fill,inner sep=1.34pt]{};
			\node at (1.2,2.7)[circle,fill,inner sep=1.1pt]{};
			\node[scale=0.7] at (2.9,0.6)  {$(m,n)$};
			\node[scale=0.7] at (1,2.95)  {$(m',n')$};
			\node[scale=0.7] at (-0.6,2.1) {scaled $\mathrm{x}$};
	\end{tikzpicture}}
	\caption{An illustration of the sum~\eqref{eq:riemann_sum}. The square-shaped points are vertical points, and the round-shaped points are horizontal points. The sum can be viewed as a Riemann sum along the axis $\mathrm{x}$, where the horizontal points contribute to the spring parts and the vertical points contribute to the thick part.}\label{fig:limit_theorem_proof}
\end{figure}

Now we apply Proposition~\ref{prop:asymptotics_density} and view the right hand side of~\eqref{eq:riemann_sum} as a Riemann sum of the quantity $\int_{\mathrm{t}_1}^\infty \int_{\mathrm{t}_2}^\infty\limp(\mathrm{s}_1,\mathrm{s}_2,\mathrm{x};\gamma)\mathrm{d}{\mathrm{s}_2}\mathrm{d}{\mathrm{s}_1}$ over an interval $\mathrm{x}\in [x_2-x_1,x'_2-x'_1]$, plus an error terms $\bigO(N^{-1}(\log N)^5)\times \bigO(N^{2/3})=\bigO(N^{-1/3}(\log N)^5)$. See Figure~\ref{fig:limit_theorem_proof} for an illustration. It is easy to see from the definition that $\int_{\mathrm{t}_1}^\infty \int_{\mathrm{t}_2}^\infty\limp(\mathrm{s}_1,\mathrm{s}_2,\mathrm{x};\gamma)\mathrm{d}{\mathrm{s}_2}\mathrm{d}{\mathrm{s}_1}$ is continuous in $\mathrm{x}$. Thus the Riemman sum converges to the desired integral in~\eqref{eq:limit_geodesic_distribution}, and we complete the proof of Theorem~\ref{thm:limiting_geodesic}.

\bigskip

The remaining part of this section is the proof of Proposition~\ref{prop:asymptotics_density}. We first realize that~\eqref{eq:estimate_horizontal} and~\eqref{eq:estimate_vertical} are equivalent. In fact, if we switch rows and columns and replace $\alpha$ by $\alpha^{-1}$ in the equation~\eqref{eq:estimate_horizontal}, we obtain~\eqref{eq:estimate_vertical} with $-\mathrm{x}$ instead of $\mathrm{x}$ appearing on the right hand side. Note that $\limp(\mathrm{s}_1,\mathrm{s}_2,\mathrm{x};\gamma)=\limp(\mathrm{s}_1,\mathrm{s}_2,-\mathrm{x};\gamma)$, see Remark~\ref{rmk:symmetry_T}. We hence obtain the equivalence of~\eqref{eq:estimate_horizontal} and~\eqref{eq:estimate_vertical}. It remains to prove one equation~\eqref{eq:estimate_vertical}.

Using Theorem~\ref{thm:finite_time_geodesic1}, we write the left hand side of~\eqref{eq:estimate_vertical} as
\begin{equation}
	\label{eq:left_vertical}
	\prob
	\left(
	\begin{array}{l}
		(m,n), (m+1,n)\in \gd_{(1,1)}{(M,N)},\\
		\mbox{and }L_{(1,1)}(m,n)\ge t_1, \\
		\mbox{and } L_{(m+1,n)}(M,N)\ge t_2\\
	\end{array}
	\right)=\oint_0\ddbar{\mathrm{z}}{\mathrm{(1-z)^2}}\sum_{k_1,k_2\ge 1}
	\frac{1}{(k_1!k_2!)^2}
	\hat T_{k_1,k_2}(\mathrm{z};t_1,t_2;m,n,M,N),
\end{equation}
where
\begin{equation}
	\label{eq:hat_T}
	\begin{split}
	&\hat T_{k_1,k_2}(\mathrm{z};t_1,t_2;m,n,M,N)\\
	&=\int_{t_1}^\infty\int_{t_2}^\infty T_{k_1,k_2}(\mathrm{z};s_1,s_2;m,n,M,N)\mathrm{d}s_2\mathrm{d}s_1\\
	&=
	\prod_{i_1=1}^{k_1}     
	\left(  \frac{1}{1-\mathrm{z}} \int_{\Sigma_{\LL,\inn}} 	    
	\ddbar{u^{(1)}_{i_1}}{}   
	-\frac{\mathrm{z}}{1-\mathrm{z}} \int_{\Sigma_{\LL,\out}} \ddbar{u^{(1)}_{i_1}}{}
	\right)
	\left(  \frac{1}{1-\mathrm{z}}  \int_{\Sigma_{\RR,\inn}}  
	\ddbar{v^{(1)}_{i_1}}{}
	-\frac{\mathrm{z}}{1-\mathrm{z}}   \int_{\Sigma_{\RR,\out}}
	\ddbar{v^{(1)}_{i_1}}{}
	\right)
	\\
	&\quad \cdot 
	\prod_{i_2=1}^{k_2}  
	\int_{\Sigma_\LL} \ddbar{u^{(2)}_{i_2}}{} 
	\int_{\Sigma_\RR} \ddbar{v^{(2)}_{i_2}}{} 
	\cdot    
	\left(1-\mathrm{z}\right)^{k_2}
	\left(1-\frac{1}{\mathrm{z}}\right)^{k_1}
	\cdot 
	\frac{f_1(U^{(1)};t_1)f_2(U^{(2)};t_2)}{f_1(V^{(1)};t_1)f_2(V^{(2)};t_2)}
	\cdot \frac{1}{\prod_{\ell=1}^2\sum_{i_\ell=1}^{k_\ell}(u_{i_\ell}^{(\ell)} -v_{i_\ell}^{(\ell)})}
	\\
	&\quad \cdot H(U^{(1)},U^{(2)};V^{(1)},V^{(2)})\cdot 
	\prod_{\ell=1}^2
	\frac{\left( \Delta(U^{(\ell)}) \right)^2
		\left( \Delta(V^{(\ell)}) \right)^2
	}
	{\left( \Delta(U^{(\ell)}; V^{(\ell)}) \right)^2}
	\cdot 
	\frac{\Delta(U^{(1)};V^{(2)})\Delta(V^{(1)};U^{(2)})}
	{\Delta(U^{(1)};U^{(2)})\Delta(V^{(1)};V^{(2)})},
	\end{split}
\end{equation}
with the functions $f_1(w;t_1)$ and $f_2(w;t_2)$ defined in~\eqref{eq:def_f}, and the function $H$ defined by~\eqref{eq:def_H}. We remark that in the above equation we evaluated the integral over $s_1$ and $s_2$ using the fact $\mathrm{Re} u_{i_\ell}^{(\ell)}<\mathrm{Re} v_{i_\ell}^{(\ell)}$ due to the order of the contours.

Similarly, we can write 
\begin{equation}
	\label{eq:right_vertical}
	\begin{split}
		\int_{\mathrm{t}_1}^\infty \int_{\mathrm{t}_2}^\infty\limp(\mathrm{s}_1,\mathrm{s}_2,\mathrm{x};\gamma)\mathrm{d}{\mathrm{s}_2}\mathrm{d}{\mathrm{s}_1}=\oint_0\ddbar{\mathrm{z}}{\mathrm{(1-z)^2}}\sum_{k_1,k_2\ge 1}
		\frac{1}{(k_1!k_2!)^2}
		\mathrm{\hat T}_{k_1,k_2}(\mathrm{z};\mathrm{t}_1,\mathrm{t}_2,\mathrm{x};\gamma)
	\end{split}
\end{equation}
with
\begin{equation}
	\label{eq:hat_rT}
	\begin{split}
		&\mathrm{\hat T}_{k_1,k_2}(\mathrm{z};\mathrm{t}_1,\mathrm{t}_2,\mathrm{x};\gamma)=\int_{\mathrm{t}_1}^\infty \int_{\mathrm{t}_2}^\infty\mathrm{ T}_{k_1,k_2}(\mathrm{z};\mathrm{s}_1,\mathrm{s}_2,\mathrm{x};\gamma)\mathrm{d}{\mathrm{s}_2}\mathrm{d}{\mathrm{s}_1}\\
		&=
		(-1)^{k_1+k_2}\prod_{i_1=1}^{k_1}     
		\left(  \frac{1}{1-\mathrm{z}} \int_{\Gamma_{\LL,\inn}} 	    
		\ddbar{\xi^{(1)}_{i_1}}{}   
		-\frac{\mathrm{z}}{1-\mathrm{z}} \int_{\Gamma_{\LL,\out}} \ddbar{\xi^{(1)}_{i_1}}{}
		\right)
		\left(  \frac{1}{1-\mathrm{z}}  \int_{\Gamma_{\RR,\inn}}  
		\ddbar{\eta^{(1)}_{i_1}}{}
		-\frac{\mathrm{z}}{1-\mathrm{z}}   \int_{\Gamma_{\RR,\out}}
		\ddbar{\eta^{(1)}_{i_1}}{}
		\right)
		\\
		&\quad \cdot 
		\prod_{i_2=1}^{k_2}  
		\int_{\Gamma_\LL} \ddbar{\xi^{(2)}_{i_2}}{} 
		\int_{\Gamma_\RR} \ddbar{\eta^{(2)}_{i_2}}{} 
		\cdot    
		\left(1-\mathrm{z}\right)^{k_2}
		\left(1-\frac{1}{\mathrm{z}}\right)^{k_1}
		\cdot 
		\frac
		{\mathrm{f}_1(\boldsymbol{\xi}^{(1)};\mathrm{t}_1)
			\mathrm{f}_2(\boldsymbol{\xi}^{(2)};\mathrm{t}_2)}
		{\mathrm{f}_1(\boldsymbol{\eta}^{(1)};\mathrm{t}_1)
			\mathrm{f}_2(\boldsymbol{\eta}^{(2)};\mathrm{t}_2)}
		\cdot
		\frac{1}{\prod_{\ell=1}^2\sum_{i_\ell=1}^{k_\ell}(\xi_{i_\ell}^{(\ell)} -\eta_{i_\ell}^{(\ell)})}
		\\
		&\quad \cdot
		\mathrm{H}(\boldsymbol{\xi}^{(1)},\boldsymbol{\xi}^{(2)};
		\boldsymbol{\eta}^{(1)},\boldsymbol{\eta}^{(2)}) \cdot 
		\prod_{\ell=1}^2
		\frac{\left( \Delta(\boldsymbol{\xi}^{(\ell)}) \right)^2
			\left( \Delta(\boldsymbol{\eta}^{(\ell)}) \right)^2
		}
		{\left( \Delta(\boldsymbol{\xi}^{(\ell)}; \boldsymbol{\eta}^{(\ell)}) \right)^2}
		\cdot 
		\frac{\Delta(\boldsymbol{\xi}^{(1)};\boldsymbol{\eta}^{(2)})
			\Delta(\boldsymbol{\eta}^{(1)};\boldsymbol{\xi}^{(2)})
		}
		{\Delta(\boldsymbol{\xi}^{(1)};\boldsymbol{\xi}^{(2)})
			\Delta(\boldsymbol{\eta}^{(1)};\boldsymbol{\eta}^{(2)})
		},
	\end{split}
\end{equation}
where  the functions $\mathrm{f}_1(\zeta;\mathrm{t})$ and $\mathrm{f}_2(\zeta;\mathrm{t})$ are defined in~\eqref{eq:def_rmf}, and the function $\mathrm{H}$ is defined in~\eqref{eq:def_rmH}. We remark that in the above calculations we exchanged the integrals and the summations. We need to justify that they are exchangeable. It is tedious but not hard to check that
\begin{equation}
	\label{eq:aux_10}
	\int_{t_1}^\infty\int_{t_2}^\infty \sum_{k_1,k_2\ge 1}\frac{1}{(k_1!k_2!)^2} \left|T_{k_1,k_2}(\mathrm{z};s_1,s_2;m,n,M,N)\right||\mathrm{d}s_2||\mathrm{d}s_1|<C(\rz)<\infty
\end{equation}
and 
\begin{equation}
	\label{eq:aux_11}
	\int_{\mathrm{t}_1}^\infty\int_{\mathrm{t}_2}^\infty \sum_{k_1,k_2\ge 1}\frac{1}{(k_1!k_2!)^2} \left|\mathrm{T}_{k_1,k_2}(\mathrm{z};\mathrm{s}_1,\mathrm{s}_2,\mathrm{x};\gamma)\right||\mathrm{d}\mathrm{s}_2||\mathrm{d}\mathrm{s}_1|<\mathrm{C}(\rz)<\infty
\end{equation}
for some constants $C(\rz)$ and $\mathrm{C}(\rz)$ which only depend on $\mathrm{z}$. Moreover, $C(\rz)$ and $\mathrm{C}(\rz)$ are both continuous in $\rz$ (except at $\rz=0$ or $-1$) hence they are uniformly bounded for $|\rz|=$constant that lies in $(0,1)$.  Here we omit the proof of these inequalities since it is similar to that of Lemma~\ref{lm:hat_rt_uniform_bounds}. Using these inequalities we verify that the exchanges of integrals and summations are valid and equations~\eqref{eq:left_vertical} and~\eqref{eq:right_vertical} hold.

To proceed, we need to compare~\eqref{eq:left_vertical} and~\eqref{eq:right_vertical} term by term and estimate their difference. There is a need to see the dependence of the error on the parameters. We will fix the contour of $\rz$ to be a circle with fixed radius $|\rz|\in(0,1)$. We also introduce the following notation.

\begin{notation}
	\label{notation:CC}
	we use the calligraphic font $\CC$ (or $\CC_i$ with some index $i$) to denote a positive constant term (independent of $N$) satisfying the following three conditions:
	\begin{enumerate}[(1)]
		\item $\CC$ is independent of $k_1$ and $k_2$.
		\item $\CC$ is continuous in $\mathrm{z}$.
		\item $\CC$ is continuous in $\mathrm{t}_1$ and $\mathrm{t}_2$, and decays exponentially as $\mathrm{t}_1\to\infty$ or $\mathrm{t}_2\to\infty$.
	\end{enumerate}
\end{notation}

Throughout this whole section, we will use $\CC$ as described in Notation~\ref{notation:CC}, and the regular $C$ as a constant independent of the parameters.

We will show the following two lemmas in subsequent subsections.

\begin{lm}
	\label{lm:hat_rt_uniform_bounds}
	We have the estimate
	\begin{equation*}
		\left|\mathrm{\hat T}_{k_1,k_2}(\mathrm{z};\mathrm{t}_1,\mathrm{t}_2,\mathrm{x};\gamma)\right| \le k_1^{k_1/2}k_2^{k_2/2}(k_1+k_2)^{(k_1+k_2)/2}\CC_1^{k_1+k_2}
	\end{equation*}
	for all $k_1,k_2\ge 1$, where  $\CC_1$ is a positive constant as described in Notation~\ref{notation:CC}.
\end{lm}

\begin{lm}
	\label{lm:estimate_error}
		With the same assumptions as in Proposition~\ref{prop:asymptotics_density}, there is a constant $\CC_2$ as described in Notation~\ref{notation:CC} such that
	\begin{equation}
		\begin{split}
		&\left|N^{2/3}\hat T_{k_1,k_2}(\mathrm{z};t_1,t_2;m,n,M,N)-\alpha^{1/3}(1+\sqrt{\alpha})^{-2/3}\mathrm{\hat T}_{k_1,k_2}(\mathrm{z};\mathrm{t}_1,\mathrm{t}_2,\mathrm{x};\gamma)
		\right|\\
		&\le k_1^{k_1/2}k_2^{k_2/2}(k_1+k_2)^{(k_1+k_2)/2}\CC_2^{k_1+k_2} N^{-1/3}(\log N)^5
		\end{split}
	\end{equation}
for all $k_1,k_2\ge 1$ as $N$ becomes sufficiently large.
\end{lm}

Now we use these two lemmas to prove~\eqref{eq:estimate_vertical}. We first use  and realize that the right hand side of~\eqref{eq:right_vertical} is uniformly bounded  by
\begin{equation*}
	\begin{split}
	&\oint_0\left|\ddbar{\mathrm{z}}{\mathrm{(1-z)^2}}\right|\sum_{k_1,k_2\ge 1}
	\frac{1}{(k_1!k_2!)^2}
	\left|\mathrm{\hat T}_{k_1,k_2}(\mathrm{z};\mathrm{t}_1,\mathrm{t}_2,\mathrm{x};\gamma)\right|\\
	&\le \oint_0\left|\ddbar{\mathrm{z}}{\mathrm{(1-z)^2}}\right|\sum_{k_1,k_2\ge 1}
	\frac{1}{(k_1!k_2!)^2} k_1^{k_1/2}k_2^{k_2/2}(k_1+k_2)^{(k_1+k_2)/2}\CC_1^{k_1+k_2}<\infty,
	\end{split}
\end{equation*}
where the last inequality is due to the Stirling's approximation formula $k!\approx k^ke^{-k}\sqrt{2\pi k}$ for large $k$.

Similarly we know that
\begin{equation*}
	\begin{split}
			&\oint_0\left|\ddbar{\mathrm{z}}{\mathrm{(1-z)^2}}\right|\sum_{k_1,k_2\ge 1}
		\frac{1}{(k_1!k_2!)^2}
		\left| N^{2/3}\hat T_{k_1,k_2}(\mathrm{z};t_1,t_2;m,n,M,N)-\alpha^{1/3}(1+\sqrt{\alpha})^{-2/3}\mathrm{\hat T}_{k_1,k_2}(\mathrm{z};\mathrm{t}_1,\mathrm{t}_2,\mathrm{x};\gamma) \right|\\
		&\le \oint_0\left|\ddbar{\mathrm{z}}{\mathrm{(1-z)^2}}\right|\sum_{k_1,k_2\ge 1}
		\frac{1}{(k_1!k_2!)^2} k_1^{k_1/2}k_2^{k_2/2}(k_1+k_2)^{(k_1+k_2)/2}\CC_2^{k_1+k_2} N^{-1/3}(\log N)^5<\infty
	\end{split}
\end{equation*}
for sufficiently large $N$.

 Combining the above two estimates we also know the right hand side of~\eqref{eq:left_vertical} multiplied by $N^{2/3}$ is also uniformly bounded by the sum of the above two bounds
 \begin{equation*}
 	\begin{split}
 		&N^{2/3}\oint_0\left|\ddbar{\mathrm{z}}{\mathrm{(1-z)^2}}\right|\sum_{k_1,k_2\ge 1}
 		\frac{1}{(k_1!k_2!)^2}
 		\left|\hat T_{k_1,k_2}(\mathrm{z};t_1,t_2;m,n,M,N)\right|\\
 		& \le \oint_0\left|\ddbar{\mathrm{z}}{\mathrm{(1-z)^2}}\right|\sum_{k_1,k_2\ge 1}
 		\frac{1}{(k_1!k_2!)^2} k_1^{k_1/2}k_2^{k_2/2}(k_1+k_2)^{(k_1+k_2)/2}\\
 		&\qquad\cdot\left(\alpha^{1/3}(1+\sqrt{\alpha})^{-2/3}\CC_1^{k_1+k_2}+\CC_2^{k_1+k_2}N^{-1/3}(\log N)^5\right)\\
 		&<\infty.
 	\end{split}
 \end{equation*}

The above estimates imply that we can rewrite, using~\eqref{eq:left_vertical} and~\eqref{eq:right_vertical},
\begin{equation*}
	\begin{split}
		&N^{2/3}\prob
		\left(
		\begin{array}{l}
			(m,n), (m+1,n)\in \gd_{(1,1)}{(M,N)},\\
			\mbox{and }L_{(1,1)}(m,n)\ge t_1, \\
			\mbox{and } L_{(m+1,n)}(M,N)\ge t_2\\
		\end{array}
		\right)
		-\alpha^{1/3}(1+\sqrt{\alpha})^{-2/3}\int_{\mathrm{t}_1}^\infty \int_{\mathrm{t}_2}^\infty\limp(\mathrm{s}_1,\mathrm{s}_2,\mathrm{x};\gamma)\mathrm{d}{\mathrm{s}_2}\mathrm{d}{\mathrm{s}_1}\\
		&=\oint_0\ddbar{\mathrm{z}}{\mathrm{(1-z)^2}}\sum_{k_1,k_2\ge 1}
		\frac{1}{(k_1!k_2!)^2}\left(N^{2/3}\hat T_{k_1,k_2}(\mathrm{z};t_1,t_2;m,n,M,N)-\alpha^{1/3}(1+\sqrt{\alpha})^{-2/3}\mathrm{\hat T}_{k_1,k_2}(\mathrm{z};\mathrm{t}_1,\mathrm{t}_2,\mathrm{x};\gamma)\right),
	\end{split}
\end{equation*}
which is uniformly bounded by, using Lemma~\ref{lm:estimate_error},
\begin{equation*}
	\oint_0\left|\ddbar{\mathrm{z}}{\mathrm{(1-z)^2}}\right|\sum_{k_1,k_2\ge 1}
	\frac{1}{(k_1!k_2!)^2} k_1^{k_1/2}k_2^{k_2/2}(k_1+k_2)^{(k_1+k_2)/2}\CC_2^{k_1+k_2} N^{-1/3}(\log N)^5=\bigO(N^{-1/3}(\log N)^5)
\end{equation*}
for sufficiently large $N$. Thus~\eqref{eq:estimate_vertical} holds. 

\bigskip

It remains to prove the two lemmas~\ref{lm:hat_rt_uniform_bounds} and~\ref{lm:estimate_error}. Note that if we did not have the factors $\frac{1}{\prod_{\ell=1}^2\sum_{i_\ell=1}^{k_\ell}(u_{i_\ell}^{(\ell)} -v_{i_\ell}^{(\ell)})}$ and
$H(U^{(1)},U^{(2)};V^{(1)},V^{(2)})$ in the integrand of $T_{k_1,k_2}(\mathrm{z};t_1,t_2;m,n,M,N)$, and the factors $\frac{1}{\prod_{\ell=1}^2\sum_{i_\ell=1}^{k_\ell}(\xi_{i_\ell}^{(\ell)} -\eta_{i_\ell}^{(\ell)})}$ and
 $\mathrm{H}(\boldsymbol{\xi}^{(1)},\boldsymbol{\xi}^{(2)};
\boldsymbol{\eta}^{(1)},\boldsymbol{\eta}^{(2)})$ in the integrand of $\mathrm{\hat T}_{k_1,k_2}(\mathrm{z};\mathrm{t}_1,\mathrm{t}_2,\mathrm{x};\gamma)$, the right hand sides of both~\eqref{eq:left_vertical} and~\eqref{eq:right_vertical} could be viewed as expansions of Fredholm determinants. They have similar structures as the expansion of the two-time distribution formulas in TASEP, see \cite[Proposition 2.10]{Liu19}. Moreover, the two lemmas above are indeed analogous to Lemmas 7.1 and 7.2 in \cite{Liu19}. So it is not surprising that we can modify the standard asymptotic analysis for Fredholm determinants to prove these two lemmas. However, we do need some tedious calculations to incorporate the extra factors, and much finer estimates in Lemmas~\ref{lm:hat_rt_uniform_bounds} and~\ref{lm:estimate_error} compared with the analogs in \cite{Liu19}. Our proof will also be illustrative to prove similar statements in our follow-up papers.

\bigskip

We will prove the Lemma~~\ref{lm:hat_rt_uniform_bounds} and~\ref{lm:estimate_error} in the following two subsections. 

\subsection{Proof of Lemma~\ref{lm:hat_rt_uniform_bounds}}
\label{sec:proof_lm_uniform_bounds}

In this subsection we prove Lemma~\ref{lm:hat_rt_uniform_bounds}. Some estimates we use here will also appear in the proof of the lemmas~\ref{lm:estimate_error} in the next subsection.

We first estimate the factor
\begin{equation*}
	\mathrm{B}(\boldsymbol{\xi}^{(1)},\boldsymbol{\eta}^{(1)};\boldsymbol{\xi}^{(2)},\boldsymbol{\eta}^{(2)}):=\prod_{\ell=1}^2
	\frac{\left( \Delta(\boldsymbol{\xi}^{(\ell)}) \right)^2
		\left( \Delta(\boldsymbol{\eta}^{(\ell)}) \right)^2
	}
	{\left( \Delta(\boldsymbol{\xi}^{(\ell)}; \boldsymbol{\eta}^{(\ell)}) \right)^2}
	\cdot 
	\frac{\Delta(\boldsymbol{\xi}^{(1)};\boldsymbol{\eta}^{(2)})
		\Delta(\boldsymbol{\eta}^{(1)};\boldsymbol{\xi}^{(2)})
	}
	{\Delta(\boldsymbol{\xi}^{(1)};\boldsymbol{\xi}^{(2)})
		\Delta(\boldsymbol{\eta}^{(1)};\boldsymbol{\eta}^{(2)})
	}.
\end{equation*}
Observe that this factor is the product of the following three Cauchy determinants up to a sign
\begin{equation*}
	\begin{split}
		\mathrm{B}_1&= \det\left[\frac{1}{\xi^{(1)}_{i_1} -\eta^{(1)}_{j_1}}\right]_{i_1,j_1=1}^{k_1} = (-1)^{k_1(k_1-1)/2}\frac{ \Delta(\boldsymbol{\xi}^{(1)}) 
			\Delta(\boldsymbol{\eta}^{(1)}) 
		}
		{ \Delta(\boldsymbol{\xi}^{(1)}; \boldsymbol{\eta}^{(1)})},\\
		\mathrm{B}_2&= \det\left[\frac{1}{\xi^{(2)}_{i_2} -\eta^{(2)}_{j_2}}\right]_{i_2,j_2=1}^{k_2} = (-1)^{k_2(k_2-1)/2}\frac{ \Delta(\boldsymbol{\xi}^{(2)}) 
			\Delta(\boldsymbol{\eta}^{(2)}) 
		}
		{ \Delta(\boldsymbol{\xi}^{(2)}; \boldsymbol{\eta}^{(2)})},\\
		\mathrm{B}_3&
		= \det \left[
		\begin{array}{ccc|ccc}
			&\vdots&&&\vdots&\\
			\cdots&\displaystyle\frac{1}{\xi_{i_1}^{(1)} -\eta_{j_1}^{(1)}}&\cdots&\cdots&\displaystyle\frac{1}{\xi_{i_1}^{(1)} -\xi_{j_2}^{(2)}}&\cdots\\
			&\vdots&&&\vdots&\\
			\hline
			&\vdots&&&\vdots&\\
			\cdots&\displaystyle\frac{1}{\eta_{i_2}^{(2)} -\eta_{j_1}^{(1)}}&\cdots&\cdots&\displaystyle\frac{1}{\eta_{i_2}^{(2)} -\xi_{j_2}^{(2)}}&\cdots\\
			&\vdots&&&\vdots&
		\end{array}
		\right]_{\substack{1\le i_1,j_1\le k_1\\ 1\le i_2,j_2\le k_2}}\\
		&=(-1)^{k_1(k_1-1)/2+k_2(k_2+1)/2}
		\frac{ \Delta(\boldsymbol{\xi}^{(1)}) 
			\Delta(\boldsymbol{\eta}^{(1)}) 
		}
		{ \Delta(\boldsymbol{\xi}^{(1)}; \boldsymbol{\eta}^{(1)})}
		\cdot \frac{ \Delta(\boldsymbol{\xi}^{(2)}) 
			\Delta(\boldsymbol{\eta}^{(2)}) 
		}
		{ \Delta(\boldsymbol{\xi}^{(2)}; \boldsymbol{\eta}^{(2)})}
		\cdot 
		\frac{\Delta(\boldsymbol{\xi}^{(1)};\boldsymbol{\eta}^{(2)})
			\Delta(\boldsymbol{\eta}^{(1)};\boldsymbol{\xi}^{(2)})
		}
		{\Delta(\boldsymbol{\xi}^{(1)};\boldsymbol{\xi}^{(2)})
			\Delta(\boldsymbol{\eta}^{(1)};\boldsymbol{\eta}^{(2)})
		}.
	\end{split}
\end{equation*}
By applying the Hadamard's inequality, we have
\begin{equation*}
	\left|\mathrm{B}_1 \right| \le \prod_{i_1=1}^{k_1}\sqrt{\sum_{j_1=1}^{k_1}\left|\xi^{(1)}_{i_1} -\eta^{(1)}_{j_1}\right|^{-2}}
	\le k_1^{k_1/2}\prod_{i_1=1}^{k_1}\frac{1}{\dg(\xi_{i_1}^{(1)})},
\end{equation*}
where $\dg(\xi)$ denotes the shortest distance from the point $\xi$ to the contours $\Gamma_{\LL,\out},\Gamma_\LL,\Gamma_{\LL,\inn},\Gamma_{\RR,\out},\Gamma_\RR,\Gamma_{\RR,\inn}$ except for the one contour which $\xi$ belongs to. For example, if  $\xi_{i_1}^{(1)}\in\Gamma_{\LL,\out}$, then $\dg(\xi_{i_1}^{(1)})$ is the distance from $\xi_{i_1}^{(1)}$ to $\Gamma_\LL\cup\Gamma_{\RR,\out}$, where we ignored the contours $\Gamma_{\LL,\out},\Gamma_{\LL,\inn},\Gamma_\RR,$ and $\Gamma_{\RR,\inn}$ since $\Gamma_{\LL,\out}$ is the contour $\xi_{i_1}^{(1)}$ belongs to, and the other three contours are farther to the point $\xi_{i_1}^{(1)}$ compared with $\Gamma_\LL$ and $\Gamma_{\RR,\out}$.

Similarly, we have
\begin{equation*}
	\mathrm{B}_2 \le k_2^{k_2/2} \prod_{i_2=1}^{k_2}\frac{1}{\dg(\eta_{i_2}^{(2)})},
\end{equation*}
and
\begin{equation*}
	\mathrm{B}_3 \le (k_1+k_2)^{(k_1+k_2)/2}\prod_{j_1=1}^{k_1}\frac{1}{\dg(\eta_{j_1}^{(1)})}\prod_{j_2=1}^{k_2}\frac{1}{\dg(\xi_{j_2}^{(2)})}.
\end{equation*}
We combine the above estimates and obtain
\begin{equation}
	\label{eq:aux_12}
	\begin{split}
		&\left|\mathrm{B}(\boldsymbol{\xi}^{(1)},\boldsymbol{\eta}^{(1)};\boldsymbol{\xi}^{(2)},\boldsymbol{\eta}^{(2)})\right|\\
		& \le k_1^{k_1/2}k_2^{k_2/2}(k_1+k_2)^{(k_1+k_2)/2}\prod_{i_1=1}^{k_1}\frac{1}{\dg(\xi_{i_1}^{(1)})}\prod_{i_2=1}^{k_2}\frac{1}{\dg(\eta_{i_2}^{(2)})}\prod_{j_1=1}^{k_1}\frac{1}{\dg(\eta_{j_1}^{(1)})}\prod_{j_2=1}^{k_2}\frac{1}{\dg(\xi_{j_2}^{(2)})}.
	\end{split}
\end{equation}

Now we consider the factor $\mathrm{H}(\boldsymbol{\xi}^{(1)},\boldsymbol{\eta}^{(1)};
\boldsymbol{\xi}^{(2)},\boldsymbol{\eta}^{(2)})=\frac{1}{12}\mathrm{S}_1^4+\frac14\mathrm{S}_2^2-\frac13\mathrm{S}_1\mathrm{S}_3$ which is defined in~\eqref{eq:def_rmH}. We use the trivial bounds 
\begin{equation*}
	\begin{split}
		|\mathrm{S}_\ell| &= \left|\sum_{i_1=1}^{k_1}
		\left( \left(\xi_{i_1}^{(1)}\right)^{\ell}
		-\left(\eta_{i_1}^{(1)}\right)^{\ell}
		\right)
		-\sum_{i_2=1}^{k_2}
		\left( \left(\xi_{i_2}^{(2)}\right)^{\ell}
		-\left(\eta_{i_2}^{(2)}\right)^{\ell}
		\right)\right| \\
		&\le \prod_{i_1=1}^{k_1} \left(1+ |\xi_{i_1}^{(1)}|^\ell\right)\left(1+ |\eta_{i_1}^{(1)}|^\ell\right)\prod_{i_2=1}^{k_2} \left(1+ |\xi_{i_2}^{(2)}|^\ell\right)\left(1+ |\eta_{i_2}^{(2)}|^\ell\right)\\
		&\le \prod_{i_1=1}^{k_1}\rg_1\left(|\xi_{i_1}^{(1)}|\right)\rg_1\left(|\eta_{i_1}^{(1)}|\right)\prod_{i_2=1}^{k_2}\rg_1\left(|\xi_{i_2}^{(2)}|\right)\rg_1\left(|\eta_{i_2}^{(2)}|\right),\quad \ell=1,2,3,
	\end{split}
\end{equation*}
where $\rg_1(y):=1+y+y^2+y^3$. Note that $\rg_1^2(y)\le \rg_1^4(y)$ for all $y\ge 0$. Thus
\begin{equation}
	\label{eq:aux_13}
	\begin{split}
		|\mathrm{H}(\boldsymbol{\xi}^{(1)},\boldsymbol{\eta}^{(1)};
		\boldsymbol{\xi}^{(2)},\boldsymbol{\eta}^{(2)})|
		&\le \frac1{12}|\mathrm{S}_1^4|+\frac{1}{4}|\mathrm{S}_2^2|+\frac{1}{3}|\mathrm{S}_1\mathrm{S}_3|\\
		&\le \prod_{i_1=1}^{k_1}\rg_1^4\left(|\xi_{i_1}^{(1)}|\right)\rg_1^4\left(|\eta_{i_1}^{(1)}|\right)\prod_{i_2=1}^{k_2}\rg_1^4\left(|\xi_{i_2}^{(2)}|\right)\rg_1^4\left(|\eta_{i_2}^{(2)}|\right).
	\end{split}
\end{equation}
Finally, we note that the locations of contours imply that $\mathrm{Re}(\xi_{i_\ell}^{(\ell)})<0$ for $\xi_{i_\ell}^{(\ell)}\in\Gamma_\LL\cup\Gamma_{\LL,\out}\cup\Gamma_{\LL,\inn}$, and $\mathrm{Re}(\eta_{i_\ell}^{(\ell)})>0$ for $\eta_{i_\ell}^{(\ell)}\in\Gamma_\RR\cup\Gamma_{\RR,\out}\cup\Gamma_{\RR,\inn}$. Thus we have a trivial bound
\begin{equation}
	\label{eq:aux_14}
	\begin{split}
		\left|\frac{1}{\prod_{\ell=1}^2\sum_{i_\ell=1}^{k_\ell}(\xi_{i_\ell}^{(\ell)} -\eta_{i_\ell}^{(\ell)})}\right|&\le \frac{1}{\mathrm{Re}(\eta_{ 1}^{(1)}-\xi_{ 1}^{(1)})}\cdot \frac{1}{\mathrm{Re}(\eta_{1}^{(2)}-\xi_{1}^{(2)})}
		\le \frac{1}{\mathrm{Re}(\eta_1^{(1)})}\cdot \frac{1}{\mathrm{Re}(\eta_1^{(2)})}\\
		&\le  \left(1+\frac{1}{\mathrm{Re}(\eta_1^{(1)})}\right)\left(1+\frac{1}{\mathrm{Re}(-\xi_1^{(1)})}\right)\left(1+\frac{1}{\mathrm{Re}(\eta_1^{(2)})}\right)\left(1+\frac{1}{\mathrm{Re}(-\xi_1^{(2)})}\right)\\
		&\le \prod_{i_1=1}^{k_1}\rg_2\left(\xi_{i_1}^{(1)}\right)\rg_2\left(\eta_{i_1}^{(1)}\right)\prod_{i_2=1}^{k_2}\rg_2\left(\xi_{i_2}^{(2)}\right)\rg_2\left(\eta_{i_2}^{(2)}\right),
	\end{split}
\end{equation}
where $\rg_2(w):=1+|\mathrm{Re}(w)|^{-1}$ for all $w\in\complexC\setminus\mathrm{i}\realR$.

Now we insert all the estimates~\eqref{eq:aux_12},~\eqref{eq:aux_13} and~\eqref{eq:aux_14} in the equation~\eqref{eq:hat_rT} and obtain
\begin{equation}
	\label{eq:aux_15}
	\begin{split}
		&\left|\mathrm{\hat T}_{k_1,k_2}(\mathrm{z};\mathrm{t}_1,\mathrm{t}_2,\mathrm{x};\gamma)\right|
		\le k_1^{k_1/2}k_2^{k_2/2}(k_1+k_2)^{(k_1+k_2)/2}\\
		&\quad\cdot 
		\prod_{i_1=1}^{k_1}     
		\left(  \frac{1}{|1-\mathrm{z}|} \int_{\Gamma_{\LL,\inn}} 	    
		\frac{|\mathrm{d} \xi_{i_1}^{(1)}|}{2\pi}   
		+\frac{\mathrm{|z|}}{|1-\mathrm{z}|} \int_{\Gamma_{\LL,\out}} \frac{|\mathrm{d} \xi_{i_1}^{(1)}|}{2\pi}   
		\right)
		\left(  \frac{1}{|1-\mathrm{z}|} \int_{\Gamma_{\RR,\inn}} 	    
		\frac{|\mathrm{d} \eta_{i_1}^{(1)}|}{2\pi}   
		+\frac{\mathrm{|z|}}{|1-\mathrm{z}|} \int_{\Gamma_{\RR,\out}} \frac{|\mathrm{d} \eta_{i_1}^{(1)}|}{2\pi}   
		\right)
		\\
		&\quad \cdot 
		\prod_{i_2=1}^{k_2}  
		\int_{\Gamma_\LL} \frac{|\mathrm{d} \xi_{i_2}^{(2)}|}{2\pi} 
		\int_{\Gamma_\RR} \frac{|\mathrm{d} \eta_{i_2}^{(2)}|}{2\pi}  
		\cdot    
		\left|1-\mathrm{z}\right|^{k_2}
		\left|1-\frac{1}{\mathrm{z}}\right|^{k_1}
		\cdot 
		\prod_{i_1=1}^{k_1}\rg\left(\xi_{i_1}^{(1)}\right)\rg\left(\eta_{i_1}^{(1)}\right)\prod_{i_2=1}^{k_2}\rg\left(\xi_{i_2}^{(2)}\right)\rg\left(\eta_{i_2}^{(2)}\right)\\
		&=k_1^{k_1/2}k_2^{k_2/2}(k_1+k_2)^{(k_1+k_2)/2}\left|1-\mathrm{z}\right|^{k_2}
		\left|1-\frac{1}{\mathrm{z}}\right|^{k_1} \CC_{1,1}^{k_1}\CC_{1,2}^{k_2}\\
		&\le k_1^{k_1/2}k_2^{k_2/2}(k_1+k_2)^{(k_1+k_2)/2} \left(\left|1-\frac{1}{\mathrm{z}}\right|\CC_{1,1}+\left|1-\mathrm{z}\right|\CC_{1,2}\right)^{k_1+k_2},
	\end{split}
\end{equation}
where 
\begin{equation*}
	\rg(\zeta)=\begin{dcases}
		|\mathrm{f}_1(\zeta;\mathrm{t}_1)|\rg^4_1(|\zeta|)\rg_2(\zeta)/\dg(\zeta), & \zeta\in\Gamma_{\LL,\out}\cup\Gamma_{\LL,\inn},\\
		|\mathrm{f}_1(\zeta;\mathrm{t}_1)^{-1}|\rg^4_1(|\zeta|)\rg_2(\zeta)/\dg(\zeta), & \zeta\in\Gamma_{\RR,\out}\cup\Gamma_{\RR,\inn},\\
		|\mathrm{f}_2(\zeta;\mathrm{t}_2)|\rg^4_1(|\zeta|)\rg_2(\zeta)/\dg(\zeta), & \zeta\in\Gamma_{\LL},\\
		|\mathrm{f}_2(\zeta;\mathrm{t}_2)^{-1}|\rg^4_1(|\zeta|)\rg_2(\zeta)/\dg(\zeta), & \zeta\in\Gamma_{\RR},\\
	\end{dcases}
\end{equation*}
and
\begin{equation*}
	\begin{split}
		\CC_{1,1}&=\left(  \frac{1}{|1-\mathrm{z}|} \int_{\Gamma_{\LL,\inn}} 	    
		\frac{\rg(\xi)|\mathrm{d} \xi|}{2\pi}   
		+\frac{\mathrm{|z|}}{|1-\mathrm{z}|} \int_{\Gamma_{\LL,\out}} \frac{\rg(\xi)|\mathrm{d} \xi|}{2\pi}   
		\right)\left(  \frac{1}{|1-\mathrm{z}|} \int_{\Gamma_{\RR,\inn}} 	    
		\frac{\rg(\eta)|\mathrm{d} \eta|}{2\pi}   
		+\frac{\mathrm{|z|}}{|1-\mathrm{z}|} \int_{\Gamma_{\RR,\out}} \frac{\rg(\eta)|\mathrm{d} \eta|}{2\pi}   
		\right),\\
		\CC_{1,2}&=\left(  \int_{\Gamma_{\LL}} 	    
		\frac{\rg(\xi)|\mathrm{d} \xi|}{2\pi}   	  
		\right)\left(\int_{\Gamma_{\RR}} 	    
		\frac{\rg(\eta)|\mathrm{d} \eta|}{2\pi}  
		\right).
	\end{split}
\end{equation*}
We used the fact that $\rg(\zeta)$ decays exponentially when $\zeta$ goes to infinity along the integration contours since all other factors are of polynomial order, $\dg(\zeta)$ is bounded below, and the dominating factor $|\mathrm{f}_\ell|$ (or $|\mathrm{f}_\ell^{-1}|$) decays super exponentially. By checking the parameters appearing in $\mathrm{f}_\ell$ (and hence in $\rg$), we find that both $\CC_{1,1}$ and $\CC_{1,2}$ satisfy the conditions described in Notation~\ref{notation:CC}. Thus~\eqref{eq:aux_15} implies Lemma~\ref{lm:hat_rt_uniform_bounds} with $\CC_1=\left|1-\frac{1}{\mathrm{z}}\right|\CC_{1,1}+\left|1-\mathrm{z}\right|\CC_{1,2}$.

\subsection{Proof of Lemma~\ref{lm:estimate_error}}
\label{sec:proof_lm_error}

The proof of Lemma~\ref{lm:estimate_error} is more tedious. We separate the argument into three parts. In the first part we illustrate the proof strategy and show that Lemma~\ref{lm:estimate_error} can be reduced to two other lemmas. In the remaining two parts we prove these lemmas respectively.

\subsubsection{Proof strategy}
\label{sec:asymptotics_strategy}

Although the quantities $\hat T_{k_1,k_2}$ and $\mathrm{\hat T}_{k_1,k_2}$ only depend on how the integration contours are nested, we choose these contours explicitly to simplify our argument. The idea is that we split each contour into two parts with one part making most of the contribution in integration and the other part contributing an exponentially small error only.

We first choose the six contours appearing in the terms $\mathrm{\hat T}_{k_1,k_2}$. As we introduced before, we assume $\Gamma_{\LL,\out},\Gamma_\LL$ and $\Gamma_{\LL,\inn}$, from right to left, are three  simple contours in the left half plane from $e^{-2\pi\ii/3}\infty$ to $e^{2\pi\ii/3}\infty$. Similarly, $\Gamma_{\RR,\out},\Gamma_\RR$ and $\Gamma_{\RR,\inn}$, from left to right, are three  simple contours in the right half plane from $e^{-\pi\ii/3}\infty$ to $e^{\pi\ii/3}\infty$. For simplification, we assume that all these contours are symmetric about the real axis.

Each of the $\Gamma_*$ contour above, $*\in\{\{\LL,\out\},\{\LL\},\{\LL,\inn\},\{\RR,\out\},\{\RR\},\{\RR,\inn\}\}$, can be split into two parts. One part is within the disk $\mathbb{D}(\log N)$, the disk of radius $\log N$ with center $0$, and the other part is outside of this disk. We denote these two parts $\Gamma^{(N)}_*$ and $\Gamma^{(\exe)}_*$. In other words, we have six contours within $\mathbb{D}(\log N)$:  $\Gamma_{\LL,\out}^{(N)},\Gamma_\LL^{(N)}$, $\Gamma_{\LL,\inn}^{(N)}$,  $\Gamma_{\RR,\out}^{(N)},\Gamma_\RR^{(N)}$, and $\Gamma_{\RR,\inn}^{(N)}$, and six contours outside of $\mathbb{D}(\log N)$:  $\Gamma_{\LL,\out}^{(\exe)},\Gamma_\LL^{(\exe)}$, $\Gamma_{\LL,\inn}^{(\exe)}$,  $\Gamma_{\RR,\out}^{(\exe)},\Gamma_\RR^{(\exe)}$, and $\Gamma_{\RR,\inn}^{(\exe)}$.

We now choose the six contours appearing in the terms $\hat T_{k_1,k_2}$. We let them all intersect a neighborhood  of the point
\begin{equation}
	\label{eq:def_w_c}
	w_c:=-\frac{1}{1+\sqrt{\alpha}},
	\end{equation}
where $\alpha$ is the constant in Proposition~\ref{prop:asymptotics_density}.
We pick, for each $*\in\{\{\LL,\out\},\{\LL\},\{\LL,\inn\},\{\RR,\out\},\{\RR\},\{\RR,\inn\}\}$, $\Sigma_*$ to be the union of two parts $\Sigma_*^{(N)}$ and $\Sigma_*^{(\exe)}$. The part $\Sigma_*^{(N)}$ lies in a neighborhood of $w_c$ and satisfies
\begin{equation}
	\label{eq:relation_contours}
	\Sigma_{*}^{(N)}= w_c + \alpha^{1/6}(1+\sqrt{\alpha})^{-4/3}N^{-1/3}\Gamma_*^{(N)},\quad *\in\{\{\LL,\out\},\{\LL\},\{\LL,\inn\},\{\RR,\out\},\{\RR\},\{\RR,\inn\}\}.
\end{equation}
See the solid contours within the dashed circle in Figure~\ref{fig:limit_theorem_proof_contours} for an illustration.

Recall $f_1(w;t_1)=(w+1)^{-m}w^{n}e^{t_1w}$ and $f_2(w;t_2)=(w+1)^{-M+m}w^{N-n}e^{t_2w}$ with the parameters satisfying~\eqref{eq:scaling}. A detailed calculation (see~\eqref{eq:aux_17} and~\eqref{eq:aux_18} for example) indicate that $f_i(w;t_i)$ behaves like a cubic-exponential function. More explicitly, $f_i(w;t_i)$ decays super-exponentially fast when $w$ moves away from $w_c$ along the contours $\Sigma_{*}^{(N)}$ on the left, and grows super-exponentially fast along the contours $\Sigma_{*}^{(N)}$ on the right. Moreover, if we denote $w_*^{ep}$ and $\overline{w_*^{ep}}$ the endpoints of $\Sigma_*^{(N)}$, using~\eqref{eq:aux_17} and~\eqref{eq:aux_18}, we have $|f_i(w_*^{ep},t_i)/f_i(w_c;t_i)|\le e^{-c(\log N)^3}$ when $w_*^{ep}$ is on the left contours, and $|f_i(w_*^{ep};t_i)/f_i(w_c,t_i)|\ge e^{c(\log N)^3}$ when $w_*^{ep}$ is on the right contours. Here $c$ is some positive constant uniformly for $\mathrm{x}$ in a compact interval and $\mathrm{t}_1,\mathrm{t}_2$ with a lower bound.

In the next step, we will define the contours $\Sigma_*^{(\exe)}$. Note that \begin{equation*}
	f_1(w;t_1)=e^{\gamma N h(w) + \bigO(N^{2/3})},\quad f_2(w;t_2)=e^{(1-\gamma )N h(w) +\bigO(N^{2/3})},
\end{equation*}
where 
\begin{equation}
	h(w)=-\alpha \log(w+1) +\log w +(\sqrt{\alpha}+1)^2 w.
\end{equation}
It is standard to analyze $\mathrm{Re}h(w)$ for $w\in\complexC$ and extend the contours $\Sigma_*^{(N)}$ to $\Sigma_*^{(\exe)}$ such that
\begin{equation}
	\label{eq:2022_1_3_01}
	\max_{u\in\Sigma_*^{(\exe)}}|f_i (u;t_i)| \le \min_{u\in\Sigma_*^{(N)}}|f_i (u;t_i)|, \qquad i=1,2, *\in\{\{\LL,\out\},\{\LL\},\{\LL,\inn\}\}
\end{equation}
and
\begin{equation}
	\label{eq:2022_1_3_02}
	\min_{v\in\Sigma_*^{(\exe)}}|f_i (v;t_i)| \ge \max_{v\in\Sigma_*^{(N)}}|f_i (v;t_i)|, \qquad i=1,2, *\in\{\{\RR,\out\},\{\RR\},\{\RR,\inn\}\}
\end{equation}
for sufficiently large $N$. See Figure~\ref{fig:limit_theorem_proof_contours} for an illustration and the figure caption for more explanation.
\begin{figure}[h]
	\centering
\includegraphics{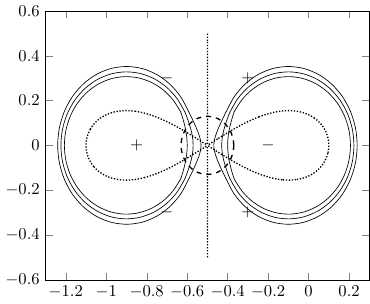}
\caption{Illustration of the contours when $\alpha=1$. The dotted lines represent the level curve $\mathrm{Re}h(w)=\mathrm{Re}h(w_c)$. It consists of two closed contours and one infinite contour all of which pass the critical point $w_c$. The complex plane thus is split into four parts, two of them marked with $-$ signs have lower levels of $\mathrm{Re}h(w)$, and the other two marked with $+$ signs have higher levels of $\mathrm{Re}h(w)$. The three solid contours on the left, from inside to outside, are $\Sigma_{\LL,\inn}$, $\Sigma_{\LL}$, $\Sigma_{\LL,\out}$ respectively. The three solid contours on the right, from inside to outside, are $\Sigma_{\RR,\inn}$, $\Sigma_{\RR}$, and $\Sigma_{\RR,\out}$ respectively. Each contour $\Sigma_*$ is split into two parts. The part within the dashed circle is $\Sigma_*^{(N)}$, and the remaining part is $\Sigma_*^{(\exe)}$.}\label{fig:limit_theorem_proof_contours}
\end{figure}

Combining with the bounds of $f_i$ at the endpoints of $\Sigma_*^{(N)}$ discussed above, we have the following two estimates
\begin{equation}
	\label{eq:constant_c}
\max_{u\in\Sigma_*^{(\exe)}}|f_i (u;t_i)/f_i(w_c;t_i)| \le \min_{u\in\Sigma_*^{(N)}}|f_i (u;t_i)/f_i(w_c;t_i)|\le e^{-c(\ln N)^3}, *\in\{\{\LL,\out\},\{\LL\},\{\LL,\inn\}\},
\end{equation}
\begin{equation}
	\label{eq:constant_c2}
	\min_{v\in\Sigma_*^{(\exe)}}|f_i (v;t_i)/f_i(w_c;t_i)| \ge \min_{v\in\Sigma_*^{(N)}}|f_i (v;t_i)/f_i(w_c;t_i)| \ge e^{c(\ln N)^3}, *\in\{\{\RR,\out\},\{\RR\},\{\RR,\inn\}\}.
\end{equation}

We remark that the contours we choose above are independent of the parameters $k_1$ and $k_2$, hence the constant $c$ above is also independent of $k_1$ and $k_2$.

\bigskip

With the contours we mentioned above, we can rewrite
\begin{equation*}
	\hat T_{k_1,k_2}(\mathrm{z};t_1,t_2;m,n,M,N) =\hat T_{k_1,k_2}^{(N)}(\mathrm{z};t_1,t_2;m,n,M,N) +\hat T_{k_1,k_2}^{(\exe)}(\mathrm{z};t_1,t_2;m,n,M,N),
\end{equation*}
where
\begin{equation}
	\label{eq:hat_TN}
	\begin{split}
		&\hat T_{k_1,k_2}^{(N)}(\mathrm{z};t_1,t_2;m,n,M,N)\\
		&=
		\prod_{i_1=1}^{k_1}     
		\left(  \frac{1}{1-\mathrm{z}} \int_{\Sigma_{\LL,\inn}^{(N)}} 	    
		\ddbar{u^{(1)}_{i_1}}{}   
		-\frac{\mathrm{z}}{1-\mathrm{z}} \int_{\Sigma_{\LL,\out}^{(N)}} \ddbar{u^{(1)}_{i_1}}{}
		\right)
		\left(  \frac{1}{1-\mathrm{z}}  \int_{\Sigma_{\RR,\inn}^{(N)}}  
		\ddbar{v^{(1)}_{i_1}}{}
		-\frac{\mathrm{z}}{1-\mathrm{z}}   \int_{\Sigma_{\RR,\out}^{(N)}}
		\ddbar{v^{(1)}_{i_1}}{}
		\right)
		\\
		&\quad \cdot 
		\prod_{i_2=1}^{k_2}  
		\int_{\Sigma_\LL^{(N)}} \ddbar{u^{(2)}_{i_2}}{} 
		\int_{\Sigma_\RR^{(N)}} \ddbar{v^{(2)}_{i_2}}{} 
		\cdot    
		\left(1-\mathrm{z}\right)^{k_2}
		\left(1-\frac{1}{\mathrm{z}}\right)^{k_1}
		\cdot 
		\frac{f_1(U^{(1)};t_1)f_2(U^{(2)};t_2)}{f_1(V^{(1)};t_1)f_2(V^{(2)};t_2)}
		\cdot \frac{1}{\prod_{\ell=1}^2\sum_{i_\ell=1}^{k_\ell}(u_{i_\ell}^{(\ell)} -v_{i_\ell}^{(\ell)})}
		\\
		&\quad \cdot H(U^{(1)},U^{(2)};V^{(1)},V^{(2)})\cdot 
		\prod_{\ell=1}^2
		\frac{\left( \Delta(U^{(\ell)}) \right)^2
			\left( \Delta(V^{(\ell)}) \right)^2
		}
		{\left( \Delta(U^{(\ell)}; V^{(\ell)}) \right)^2}
		\cdot 
		\frac{\Delta(U^{(1)};V^{(2)})\Delta(V^{(1)};U^{(2)})}
		{\Delta(U^{(1)};U^{(2)})\Delta(V^{(1)};V^{(2)})}.
	\end{split}
\end{equation}
Note that $\hat T_{k_1,k_2}^{(N)}$ has the same formula as $\hat T_{k_1,k_2}$ in~\eqref{eq:hat_T} except that we replace all the $\Sigma_*$ contours to $\Sigma_*^{(N)}$. Recall that we have $\Sigma_*=\Sigma_*^{(N)}\cup\Sigma_*^{(\exe)}$. Hence
\begin{equation}
	\label{eq:hat_Texe}
	\begin{split}
		&\hat T_{k_1,k_2}^{(\exe)}(\mathrm{z};t_1,t_2;m,n,M,N)\\
		&=
		\sum_{\Delta}
		\prod_{i_1=1}^{k_1}     
		\left(  \frac{1}{1-\mathrm{z}} \int_{\Sigma_{\LL,\inn}^{(\Delta)}} 	    
		\ddbar{u^{(1)}_{i_1}}{}   
		-\frac{\mathrm{z}}{1-\mathrm{z}} \int_{\Sigma_{\LL,\out}^{(\Delta)}} \ddbar{u^{(1)}_{i_1}}{}
		\right)
		\left(  \frac{1}{1-\mathrm{z}}  \int_{\Sigma_{\RR,\inn}^{(\Delta)}}  
		\ddbar{v^{(1)}_{i_1}}{}
		-\frac{\mathrm{z}}{1-\mathrm{z}}   \int_{\Sigma_{\RR,\out}^{(\Delta)}}
		\ddbar{v^{(1)}_{i_1}}{}
		\right)
		\\
		&\quad \cdot 
		\prod_{i_2=1}^{k_2}  
		\int_{\Sigma_\LL^{(\Delta)}} \ddbar{u^{(2)}_{i_2}}{} 
		\int_{\Sigma_\RR^{(\Delta)}} \ddbar{v^{(2)}_{i_2}}{} 
		\cdots
	\end{split}
\end{equation}
where we did not write out the integrand which is the same as in~\eqref{eq:hat_TN}, and the summation is over all possible $\Delta$'s each of which belongs to $\{N,\exe\}$ and at least one $\Delta$ is $\exe$. We also point out that we omit the indices of $\Delta$ in $\Sigma_*^{(\Delta)}$: It indeed depends on the choice of $*$ and $i_1$ or $i_2$. Since we have $4k_1+2k_2$ integration contours, we have $2^{4k_1+2k_2}-1$ possible choices of $\Delta$ in the above summation.

Similarly we can write
\begin{equation*}
	\mathrm{\hat T}_{k_1,k_2}(\mathrm{z};\mathrm{t}_1,\mathrm{t}_2,\mathrm{x};\gamma)
	=\mathrm{\hat T}_{k_1,k_2}^{(N)}(\mathrm{z};\mathrm{t}_1,\mathrm{t}_2,\mathrm{x};\gamma)
	+
	\mathrm{\hat T}_{k_1,k_2}^{(\exe)}(\mathrm{z};\mathrm{t}_1,\mathrm{t}_2,\mathrm{x};\gamma),
\end{equation*}
where $\mathrm{\hat T}_{k_1,k_2}^{(N)}(\mathrm{z};\mathrm{t}_1,\mathrm{t}_2,\mathrm{x};\gamma)$ has the same formula as~\eqref{eq:hat_rT} with all the integration contours $\Gamma_*$ replaced by $\Gamma_*^{(N)}$, and $\mathrm{\hat T}_{k_1,k_2}^{(\exe)}(\mathrm{z};\mathrm{t}_1,\mathrm{t}_2,\mathrm{x};\gamma)$ is a summation of $2^{4k_1+2k_2}-1$ terms each of which has the same formula as ~\eqref{eq:hat_rT} except that the integration contours are all replaced by $\Gamma_*^{(N)}$ or $\Gamma^{(\exe)}_*$ and at least one of the contours is replaced by $\Gamma^{(\exe)}_*$.

We will show the following two lemmas.

\begin{lm}
	\label{lm:hat_T_main}
With the same assumptions as in Proposition~\ref{prop:asymptotics_density}, there exists a constant $\CC_{2,1}$ as described in Notation~\ref{notation:CC}, such that
\begin{equation*}
	\begin{split}
	&\left|\alpha^{-1/3}(1+\sqrt{\alpha})^{2/3}N^{2/3}\hat T_{k_1,k_2}^{(N)}(\mathrm{z};t_1,t_2;m,n,M,N)
	-\mathrm{\hat T}_{k_1,k_2}^{(N)}(\mathrm{z};\mathrm{t}_1,\mathrm{t}_2,\mathrm{x};\gamma)\right| \\
	&\le k_1^{k_1/2}k_2^{k_2/2}(k_1+k_2)^{(k_1+k_2)/2}\CC_{2,1}^{k_1+k_2}N^{-1/3}(\log N)^{5}
	\end{split}
\end{equation*}
for all $k_1,k_2\ge 1$ as $N$ becomes sufficiently large.
\end{lm}
\begin{lm}
	\label{lm:hat_T_err}
	With the same assumptions as in Proposition~\ref{prop:asymptotics_density}, there exist two constants $\CC_{2,3}$ and $\CC_{2,4}$ as described in Notation~\ref{notation:CC}, such that
	\begin{equation*}
		N^{2/3} \left|\hat T_{k_1,k_2}^{(\exe)}(\mathrm{z};t_1,t_2;m,n,M,N)\right| \le k_1^{k_1/2}k_2^{k_2/2}(k_1+k_2)^{(k_1+k_2)/2}\CC_{2,3}^{k_1+k_2}\cdot e^{-c\cdot (\ln N)^3/2},
	\end{equation*}
and
\begin{equation*}
	\left|\mathrm{\hat T}_{k_1,k_2}^{(\exe)}(\mathrm{z};\mathrm{t}_1,\mathrm{t}_2,\mathrm{x};\gamma)\right| \le k_1^{k_1/2}k_2^{k_2/2}(k_1+k_2)^{(k_1+k_2)/2}\CC_{2,4}^{k_1+k_2} \cdot  e^{-c\cdot (\ln N)^3/2}
\end{equation*}
for all $k_1,k_2\ge 1$ as $N$ becomes sufficiently large. Here the constant $c$ is the same as in~\eqref{eq:constant_c} and~\eqref{eq:constant_c2}.
\end{lm}

It is obvious that Lemmas~\ref{lm:estimate_error} follows immediately by the above lemmas. In the next two subsubsections we will prove Lemmas~\ref{lm:hat_T_main} and~\ref{lm:hat_T_err} respectively.

\subsubsection{Proof of Lemma~\ref{lm:hat_T_main}}

We recall the formula~\eqref{eq:hat_TN} for $\hat T_{k_1,k_2}^{(N)}$. We change the integration variables
\begin{equation}
	\label{eq:change_variables}
	\begin{split}
	u_{i_1}^{(1)} &= w_c + \alpha^{1/6} (1+\sqrt{\alpha})^{-4/3}N^{-1/3} \xi_{i_1}^{(1)},\\
	v_{i_1}^{(1)} &= w_c + \alpha^{1/6} (1+\sqrt{\alpha})^{-4/3}N^{-1/3} \eta_{i_1}^{(1)},\\
	u_{i_2}^{(2)} &= w_c + \alpha^{1/6} (1+\sqrt{\alpha})^{-4/3}N^{-1/3} \xi_{i_2}^{(2)},\\
	v_{i_2}^{(2)} &= w_c + \alpha^{1/6} (1+\sqrt{\alpha})^{-4/3}N^{-1/3} \eta_{i_2}^{(2)},
	\end{split}
\end{equation}
where $w_c=-(1+\sqrt{\alpha})^{-1}$ is defined in~\eqref{eq:def_w_c},  $\xi_{i_1}^{(1)} \in \Gamma_{\LL,\inn}^{(N)} \cup \Gamma_{\LL,\out}^{(N)}$, $\xi_{i_2}^{(2)} \in \Gamma_\LL^{(N)}$, $\eta_{i_1}^{(1)} \in \Gamma_{\RR,\inn}^{(N)} \cup \Gamma_{\RR,\out}^{(N)}$, and $\eta_{i_2}^{(2)} \in \Gamma_\RR^{(N)}$. Note the relation between  $\Sigma_*^{(N)}$ contours and $\Gamma_*^{(N)}$ contours in~\eqref{eq:relation_contours}. Thus we have
\begin{equation}
	\label{eq:aux_16}
	\begin{split}
		&\alpha^{-1/3}(1+\sqrt{\alpha})^{2/3}N^{2/3}\hat T_{k_1,k_2}^{(N)}(\mathrm{z};t_1,t_2;m,n,M,N)\\
		&=
		\prod_{i_1=1}^{k_1}     
		\left(  \frac{1}{1-\mathrm{z}} \int_{\Gamma^{(N)}_{\LL,\inn}} 	    
		\ddbar{\xi^{(1)}_{i_1}}{}   
		-\frac{\mathrm{z}}{1-\mathrm{z}} \int_{\Gamma^{(N)}_{\LL,\out}} \ddbar{\xi^{(1)}_{i_1}}{}
		\right)
		\left(  \frac{1}{1-\mathrm{z}}  \int_{\Gamma^{(N)}_{\RR,\inn}}  
		\ddbar{\eta^{(1)}_{i_1}}{}
		-\frac{\mathrm{z}}{1-\mathrm{z}}   \int_{\Gamma^{(N)}_{\RR,\out}}
		\ddbar{\eta^{(1)}_{i_1}}{}
		\right)
		\\
		&\quad \cdot 
		\prod_{i_2=1}^{k_2}  
		\int_{\Gamma^{(N)}_\LL} \ddbar{\xi^{(2)}_{i_2}}{} 
		\int_{\Gamma^{(N)}_\RR} \ddbar{\eta^{(2)}_{i_2}}{} 
		\cdot    
		\left(1-\mathrm{z}\right)^{k_2}
		\left(1-\frac{1}{\mathrm{z}}\right)^{k_1}
		\cdot 
		\frac
		{\tilde f_1(\boldsymbol{\xi}^{(1)};t_1)
			\tilde f_2(\boldsymbol{\xi}^{(2)};t_2)}
		{\tilde f_1(\boldsymbol{\eta}^{(1)};t_1)
			\tilde f_2(\boldsymbol{\eta}^{(2)};t_2)}
		\cdot
		\frac{1}{\prod_{\ell=1}^2\sum_{i_\ell=1}^{k_\ell}(\xi_{i_\ell}^{(\ell)} -\eta_{i_\ell}^{(\ell)})}
		\\
		&\quad \cdot
		\tilde {H}(\boldsymbol{\xi}^{(1)},\boldsymbol{\xi}^{(2)};
		\boldsymbol{\eta}^{(1)},\boldsymbol{\eta}^{(2)}) \cdot 
		\prod_{\ell=1}^2
		\frac{\left( \Delta(\boldsymbol{\xi}^{(\ell)}) \right)^2
			\left( \Delta(\boldsymbol{\eta}^{(\ell)}) \right)^2
		}
		{\left( \Delta(\boldsymbol{\xi}^{(\ell)}; \boldsymbol{\eta}^{(\ell)}) \right)^2}
		\cdot 
		\frac{\Delta(\boldsymbol{\xi}^{(1)};\boldsymbol{\eta}^{(2)})
			\Delta(\boldsymbol{\eta}^{(1)};\boldsymbol{\xi}^{(2)})
		}
		{\Delta(\boldsymbol{\xi}^{(1)};\boldsymbol{\xi}^{(2)})
			\Delta(\boldsymbol{\eta}^{(1)};\boldsymbol{\eta}^{(2)})
		},
	\end{split}
\end{equation}
where
\begin{equation}
	\label{eq:aux_28}
	\begin{split}
	&\tilde f_\ell (\xi_{i_\ell}^{(\ell)};t_i) = f_\ell (u_{i_\ell}^{(\ell)};t_i)/f_\ell(w_c;t_i), \qquad
	\tilde f_\ell (\eta_{i_\ell}^{(\ell)};t_i) = f_\ell (v_{i_\ell}^{(\ell)};t_i)/f_\ell(w_c;t_i),\\
	&\tilde {H}(\boldsymbol{\xi}^{(1)},\boldsymbol{\xi}^{(2)};
	\boldsymbol{\eta}^{(1)},\boldsymbol{\eta}^{(2)})=\alpha^{-2/3}(1+\sqrt{\alpha})^{10/3}N^{4/3}H(U^{(1)},U^{(2)};V^{(1)},V^{(2)})
	\end{split}
\end{equation}
with the $u_{i_\ell}^{(\ell)}, v_{i_\ell}^{(\ell)}$ being viewed as functions of $\xi_{i_\ell}^{(\ell)}$ and $\eta_{i_\ell}^{(\ell)}$ as in~\eqref{eq:change_variables}. Note that~\eqref{eq:aux_16} equals to $\mathrm{\hat T}_{k_1,k_2}^{(N)}(\mathrm{z};\mathrm{t}_1,\mathrm{t}_2,\mathrm{x};\gamma)$ if we replace $\tilde f_\ell$ by $\mathrm{f_\ell}$ and $\tilde H$ by $\mathrm{H}$, see~\eqref{eq:hat_rT} for the formula of $\mathrm{\hat T}_{(k_1,k_2)}$ and note that replacing the contours $\Gamma_*$ by $\Gamma_*^{(N)}$ in~\eqref{eq:hat_rT} gives the formula of $\mathrm{\hat T}_{(k_1,k_2)}^{(N)}$.

Recall that $f_1(w;t_1)=(w+1)^{-m} w^n e^{t_1w}$. Note the scaling in~\eqref{eq:scaling}. For all $|\zeta|\le \log N$, we have the following Taylor expansion
\begin{equation}
	\label{eq:aux_17a}
	\begin{split}&\log \left(f_1(w_c + \alpha^{1/6}(1+\sqrt{\alpha})^{-4/3}\zeta N^{-1/3};t_1)/f_1(w_c;t_1) \right)\\
		&= -m\log \left(1+\alpha^{-1/3}(1+\sqrt{\alpha})^{-1/3} \zeta N^{-1/3} \right) + n\log  \left(1-\alpha^{1/6}(1+\sqrt{\alpha})^{-1/3} \zeta N^{-1/3} \right) + t_1 \alpha^{1/6}(1+\sqrt{\alpha})^{-4/3} \zeta N^{-1/3}\\
		&=-\frac{1}{3}\gamma\zeta^3 -\frac{1}{2}({x}_2-{x}_1)\zeta^2 +\left(\mathrm{t}_1 -\frac{1}{4\gamma}({x}_2-{x}_1)^2\right)\zeta +\bigO(N^{-1/3}(\log N)^4),
		\end{split}
\end{equation}
and hence, using the fact $e^{\bigO(N^{-1/3}(\log N)^4)}=1+\bigO(N^{-1/3}(\log N)^4)$,
\begin{equation}
	\label{eq:aux_17}
	\tilde f_1(w_c + \alpha^{1/6}(1+\sqrt{\alpha})^{-4/3}\zeta N^{-1/3};t_1)=\mathrm{f}_1(\zeta;\mathrm{t}_1)\cdot \left(1 +\bigO(N^{-1/3}(\log N)^4)\right).
\end{equation}
Note here the error term $\bigO(N^{-1/3}(\log N)^4)$ is uniformly for all $|\zeta|\le \log N$.
Similarly, for all $|\zeta|\le \log N$,
\begin{equation}
	\label{eq:aux_18}
	\tilde f_2(w_c + \alpha^{1/6}(1+\sqrt{\alpha})^{-4/3}\zeta N^{-1/3};t_2)=\mathrm{f}_2(\zeta;\mathrm{t}_2)\cdot \left(1 +\bigO(N^{-1/3}(\log N)^4)\right).
\end{equation}
Inserting the above estimates, we have
\begin{equation}
	\label{eq:aux_19}
	\frac
	{\tilde f_1(\boldsymbol{\xi}^{(1)};t_1)
		\tilde f_2(\boldsymbol{\xi}^{(2)};t_2)}
	{\tilde f_1(\boldsymbol{\eta}^{(1)};t_1)
		\tilde f_2(\boldsymbol{\eta}^{(2)};t_2)}=
	\frac
	{\mathrm f_1(\boldsymbol{\xi}^{(1)};\mathrm{t}_1)
		\mathrm f_2(\boldsymbol{\xi}^{(2)};\mathrm{t}_2)}
	{\mathrm f_1(\boldsymbol{\eta}^{(1)};\mathrm{t}_1)
		\mathrm f_2(\boldsymbol{\eta}^{(2)};\mathrm{t}_2)}
	\left(1+c_1^{k_1+k_2}\bigO(N^{-1/3}(\log N)^4)\right),
\end{equation}
where $c_1=4$ and we used the inequality
\begin{equation}
	\label{eq:aux_31}
	\left|\prod_{i=1}^n(1+x_i)-1\right|\le (1+x)^n-1 \le 2^nx
\end{equation}
for all $x_1,\cdots,x_n\in\complexC$ and $x>0$ satisfying $|x_i|\le x<1$.

Now we consider the term $\tilde H$. Recall the formulas of $H$ in~\eqref{eq:def_H} and $\mathrm{S}_\ell$ in~\eqref{eq:def_S}. We have
\begin{equation}
	\label{eq:aux_20}\begin{split}
		&\tilde {H}(\boldsymbol{\xi}^{(1)},\boldsymbol{\xi}^{(2)};
		\boldsymbol{\eta}^{(1)},\boldsymbol{\eta}^{(2)})=\alpha^{-2/3}(1+\sqrt{\alpha})^{10/3}N^{4/3}H(U^{(1)},U^{(2)};V^{(1)},V^{(2)})\\
	 &=\frac{1}{2}\epsilon^{-2}\left(\mathrm{S}_1^2-\mathrm{S}_2\right) N^{2/3}+\epsilon^{-3}N \mathrm{S}_1 +\left(\frac{1}{2}\epsilon^{-2}(\mathrm{S}_1^2+\mathrm{S}_2) N^{2/3}-\epsilon^{-3}N \mathrm{S}_1\right)\cdot \prod_{i_1=1}^{k_1}\frac{v_{i_1}^{(1)}}{u_{i_1}^{(1)}}\prod_{i_2=1}^{k_2}\frac{u_{i_2}^{(2)}}{v_{i_2}^{(2)}},
	\end{split}
\end{equation}
where  $\epsilon :=\alpha^{1/6}(1+\sqrt{\alpha})^{-1/3}$. Note the following estimate
\begin{equation*}
	\begin{split}
	\frac{w_c+(1+\sqrt{\alpha})^{-1}\epsilon \zeta N^{-1/3}}{w_c} &= \exp\left(-\epsilon N^{-1/3}\zeta -\frac{1}{2}\epsilon^2N^{-2/3}\zeta^2 -\frac{1}{3}\epsilon^3N^{-1}\zeta^3 +\bigO(N^{-4/3}(\log N)^4)\right)\\
	&=\exp\left(-\epsilon N^{-1/3}\zeta -\frac{1}{2}\epsilon^2N^{-2/3}\zeta^2 -\frac{1}{3}\epsilon^3N^{-1}\zeta^3 \right)\left(1+\bigO(N^{-4/3}(\log N)^4)\right)
	\end{split}
\end{equation*}
for all $|\zeta|\le \log N$, where $\bigO(N^{-4/3}(\log N)^4)$ is uniformly on $\zeta$. Using the inequality~\eqref{eq:aux_31}, we obtain
\begin{equation}
	\label{eq:aux_32}
	\prod_{i_1=1}^{k_1}\frac{v_{i_1}^{(1)}}{u_{i_1}^{(1)}}\prod_{i_2=1}^{k_2}\frac{u_{i_2}^{(2)}}{v_{i_2}^{(2)}}=\exp\left(\epsilon N^{-1/3}\mathrm{S}_1+\frac12\epsilon^2 N^{-2/3} \mathrm{S}_2 +\frac{1}{3}\epsilon^3N^{-1}\mathrm{S}_3\right)\left(1+c_1^{k_1+k_2} \bigO(N^{-4/3}(\log N)^4)\right).
\end{equation}
Note the trivial bound  $|\mathrm{S}_\ell| \le  (k_1+k_2)(\log N)^\ell$. We have
\begin{equation*}
	\begin{split}
	\left|\exp\left(\epsilon N^{-1/3}\mathrm{S}_1\right)-\sum_{n\le 3}\frac{1}{n!}(\epsilon N^{-1/3}\mathrm{S}_1)^n\right|
	&\le \sum_{n\ge 4}\frac{1}{n!}(\epsilon(k_1+k_2) N^{-1/3}\log N)^n\\
	&\le (N^{-1/3}\log N)^4\sum_{n\ge 4}\frac{1}{n!}(\epsilon(k_1+k_2))^n\\
	&\le c_2^{k_1+k_2}(N^{-1/3}\log N)^4,
	\end{split}
\end{equation*}
where $c_2=e^\epsilon$. Thus
\begin{equation*}
	\exp\left(\epsilon N^{-1/3}\mathrm{S}_1\right)=1+\epsilon N^{-1/3}\mathrm{S}_1 +\frac{1}{2} \epsilon^2 N^{-2/3} \mathrm{S}_1^2 +\frac16 \epsilon^3 N^{-1}\mathrm{S}_1^3 +c_2^{k_1+k_2}\bigO(N^{-4/3}(\log N)^4).
\end{equation*}
Similarly we have
\begin{equation*}
	\begin{split}
		\exp\left(\frac12\epsilon^2 N^{-2/3} \mathrm{S}_2 \right)&=1+\frac12\epsilon^2 N^{-2/3}\mathrm{S}_2 +c_3^{k_1+k_2}\bigO(N^{-4/3}(\log N)^4),\\
		\exp\left(\frac{1}{3}\epsilon^3N^{-1}\mathrm{S}_3\right)&=1+\frac13\epsilon^3 N^{-1}\mathrm{S}_3+c_4^{k_1+k_2}\bigO\left(N^{-2}(\log N)^6\right)
	\end{split}
\end{equation*}
for some positive constants $c_3$ and $c_4$. 
Inserting the above equations to~\eqref{eq:aux_32}, and then combining~\eqref{eq:aux_32} and~\eqref{eq:aux_20}, we obtain, after a careful calculation,
\begin{equation}
	\label{eq:aux_21}
	\begin{split}
	\tilde {H}(\boldsymbol{\xi}^{(1)},\boldsymbol{\xi}^{(2)};
	\boldsymbol{\eta}^{(1)},\boldsymbol{\eta}^{(2)})&= \frac{1}{12}\mathrm{S}_1^4 +\frac{1}{4}\mathrm{S}_2^2-\frac{1}{3}\mathrm{S}_1\mathrm{S}_3 +c_5^{k_1+k_2}\bigO(N^{-1/3}(\log N)^5)\\
	&\  =\mathrm{H}(\boldsymbol{\xi}^{(1)},\boldsymbol{\xi}^{(2)};
	\boldsymbol{\eta}^{(1)},\boldsymbol{\eta}^{(2)})  +c_5^{k_1+k_2}\bigO(N^{-1/3}(\log N)^5)
	\end{split}
\end{equation}
for some positive constant $c_5$.

Now we insert~\eqref{eq:aux_19} and~\eqref{eq:aux_21} into~\eqref{eq:aux_16}, and obtain
\begin{equation}
	\label{eq:aux_22}
	\begin{split}
		&\alpha^{-1/3}(1+\sqrt{\alpha})^{2/3}N^{2/3}\hat T_{k_1,k_2}^{(N)}(\mathrm{z};t_1,t_2;m,n,M,N) - \mathrm{\hat T}_{k_1,k_2}^{(N)}(\mathrm{z};\mathrm{t}_1,\mathrm{t}_2,\mathrm{x};\gamma)\\
		&=c_1^{k_1+k_2}\bigO(N^{-1/3}(\log N)^4)E_1+c_5^{k_1+k_2}\bigO(N^{-1/3}(\log N)^5) E_2 + (c_1c_5)^{k_1+k_2}\bigO(N^{-2/3}(\log N)^9) E_2,
	\end{split}
\end{equation}
where
\begin{equation}
	\begin{split}
		E_j&=
		\prod_{i_1=1}^{k_1}     
		\left(  \frac{1}{1-\mathrm{z}} \int_{\Gamma^{(N)}_{\LL,\inn}} 	    
		\ddbar{\xi^{(1)}_{i_1}}{}   
		-\frac{\mathrm{z}}{1-\mathrm{z}} \int_{\Gamma^{(N)}_{\LL,\out}} \ddbar{\xi^{(1)}_{i_1}}{}
		\right)
		\left(  \frac{1}{1-\mathrm{z}}  \int_{\Gamma^{(N)}_{\RR,\inn}}  
		\ddbar{\eta^{(1)}_{i_1}}{}
		-\frac{\mathrm{z}}{1-\mathrm{z}}   \int_{\Gamma^{(N)}_{\RR,\out}}
		\ddbar{\eta^{(1)}_{i_1}}{}
		\right)
		\\
		&\quad \cdot 
		\prod_{i_2=1}^{k_2}  
		\int_{\Gamma^{(N)}_\LL} \ddbar{\xi^{(2)}_{i_2}}{} 
		\int_{\Gamma^{(N)}_\RR} \ddbar{\eta^{(2)}_{i_2}}{} 
		\cdot    
		\left(1-\mathrm{z}\right)^{k_2}
		\left(1-\frac{1}{\mathrm{z}}\right)^{k_1}
		\cdot \frac
		{\mathrm f_1(\boldsymbol{\xi}^{(1)};\mathrm{t}_1)
			\mathrm f_2(\boldsymbol{\xi}^{(2)};\mathrm{t}_2)}
		{\mathrm f_1(\boldsymbol{\eta}^{(1)};\mathrm{t}_1)
			\mathrm f_2(\boldsymbol{\eta}^{(2)};\mathrm{t}_2)}\cdot 
		\frac{1}{\prod_{\ell=1}^2\sum_{i_\ell=1}^{k_\ell}(\xi_{i_\ell}^{(\ell)} -\eta_{i_\ell}^{(\ell)})}
		\\
		&\quad \cdot
		K_j(\boldsymbol{\xi}^{(1)},\boldsymbol{\xi}^{(2)};
		\boldsymbol{\eta}^{(1)},\boldsymbol{\eta}^{(2)}) \cdot 
		\prod_{\ell=1}^2
		\frac{\left( \Delta(\boldsymbol{\xi}^{(\ell)}) \right)^2
			\left( \Delta(\boldsymbol{\eta}^{(\ell)}) \right)^2
		}
		{\left( \Delta(\boldsymbol{\xi}^{(\ell)}; \boldsymbol{\eta}^{(\ell)}) \right)^2}
		\cdot 
		\frac{\Delta(\boldsymbol{\xi}^{(1)};\boldsymbol{\eta}^{(2)})
			\Delta(\boldsymbol{\eta}^{(1)};\boldsymbol{\xi}^{(2)})
		}
		{\Delta(\boldsymbol{\xi}^{(1)};\boldsymbol{\xi}^{(2)})
			\Delta(\boldsymbol{\eta}^{(1)};\boldsymbol{\eta}^{(2)})
		},
	\end{split}
\end{equation}
with
\begin{equation}
	K_j(\boldsymbol{\xi}^{(1)},\boldsymbol{\xi}^{(2)};
	\boldsymbol{\eta}^{(1)},\boldsymbol{\eta}^{(2)})=
	\begin{dcases}
		\mathrm{H}(\boldsymbol{\xi}^{(1)},\boldsymbol{\xi}^{(2)};
		\boldsymbol{\eta}^{(1)},\boldsymbol{\eta}^{(2)}), & j=1,\\
		1,& j=2.
	\end{dcases}
\end{equation}
Note that these $E_j$ terms have similar structure with $\mathrm{\hat T}_{k_1,k_2}(\mathrm{z};\mathrm{s}_1,\mathrm{s}_2,\mathrm{x};\gamma)$, except that the integration contours $\Gamma_*^{(N)}$ are subsets of $\Gamma_*$ appearing in the definition of $\mathrm{\hat T}_{k_1,k_2}(\mathrm{z};\mathrm{s}_1,\mathrm{s}_2,\mathrm{x};\gamma)$. Recall~\eqref{eq:aux_15} in the proof of Lemma~\ref{lm:hat_rt_uniform_bounds}. It is obvious that we have the same upper bound if we use  contours $\Gamma_*^{(N)}$  instead of $\Gamma_*$. Thus we obtain
\begin{equation*}
	|E_1| \le k_1^{k_1/2}k_2^{k_2/2}(k_1+k_2)^{(k_1+k_2)/2}\CC_{1}^{k_1+k_2}.
\end{equation*}
Similarly we have, by removing the factor $\mathrm{g_1}^4$, which comes from the estimate of $\mathrm{H}$, in the inequality~\eqref{eq:aux_15}, 
\begin{equation*}
	|E_2| \le k_1^{k_1/2}k_2^{k_2/2}(k_1+k_2)^{(k_1+k_2)/2}(\CC'_{1})^{k_1+k_2},
\end{equation*}
where $\CC'_1\le \CC_1$ is a positive constant satisfying the conditions described in Notation~\ref{notation:CC}. Combining the estimates of $|E_j|$ with~\eqref{eq:aux_22}, we obtain Lemma~\ref{lm:hat_T_main}.

\subsubsection{Proof of Lemma~\ref{lm:hat_T_err}}

The proofs for the two estimates are similar, hence we only prove the estimate for $\hat T_{k_1,k_2}^{(\exe)}(\mathrm{z};t_1,t_2;m,n,M,N)$.

Recall~\eqref{eq:hat_Texe}. We have
\begin{equation}
	\label{eq:aux_23}
	\begin{split}
		&\left|\hat T_{k_1,k_2}^{(\exe)}(\mathrm{z};t_1,t_2;m,n,M,N)\right|\\
		&\le 
		\sum_{\Delta}
		\prod_{i_1=1}^{k_1}     
		\left(  \left|\frac{1}{1-\mathrm{z}}\right| \int_{\Sigma_{\LL,\inn}^{(\Delta)}} 	    
		\frac{|\mathrm{d}u^{(1)}_{i_1}|}{2\pi}   
		+\left|\frac{\mathrm{z}}{1-\mathrm{z}}\right| \int_{\Sigma_{\LL,\out}^{(\Delta)}} \frac{|\mathrm{d}u^{(1)}_{i_1}|}{2\pi}   
		\right)
		\left(  \left|\frac{1}{1-\mathrm{z}}\right| \int_{\Sigma_{\RR,\inn}^{(\Delta)}} 	    
		\frac{|\mathrm{d}v^{(1)}_{i_1}|}{2\pi}   
		+\left|\frac{\mathrm{z}}{1-\mathrm{z}}\right| \int_{\Sigma_{\RR,\out}^{(\Delta)}} \frac{|\mathrm{d}v^{(1)}_{i_1}|}{2\pi}   
		\right)
		\\
		&\quad \cdot 
		\prod_{i_2=1}^{k_2}  
		\int_{\Sigma_\LL^{(\Delta)}} \frac{|\mathrm{d}u^{(2)}_{i_2}|}{2\pi}   
		\int_{\Sigma_\RR^{(\Delta)}} \frac{|\mathrm{d}v^{(2)}_{i_2}|}{2\pi}  
		\cdot    
		\left|1-\mathrm{z}\right|^{k_2}
		\left|1-\frac{1}{\mathrm{z}}\right|^{k_1}
		\cdot 
		\left|\frac{f_1(U^{(1)};t_1)f_2(U^{(2)};t_2)}{f_1(V^{(1)};t_1)f_2(V^{(2)};t_2)}\right|
		\cdot \frac{1}{\prod_{\ell=1}^2\left|\sum_{i_\ell=1}^{k_\ell}(u_{i_\ell}^{(\ell)} -v_{i_\ell}^{(\ell)})\right|}
		\\
		&\quad \cdot \left| H(U^{(1)},U^{(2)};V^{(1)},V^{(2)})\right|\cdot 
		\prod_{\ell=1}^2
		\left|\frac{\left( \Delta(U^{(\ell)}) \right)^2
			\left( \Delta(V^{(\ell)}) \right)^2
		}
		{\left( \Delta(U^{(\ell)}; V^{(\ell)}) \right)^2}\right|
		\cdot 
		\left|\frac{\Delta(U^{(1)};V^{(2)})\Delta(V^{(1)};U^{(2)})}
		{\Delta(U^{(1)};U^{(2)})\Delta(V^{(1)};V^{(2)})}\right|.
	\end{split}
\end{equation}

Recall the the sum is over all possible $2^{4k_1+2k_2}-1$ combinations of the contours, except for the only one combination that all the contours are of the form $\Sigma_*^{(N)}$ (i.e., near the critical point $w_c$). Also recall that $\Sigma_*=\Sigma_*^{(N)}\cup\Sigma_*^{(\exe)}$. The right hand side of~\eqref{eq:aux_23} can be rewritten as
\begin{equation}
	\label{eq:aux_24}
	\begin{split}
	&\prod_{i_1=1}^{k_1}\left(  \int_{\Sigma_{\LL,\inn}}|\mathrm{d}u^{(1)}_{i_1}|
	+\int_{\Sigma_{\LL,\out}}|\mathrm{d}u^{(1)}_{i_1}|
	\right)
	\left( \int_{\Sigma_{\RR,\inn}}  |\mathrm{d}v^{(1)}_{i_1}|
	+ \int_{\Sigma_{\RR,\out}}|\mathrm{d}v^{(1)}_{i_1}|
	\right)\prod_{i_2=1}^{k_2}  
	\int_{\Sigma_\LL} |\mathrm{d}u^{(2)}_{i_2}| 
	\int_{\Sigma_\RR} |\mathrm{d}v^{(2)}_{i_2}|\\
	&-\prod_{i_1=1}^{k_1}\left(  \int_{\Sigma_{\LL,\inn}^{(N)}}|\mathrm{d}u^{(1)}_{i_1}|
	+\int_{\Sigma_{\LL,\out}^{(N)}}|\mathrm{d}u^{(1)}_{i_1}|
	\right)
	\left( \int_{\Sigma_{\RR,\inn}^{(N)}}   |\mathrm{d}v^{(1)}_{i_1}|
	+ \int_{\Sigma_{\RR,\out}^{(N)}}|\mathrm{d}v^{(1)}_{i_1}|   
	\right)\prod_{i_2=1}^{k_2}  
	\int_{\Sigma_\LL^{(N)}} |\mathrm{d}u^{(2)}_{i_2}| 
	\int_{\Sigma_\RR^{(N)}} |\mathrm{d}v^{(2)}_{i_2}|,
	\end{split}
\end{equation}
where we suppressed the factors and the integrand for simplifications since they do not affect our argument here. Note the following simple inequality
\begin{equation*}
	\prod_{i} (a_i+b_i) -\prod_{i}a_i \le \sum_\ell b_\ell\prod_{i\ne \ell}(a_i+b_i)
\end{equation*}
for all nonnegative numbers $a_i,b_i$. We apply this inequality for $a_i=\int_{\Sigma_*^{(N)}}$ and $b_i=\int_{\Sigma_*^{(\exe)}}$ in~\eqref{eq:aux_24}. We find that~\eqref{eq:aux_24} can be bounded by
\begin{equation}
	\sum_{j_1=1}^{k_1}\left(\delta_{j_1;1}+\delta_{j_1;2}+\delta_{j_1;3}+\delta_{j_1;4}\right)+
	\sum_{j_2=1}^{k_2}\left(\delta_{j_2;5} +\delta_{j_2;6}\right).
\end{equation}
The quantities $\delta_{j,i}$ in the above equation are given by
\begin{equation*}
	\begin{split}
	\delta_{j_1;1}&=\int_{\Sigma_{\LL,\inn}^{(\exe)}}|\mathrm{d}u^{(1)}_{j_1}|\prod_{i_1\ne j_1}\left(  \int_{\Sigma_{\LL,\inn}}|\mathrm{d}u^{(1)}_{i_1}|
	+\int_{\Sigma_{\LL,\out}}|\mathrm{d}u^{(1)}_{i_1}|
	\right)
	\prod_{i_1=1}^{k_1}\left( \int_{\Sigma_{\RR,\inn}}  |\mathrm{d}v^{(1)}_{i_1}|
	+ \int_{\Sigma_{\RR,\out}}|\mathrm{d}v^{(1)}_{i_1}|
	\right)\cdots\\
	\delta_{j_1;2}&=\int_{\Sigma_{\LL,\out}^{(\exe)}}|\mathrm{d}u^{(1)}_{j_1}|\prod_{i_1\ne j_1}\left(  \int_{\Sigma_{\LL,\inn}}|\mathrm{d}u^{(1)}_{i_1}|
	+\int_{\Sigma_{\LL,\out}}|\mathrm{d}u^{(1)}_{i_1}|
	\right)
	\prod_{i_1=1}^{k_1}\left( \int_{\Sigma_{\RR,\inn}}  |\mathrm{d}v^{(1)}_{i_1}|
	+ \int_{\Sigma_{\RR,\out}}|\mathrm{d}v^{(1)}_{i_1}|
	\right)\cdots\\
	\delta_{j_1;3}&=\int_{\Sigma_{\LL,\inn}^{(\exe)}}|\mathrm{d}v^{(1)}_{j_1}|\prod_{i_1=1}^{k_1}\left(  \int_{\Sigma_{\LL,\inn}}|\mathrm{d}u^{(1)}_{i_1}|
	+\int_{\Sigma_{\LL,\out}}|\mathrm{d}u^{(1)}_{i_1}|
	\right)
	\prod_{i_1\ne j_1}\left( \int_{\Sigma_{\RR,\inn}}  |\mathrm{d}v^{(1)}_{i_1}|
	+ \int_{\Sigma_{\RR,\out}}|\mathrm{d}v^{(1)}_{i_1}|
	\right)\cdots\\
	\delta_{j_1;4}&=\int_{\Sigma_{\LL,\inn}^{(\exe)}}|\mathrm{d}v^{(1)}_{j_1}|\prod_{i_1=1}^{k_1}\left(  \int_{\Sigma_{\LL,\inn}}|\mathrm{d}u^{(1)}_{i_1}|
	+\int_{\Sigma_{\LL,\out}}|\mathrm{d}u^{(1)}_{i_1}|
	\right)
	\prod_{i_1\ne j_1}\left( \int_{\Sigma_{\RR,\inn}}  |\mathrm{d}v^{(1)}_{i_1}|
	+ \int_{\Sigma_{\RR,\out}}|\mathrm{d}v^{(1)}_{i_1}|
	\right)\cdots
	\end{split}
\end{equation*}
where $\cdots$ stands for $\prod_{i_2=1}^{k_2}  
\int_{\Sigma_\LL^{(N)}} |\mathrm{d}u^{(2)}_{i_2}| 
\int_{\Sigma_\RR^{(N)}} |\mathrm{d}v^{(2)}_{i_2}|$, and
\begin{equation*}
	\begin{split}
		\delta_{j_2;5}&=\cdots \int_{\Sigma_{\LL}^{(\exe)}}|\mathrm{d}u^{(2)}_{j_2}|\prod_{i_2\ne j_2}|\mathrm{d}u^{(2)}_{i_2}|
		\prod_{i_2=1}^{k_2}  
		\int_{\Sigma_\RR^{(N)}} |\mathrm{d}v^{(2)}_{i_2}|,\\
		\delta_{j_2;6}&=\cdots \int_{\Sigma_{\RR}^{(\exe)}}|\mathrm{d}v^{(2)}_{j_2}|\prod_{i_2=1 }^{k_2}|\mathrm{d}u^{(2)}_{i_2}|
		\prod_{i_2\ne j_2}  
		\int_{\Sigma_\RR^{(N)}} |\mathrm{d}v^{(2)}_{i_2}|,
	\end{split}
\end{equation*}
where $\cdots$ stands for $\prod_{i_1=1}^{k_1}\left(  \int_{\Sigma_{\LL,\inn}}|\mathrm{d}u^{(1)}_{i_1}|
+\int_{\Sigma_{\LL,\out}}|\mathrm{d}u^{(1)}_{i_1}|
\right)
\left( \int_{\Sigma_{\RR,\inn}}  |\mathrm{d}v^{(1)}_{i_1}|
+ \int_{\Sigma_{\RR,\out}}|\mathrm{d}v^{(1)}_{i_1}|
\right)$. Here we suppressed the factors and integrands in $\delta_{j;\ell}$ for simplifications: They are the same as in~\eqref{eq:aux_23}.

We have the following estimates:
\begin{equation}
	\label{eq:aux_26}
	\delta_{j_1;\ell}\le k_1^{k_1/2}k_2^{k_2/2}(k_1+k_2)^{(k_1+k_2+4)/2} \CC_{2,3}^{k_1+k_2} Ne^{-c(\ln N)^3},\quad 1\le \ell\le 4, \ 1\le j_1\le k_1,
\end{equation}
and
\begin{equation}
	\label{eq:aux_27}
	\delta_{j_2;\ell}\le k_1^{k_1/2}k_2^{k_2/2}(k_1+k_2)^{(k_1+k_2+4)/2} \CC_{2,3}^{k_1+k_2} Ne^{-c(\ln N)^3},\quad 5\le \ell\le 6, \ 1\le j_2\le k_2,
\end{equation}
for all $k_1,k_2\ge 1$ and sufficiently large $N$, where $\CC_{2,3}$ is a constant satisfying the conditions described in Notation~\ref{notation:CC}, and $c>0$ is a constant appearing in~\eqref{eq:constant_c} and~\eqref{eq:constant_c2}. With these estimates, and noting that $(k_1+k_2)^3\le e^{2(k_1+k_2)}$ for all $k_1,k_2\ge 0$ and that $Ne^{-c(\ln N)^3}\ll e^{-c(\ln N)^3/2}$ for sufficiently large $N$, we obtain Lemma~\ref{lm:hat_T_err} immediately.

It remains to show~\eqref{eq:aux_26} and~\eqref{eq:aux_27}. We only prove one representative inequality due to their similarity. Below we show~\eqref{eq:aux_26}  for $j_1=\ell =1$.

We write down the full expression of $\delta_{1;1}$,

	\begin{equation}
		\label{eq:aux_25}
		\begin{split}
			&\delta_{1;1}=\left|\frac{1}{1-\mathrm{z}}\right| \int_{\Sigma_{\LL,\inn}^{(\exe)}} 	    
			\frac{|\mathrm{d}u^{(1)}_{1}|}{2\pi}  \cdot			
			\prod_{i_1=2}^{k_1}     
			\left(  \left|\frac{1}{1-\mathrm{z}}\right| \int_{\Sigma_{\LL,\inn}^{(\Delta)}} 	    
			\frac{|\mathrm{d}u^{(1)}_{i_1}|}{2\pi}   
			+\left|\frac{\mathrm{z}}{1-\mathrm{z}}\right| \int_{\Sigma_{\LL,\out}^{(\Delta)}} \frac{|\mathrm{d}u^{(1)}_{i_1}|}{2\pi}   
			\right)\\
			&\quad \cdot	\prod_{i_1=1}^{k_1}
			\left(  \left|\frac{1}{1-\mathrm{z}}\right| \int_{\Sigma_{\RR,\inn}^{(\Delta)}} 	    
			\frac{|\mathrm{d}v^{(1)}_{i_1}|}{2\pi}   
			+\left|\frac{\mathrm{z}}{1-\mathrm{z}}\right| \int_{\Sigma_{\RR,\out}^{(\Delta)}} \frac{|\mathrm{d}v^{(1)}_{i_1}|}{2\pi}   
			\right)\cdot 
			\prod_{i_2=1}^{k_2}  
			\int_{\Sigma_\LL} \frac{|\mathrm{d}u^{(2)}_{i_2}|}{2\pi}   
			\int_{\Sigma_\RR} \frac{|\mathrm{d}v^{(2)}_{i_2}|}{2\pi}  \\
			&\quad
			\cdot    
			\left|1-\mathrm{z}\right|^{k_2}
			\left|1-\frac{1}{\mathrm{z}}\right|^{k_1}
			\cdot 
			\left|\frac{f_1(U^{(1)};t_1)f_2(U^{(2)};t_2)}{f_1(V^{(1)};t_1)f_2(V^{(2)};t_2)}\right|
			\cdot \frac{1}{\prod_{\ell=1}^2\left|\sum_{i_\ell=1}^{k_\ell}(u_{i_\ell}^{(\ell)} -v_{i_\ell}^{(\ell)})\right|}
			\\
			&\quad \cdot \left| H(U^{(1)},U^{(2)};V^{(1)},V^{(2)})\right|\cdot 
			\prod_{\ell=1}^2
			\left|\frac{\left( \Delta(U^{(\ell)}) \right)^2
				\left( \Delta(V^{(\ell)}) \right)^2
			}
			{\left( \Delta(U^{(\ell)}; V^{(\ell)}) \right)^2}\right|
			\cdot 
			\left|\frac{\Delta(U^{(1)};V^{(2)})\Delta(V^{(1)};U^{(2)})}
			{\Delta(U^{(1)};U^{(2)})\Delta(V^{(1)};V^{(2)})}\right|.	
		\end{split}
	\end{equation}
 Note that, due to the assumptions of the contours,
\begin{equation*}
	\frac{1}{\prod_{\ell=1}^2\left|\sum_{i_\ell=1}^{k_\ell}(u_{i_\ell}^{(\ell)} -v_{i_\ell}^{(\ell)})\right|}\le \frac{1}{|\mathrm{Re}(u_1^{(1)}-w_c)|}\cdot\frac{1}{|\mathrm{Re}(u_1^{(2)}-w_c)|}.
\end{equation*}
We also use a looser bound for $H$, using the facts that all the contours are bounded and away from $0$,
\begin{equation*}
	\left| H(U^{(1)},U^{(2)};V^{(1)},V^{(2)})\right| \le C\cdot (k_1+k_2)^2
\end{equation*}
for all $k_1,k_2\ge 1$, where $C$ is positive constant independent of  $k_1,k_2$ and all the parameters. Now we use a similar argument as in Section~\ref{sec:proof_lm_uniform_bounds} and obtain
\begin{equation}
	\label{eq:aux_29}
	\begin{split}
	\delta_{1;1}&\le C\cdot k_1^{k_1/2}k_2^{k_2/2}(k_1+k_2)^{(k_1+k_2+4)/2}\left(\left|1-\frac{1}{\mathrm{z}}\right|\theta_{1}+\left|1-\mathrm{z}\right|\theta_{2}\right)^{k_1+k_2-2} \theta_{3} \cdot |1-\mathrm{z}|\cdot\left|1-\frac{1}{\mathrm{z}}\right| \\
	&\quad\cdot \left|\frac{1}{1-\mathrm{z}}\right| \int_{\Sigma_{\LL,\inn}^{(\exe)}} 	    
	\frac{|\mathrm{d}u^{(1)}_{1}|}{2\pi} \frac{|\tilde f_1(u_1^{(1)};t_1)|}{\dtt(u_1^{(1)})\cdot |\mathrm{Re}(u_1^{(1)}-w_c)|}\cdot \int_{\Sigma_{\LL}} 	    
	\frac{|\tilde f_2(u^{(2)}_1;t_2)||\mathrm{d} u_1^{(2)}|}{2\pi\cdot \dtt(u_1^{(2)})\cdot |\mathrm{Re}(u_1^{(2)}-w_c)|},
	\end{split}
\end{equation}
where $\tilde f_\ell(w;t_\ell)=f_\ell(w;t_\ell)/f(w_c;t_\ell)$ as introduced in~\eqref{eq:aux_28}, and $\theta_i$'s are given by
\begin{equation*}
	\begin{split}
		\theta_{1}&=\left(  \frac{1}{|1-\mathrm{z}|} \int_{\Sigma_{\LL,\inn}} 	    
		\frac{|\tilde f_1(u;t_1)||\mathrm{d} u|}{2\pi\cdot \dtt(u)}   
		+\frac{\mathrm{|z|}}{|1-\mathrm{z}|} \int_{\Sigma_{\LL,\out}} \frac{|\tilde f_1(u;t_1)||\mathrm{d} u|}{2\pi\cdot \dtt(u)} 
		\right)\\
		&\quad\cdot\left(  \frac{1}{|1-\mathrm{z}|} \int_{\Sigma_{\RR,\inn}} 	    
		\frac{|\mathrm{d}v|}{2\pi\cdot |\tilde f_1(v;t_1)|\cdot \dtt(v)} 
		+\frac{\mathrm{|z|}}{|1-\mathrm{z}|} \int_{\Gamma_{\RR,\out}} \frac{|\mathrm{d} v|}{2\pi\cdot |\tilde f_1(v;t_1)|\cdot \dtt(v)}  
		\right),\\
		\theta_{2}&=\left(  \int_{\Sigma_{\LL}} 	    
		\frac{|\tilde f_2(u;t_2)||\mathrm{d} u|}{2\pi\cdot \dtt(u)}  	  
		\right)\left(\int_{\Sigma_{\RR}} 	    
		\frac{|\mathrm{d}v|}{2\pi\cdot |\tilde f_2(v;t_2)|\cdot \dtt(v)} 
		\right),\\
		\theta_{3}&=\left(  \frac{1}{|1-\mathrm{z}|} \int_{\Sigma_{\RR,\inn}} 	    
		\frac{|\mathrm{d}v|}{2\pi\cdot |\tilde f_1(v;t_1)|\cdot \dtt(v)} 
		+\frac{\mathrm{|z|}}{|1-\mathrm{z}|} \int_{\Gamma_{\RR,\out}} \frac{|\mathrm{d} v|}{2\pi\cdot |\tilde f_1(v;t_1)|\cdot \dtt(v)}  
		\right)\\
		&\quad\cdot \left(\int_{\Sigma_{\RR}} 	    
		\frac{|\mathrm{d}v|}{2\pi\cdot |\tilde f_2(v;t_2)|\cdot \dtt(v)} 
		\right),
	\end{split}
\end{equation*}
and $\dtt(w)$, for $w\in\Sigma_{\LL,\inn}\cup\Sigma_{\LL}\cup\Sigma_{\LL,\out}\cup\Sigma_{\RR,\inn}\cup\Sigma_{\RR}\cup\Sigma_{\RR,\out}$, is the distance between $w$ and the contours $\Sigma_{\LL,\inn}\cup\Sigma_{\LL}\cup\Sigma_{\LL,\out}\cup\Sigma_{\RR,\inn}\cup\Sigma_{\RR}\cup\Sigma_{\RR,\out}$ except for the one $w$ belongs to. This $\dtt(w)$ has a similar definition as $\dtt(\zeta)$ in Section~\ref{sec:proof_lm_uniform_bounds} but with different contours. 

We claim that all of the integrals appearing in $\theta_i$ values are bounded by  some constant $\CC_{2;3}'$ satisfying the conditions described in Notation~\ref{notation:CC}. For example, consider the first integral in $\theta_1$,
\begin{equation*}
	\int_{\Sigma_{\LL,\inn}} 	    
	\frac{|\tilde f_1(u;t_1)||\mathrm{d} u|}{2\pi\cdot \dtt(u)}  =\int_{\Sigma_{\LL,\inn}^{(N)}} 	    
	\frac{|\tilde f_1(u;t_1)||\mathrm{d} u|}{2\pi\cdot \dtt(u)}  +\int_{\Sigma_{\LL,\inn}^{(\exe)}} 	    
	\frac{|\tilde f_1(u;t_1)||\mathrm{d} u|}{2\pi\cdot \dtt(u)},
\end{equation*}
where the first term is approximately, using~\eqref{eq:aux_17},
\begin{equation*}
	C' \cdot \int_{\Gamma_{\LL,\inn}^{(N)}}\frac{|\mathrm{f}_1(\xi;\mathrm{t}_1)||\mathrm{d}\xi| }{\dtt(\xi)}\le C' \cdot \int_{\Gamma_{\LL,\inn}}\frac{|\mathrm{f}_1(\xi;\mathrm{t}_1)||\mathrm{d}\xi| }{\dtt(\xi)}
\end{equation*}
for some constant $C'$, 
and the second term is bounded above by, using~\eqref{eq:constant_c},
\begin{equation*}
	C''\cdot N^{1/3} \cdot e^{-c(\ln N)^3}
\end{equation*}
for some constant $C''$, where the extra $N^{1/3}$ comes from a possible large factor $1/\dtt(u)$. These two estimates confirm the claim for the first factor. Similarly we have the claims for other factors. Thus we have
\begin{equation*}
	\theta_1 , \theta_2, \theta_3 \le \CC'_{2,3}.
\end{equation*}

Using the similar estimates, we can also obtain
\begin{equation*}
	\int_{\Sigma_{\LL,\inn}^{(\exe)}} 	    
	\frac{|\mathrm{d}u^{(1)}_{1}|}{2\pi} \frac{|\tilde f_1(u_1^{(1)};t_1)|}{\dtt(u_1^{(1)})\cdot |\mathrm{Re}(u_1^{(1)}-w_c)|}\le C'''N^{2/3}e^{-c(\ln N)^3}
\end{equation*}
and
\begin{equation*}
	\int_{\Sigma_{\LL}} 	    
	\frac{|\tilde f_2(u^{(2)}_1;t_2)||\mathrm{d} u_1^{(2)}|}{2\pi\cdot \dtt(u_1^{(2)})\cdot |\mathrm{Re}(u_1^{(2)}-w_c)|} \le C'''N^{1/3}\CC''_{2,3},
\end{equation*}
where the extra $N^{1/3}$ comes from a possible large factor $1/|\mathrm{Re}(w-w_c)|$. Combing all these estimates in~\eqref{eq:aux_29}, we obtain~\eqref{eq:aux_26} for $j_1=\ell=1$. Other estimates in~\eqref{eq:aux_26} and~\eqref{eq:aux_27} are similar.

\def\cydot{\leavevmode\raise.4ex\hbox{.}}

\end{document}